\documentclass[reqno]{amsart}

\usepackage{amsmath}
\usepackage{amsfonts}
\usepackage{amsthm}
\usepackage{amssymb}
\usepackage{graphicx}
\usepackage{graphics}
\usepackage{bm}
\usepackage{dsfont}
\usepackage{color}
\usepackage{enumerate}
\usepackage[font=footnotesize]{caption}
\usepackage[all]{xy}
\usepackage{xypic}
\allowdisplaybreaks

\numberwithin{equation}{section}
\newtheorem{thm}{Theorem}[section]
\newtheorem{cor}[thm]{Corollary}
\newtheorem{lem}[thm]{Lemma}
\newtheorem{prop}[thm]{Proposition}
\theoremstyle{definition}

\newtheorem{dfn}[thm]{Definition}
\newtheorem{rmk}[thm]{Remark}
\newcommand{\N}{\mathds{N}}

\newcommand{\Z}{\mathds{Z}}

\newcommand{\R}{\mathds{R}}

\newcommand{\C}{\mathds{C}}
\newcommand{\BB}{\mathcal B}
\newcommand{\BBB}{{\mathcal B'}}
\newcommand{\BBBB}{{\mathcal B''}}
\newcommand{\T}{\mathds{T}}

\newcommand{\diff}{\mathrm{d}}

\newcommand{\ccc}{\mathrm c}
\newcommand{\M}{\mathcal{M}}
\newcommand{\U}{\mathcal{U}}

\newcommand{\V}{\mathcal{V}}

\newcommand{\A}{\mathbb{A}}

\newcommand{\q}{{\mathbf{q}}}
\newcommand{\p}{\mathbf{p}}

\newcommand{\hh}{\mathbf{h}}
\newcommand{\kk}{\mathbf{k}}

\newcommand{\HH}{\mathbb H}

\newcommand{\ver}{{\mathrm{v}}}
\newcommand{\hor}{{\mathrm{h}}}

\begin{document}

\title[Morsex homology for the Hamiltonian action in cotangent bundles]{Morse homology for the Hamiltonian action in cotangent bundles}

\author[L. Asselle]{Luca Asselle}
\address{Ruhr-Universit\"at Bochum, Universit\"atsstra\ss e 150, 44801, Bochum, Germany}
\email{luca.asselle@math.uni-giessen.de}

\author[M. Starostka]{Maciej Starostka}
\address{Gda\'nks University of Technology, Gabriela Narutowicza 11/12, 80233 Gda\'nsk, Poland
\newline
\indent Institut f\"ur Mathematik, Naturwissenschaftliche Fakult\"at II, Martin-Luther-Universit\"at Halle-Wittenberg, \newline 
\indent 06099 Halle (Saale), Germany}
\email{maciejstarostka@pg.edu.pl}

\date{\today}

\begin{abstract}
In this paper we use the gradient flow equation introduced in \cite{Asselle:2020b} 
to construct a Morse complex for the Hamiltonian action $\A_H$ on a mixed regularity space of loops in the cotangent bundle $T^*M$ of a closed manifold $M$. 
Connections between pairs of critical points are realized as genuine intersections between unstable and stable manifolds, which (despite being infinite dimensional objects) 
turn out to have finite dimensional intersection with good compactness properties. This follows from the existence of an additional structure, namely 
a strongly integrable (0)-essential subbundle, which behaves nicely under the negative gradient flow of the Hamiltonian action and which is needed to make comparisons. Transversality is achieved by generically perturbing 
the negative gradient vector field $-\nabla \A_H$ of the Hamiltonian action within a class of pseudo-gradient vector fields preserving all good compactness properties of $-\nabla \A_H$. 
This follows from an abstract 
transversality result of independent interest for vector fields on a Hilbert manifold for which stable and unstable manifolds of rest points are infinite dimensional. The 
resulting Morse homology is independent of the choice of the Hamiltonian (and of all other choices but the choice of the (0)-essential subbundle, which however only changes the Morse-complex by a shift of the indices)
and is
isomorphic to the Floer homology of $T^*M$ as well as to the singular homology of the free loop space of $M$. 

\tableofcontents
\end{abstract}

\maketitle

%%%%%%%%%%%%
%%%%%%%%%%%%
%%%%%%%%%%%%
\vspace{-15mm}

\section{Introduction}

One of the central questions in the theory of Hamiltonian systems is the existence of (one-)periodic solutions $x:\T\to W$, $(W,\omega)$ symplectic manifold, to Hamilton's equation
\begin{equation}
\dot x(t) = X_H(x(t)),
\label{Hintro}
\end{equation}
where $X_H$ is the Hamiltonian vector field associated with a smooth time-dependent one-periodic Hamiltonian function $H:\T\times W\to \R$ via
$$\imath_{X_H} \omega = - \diff H.$$
Such periodic solutions can be characterized (at least at a formal level) as critical points of the \textit{Hamiltonian action functional}, which in the case of an exact symplectic form $\omega = \diff \lambda$ takes the form 
$$\A_H (x) = \int_0^1 x^* \lambda - \int_0^1 H(t,x(t))\, \diff t.$$
Unfortunately, one cannot infer the existence of critical points of $\A_H$ using classical variational methods such as Morse theory (and its infinite dimensional version due to Palais \cite{Palais:1963jp}), 
since every critical point of $\A_H$ has 
infinite Morse index and co-index.  In fact, for a long time people believed that no variational methods could be successfully applied to the functional $\A_H$. The first breakthrough in this direction is due to Rabinowitz: in  \cite{Rabinowitz:1978ts} he observed that the Hamiltonian action in $(\R^{2n},\omega_{\text{std}} = \diff p \wedge \diff q)$ is of the form
$$\A_H(x) = \frac {1}{2}\int_0^1 \langle -J\dot x,x\rangle \, \diff t - \int_0^1 H(t,x(t))\, \diff t,$$
where the first term 
defines a continuous quadratic form on $H^{1/2}(\T,\R^{2n})$, the space of Sobolev loops in $\R^{2n}$ of Sobolev regularity $1/2$, and the Hamiltonian term 
has compact $H^{1/2}$-gradient. This enabled him to prove what is nowadays universally known as the Weinstein conjecture, in the case of smooth compact connected 
hypersurfaces in $\R^{2n}$ bounding a compact and strictly convex region. By using finite dimensional approximations, the $H^{1/2}$-approach has been then successfully implemented 
by Conley and Zehnder in \cite{Conley:1983nr} to prove the Arnold conjecture on $(\T^{2n},\omega_{\text{std}})$. The case of the torus $\T^{2n}$ is somehow special (even though the $H^{1/2}$-approach can 
be used also for other symplectic manifolds such as $\C\mathbb P^n$, see \cite{For:85}): since $\T^{2n}$ is a quotient of $\R^{2n}$, the space of contractible $H^{1/2}$-loops on $\T^{2n}$ can be identified with $\T^{2n}$ times 
the Hilbert space of $H^{1/2}$-loops in $\R^{2n}$ having zero mean. This does not work for general symplectic manifolds, in particular for cotangent bundles $T^*M$ over a closed manifold $M$ 
equipped with the standard symplectic form 
$\omega_{\text{std}}=\diff \lambda$, and indeed the space of loops of class $H^{1/2}$ in $T^*M$ does not have a good structure of an infinite dimensional manifold, due to the fact that curves of class $H^{1/2}$ might have discontinuities. 

A brilliant idea on how to overcome such a difficulty came few years after the work of Conley and Zehnder with Floer (see \cite{Floer:1988oq} and its further extensions \cite{Floer:1989xr,HS:95,LT:98,FO:99}) and it is not limited to cotangent bundles but works for arbitrary symplectic manifolds which (if non-compact) are suitably convex at infinity: replacing the $H^{1/2}$-gradient of $\A_H$ with the $L^2$-gradient yields a gradient flow equation only formally, but if one interprets such a ``gradient flow equation'' as a PDE, this turns out to be a non-linear perturbed Cauchy-Riemann equation. This allowed Floer to use holomorphic curves techniques as developed by Gromov in \cite{Gromov:1985rv} to define a chain complex for $\A_H$, which is generated by contractible one-periodic solutions of~\eqref{Hintro}, by counting the number of solutions to the perturbed Cauchy-Riemann PDE which are asymptotic to pairs of periodic orbits whose Conley-Zehnder indices differ by one. The resulting homology is called \textit{Floer homology}, and its importance goes way beyond Floer's original motivation of proving the Arnold conjecture on the number of fixed points of non-degenerate Hamiltonian diffeomorphisms of closed symplectic manifolds: Lagrangian intersection Floer theory (see \cite{Fukaya}), symplectic homology (see \cite{FH94,CFH95,Viterbo:1999dp}), contact homology (see \cite{EGH00}) and Rabinowitz-Floer homology (see \cite{Cieliebak:2009va}) are just some of the famous derivations of Floer's seminal ideas (not to mention the ones in low dimensional topology).
Another ingenious approach for the Hamiltonian action in cotangent bundles, based on finite dimensional approximations and spectra, can be found in the work of Kragh \cite{Kragh:2012,Kragh:2013} and have led to a proof
of a homotopy version of the nearby Lagrangian conjecture.

Despite the great success of the theory of Floer, the question remained how far the original approach by Rabinowitz, Conley, and Zehnder can be generalized to manifolds different from the torus. In \cite{Asselle:2020b} we started 
addressing this question by showing that, for a Hamiltonian $H:\T\times T^*M \to \R$ with quadratic growth at infinity, the Hamiltonian action 
\begin{equation}
\A_H:\M^{1-s}\to \R, \quad s\in (1/2,1),
\label{AHintro}
\end{equation}
satisfies the Palais-Smale condition, $\M^{1-s}$ being the Hilbert bundle over the Hilbert manifold of loops $H^s(\T,M)$, $s\in (1/2,1)$, whose typical fibre is given by the space of $H^{1-s}$-sections of the pull-back bundle $\ccc^*(T^*M)$, where
$\ccc:\T\to M$ is any smooth loop. Roughly speaking, instead of considering loops with Sobolev regularity $1/2$, one considers loops in $T^*M$ whose projection to the base has regularity larger than $1/2$ (but strictly less than $1$), thus making them continuous in the base direction, and whose projection to the fiber has regularity smaller than $1/2$ (but strictly larger than $L^2$-regularity) in such a way that the mean regularity is $1/2$. The key observation is that the regularity loss in the fiber direction does not pose any major difficulty
because of the linear structure of the fibers.

\begin{rmk}
The case $s=1/2$ in~\eqref{AHintro} would correspond to the $H^{1/2}$-approach of Rabinowitz, Conley, and Zehnder. However, as we have already observed, this cannot be used in the case of cotangent bundles since $H^{1/2}(\T,M)$ is not a Hilbert manifold. Notice also that the Hamiltonian action~\eqref{AHintro} is actually well-defined also for $s=1$. Such a setting is used by Hofer and Viterbo in \cite{Hofer:1988} to prove the Weinstein conjecture for a class of compact hypersurfaces in cotangent bundles. However, for $s=1$ the Hamiltonian term does not have compact gradient, and this reflects into the fact that $\A_H$ does not satisfies the Palais-Smale condition. This forces to introduce approximations of $\A_H$ to achieve compactness and then pass to the limit for the approximation going to zero using a very delicate diagonal argument.  \qed 
\end{rmk}

In this paper we upgrade the results in \cite{Asselle:2020b} by showing that one can obtain a genuine Morse complex for $\A_H$ using the Morse complex approach which is developed in \cite{AM:01,AM:03,AM:05,AM:09}. 
In this approach, one constructs a chain complex by looking at the one-dimensional intersections of unstable and stable manifolds of pairs of critical points. %The difference with respect to Floer homology is that the Cauchy-Riemann PDE is replaced by an ODE in the infinite dimensional manifold $\M^{1-s}$.  
Despite being infinite dimensional, unstable and stable manifolds of pairs of critical points always intersect in finite dimensional objects with good compactness properties; see Section 4. This is the case because of the existence of an additional structure, namely a (0)-essential subbundle of the tangent bundle $T\M^{1-s}$, which behaves well under the negative gradient flow of $\A_H$ (more precisely, under a suitable negative pseudo gradient flow). Such an additional structure is needed to make comparisons, in particular to define a notion of relative Morse index for critical points of $\A_H$, and allows to recover a structure which is described in \cite{Cohen} in terms of polarizations. 
Its construction will be performed in Section 3. 
If stable and unstable manifolds of critical points intersect transversally (a condition which can be achieved by a generic perturbation of the negative pseudo gradient vector field preserving all its good compactness properties; see Section 5 for further details), then the intersection is a finite dimensional manifold with dimension equal to the difference of the relative Morse indices. 

In the theorem below the assumption $s\in (1/2,3/4)$ is needed to guarantee that the Hamiltonian action $\A_H$ is sufficiently regular to apply the abstract transversality theorem~\ref{thm:morsesmalegeneral}.

\begin{thm}
\label{thm:main}
Let $M$ be a closed manifold, and let $H:\T\times T^*M \to \R$ be a smooth Hamiltonian which is fiberwise convex and quadratic outside a compact set, see~\eqref{eq:growthcondition}. Then, for every $s\in (1/2,3/4)$, there is a well-defined Morse complex with $\Z_2$-coefficients for the Hamiltonian action $\A_H:\M^{1-s}\to \R$. The induced Morse homology does not depend on the Hamiltonian, and is isomorphic to the singular homology of the free loop space of $M$ as well as to the Floer homology of $T^*M$.  
\end{thm}

In this paper we use $\Z_2$-coefficients instead of $\Z$-coefficients for the sake of simplicity only. Indeed, the (0)-essential subbundle constructed in Section 3 can be given an orientation (in a sense which we do not bother to specify here) which allows us to define coherent orientations for the intersection between stable and unstable manifolds, and hence to construct a Morse complex with $\Z$-coefficients. It will be highly interesting to see to what extent
the orientation of the (0)-essential subbundle enters the construction of the isomorphism between the Morse homology and the singular homology of the free loopspace. Recall indeed that  the use of a twisted version of the Floer complex is required when constructing the isomorphism between 
the Floer homology of $T^*M$ and the singular homology of the free loop space in case $M$ is not spin (more precisely, in case the second Stiefel-Whitney class of $M$ does not vanish on 2-tori, see \cite{Abbondandolo:2006jf,AS:14b}). In case of $\Z_2$-coefficients, a sketch of the construction of the isomorphism between the Morse homology and the the singular homology of the free loopspace is provided in Section 7. A direct comparison 
between the Morse complex and the Floer complex should also be possible by generalizing to cotangent bundles the techniques introduced by Hecht \cite{Hecht} in the case of Hamiltonian systems on tori. 

Despite yielding isomorphic homologies, the Morse complex approach has many advantages over Floer's approach which we shall now briefly describe. First, in the Morse complex approach transversality is achieved in a much 
more elementary way (namely, by generically perturbing the negative pseudo-gradient vector field) and this might be an advantage in more complicated situations where transversality is hard to achieve resp. 
cannot be achieved in Floer theory. Second, the fact that the Morse complex is constructed by intersecting genuine geometric objects (namely, stable and unstable manifolds of pairs of critical points) might suggest the possibility
of applying homotopy arguments rather than merely homology arguments. In this direction, it would be interesting to see to what extent the results of Kragh \cite{Kragh:2012,Kragh:2013} can be recovered using the Morse complex approach. All of this goes the direction of developing methods alternative to Floer theory which are more topological in nature, with e.g. the concrete motivation of improving the known results about the degenerate Arnold conjecture on toric manifolds (in this respect, in the recent preprint \cite{Asselle:2022} in collaboration with Izydorek we give an alternative, purely Conley index based proof of the Arnold conjecture in $\C\mathbb P^n$).

We shall finally notice that, in more linear settings, mixed regularity spaces of loops are used in much more general contexts than the one considered here (for instance, assuming that the regularity loss occur only in
some specific directions). This suggests the possibility that the Morse complex approach can be successfully used for broader classes of symplectic manifolds  (e.g. symplectizations of contact manifolds). 
This is subject of ongoing research. 

We finish this introduction with a brief summary of the content of the paper: 
\begin{itemize}
\item In Section 2, we introduce all notions needed throughout the paper, the functional setting for $\A_H$, and provide 
some preliminary computations showing e.g. the compactness of certain commutator operators on fractional Sobolev spaces with super-critical exponent. 
\item In Section 3 we construct the additional structure needed to make comparisons and define the relative Morse index of critical points of $\A_H$. 
\item In Section 4 we show that such an additional structure behaves well under a suitable negative pseudo-gradient flow of $\A_H$, thus enabling us to prove that the intersection between stable and unstable manifolds of critical points of $\A_H$ is pre-compact. 
\item In Section 5 we prove an abstract transversality result for vector fields on a Hilbert manifold for which the stable and unstable manifolds of rest points are infinite dimensional, and then apply it to show that after a generic perturbation of the negative pseudo-gradient vector field of $\A_H$ we can assume that the stable and unstable manifolds of critical points of $\A_H$ whose relative Morse indices differ at most by 2 intersect transversally. 
\item In Section 6 we employ the content of Sections 2-6 to construct the Morse complex for $\A_H$. 
\item In Section 7 we finally discuss the functioriality properties of the Morse homology.
\end{itemize}

%\begin{rmk}
%The case $s=\frac 12$ is a Sobolev borderline case. For general manifolds $H^{1/2}$-loops need not be continuous; 
%for this reason, the space $H^{1/2}(\T,M)$ does not have the structure of a Hilbert manifold. 
%Nevertheless, the $H^{1/2}$-approach can still be implemented in interesting cases, for example in the proof of Arnold conjecture about the number of fixed points of Hamiltonian diffeomorphisms on tori and projective spaces, see \cite{Conley:1984xb,For:85}.\qed
%\end{rmk}

\vspace{2mm}

\textbf{Acknowledgments.} We are indebted with Alberto Abbondandolo for the countless discussions about the content of this paper.
We also thank him and Pietro Majer for providing us a set of unpublished notes about the construction of the Morse complex for abstract strongly indefinite functionals addressing in particular the questions about transversality 
and functoriality.  
This research is supported by the DFG-project 380257369 ``Morse theoretical methods in Hamiltonian dynamics''.
M.S. is partially supported by the Beethoven2-grant  2016/23/G/ST1/04081 of the National Science Centre, Poland.

%%%%%%%%%%%%%%%%%
%%%%%%%%%%%%%%%%%
%%%%%%%%%%%%%%%%%

\section{Preliminaries}
\label{section:2}

%%%%%%%%%%

\subsection{The functional setting} In this subsection, we introduce the functional setting for the Hamiltonian action $\A_H$. Thus, let $M$ be a closed (i.e. compact without boundary) manifold. To avoid extra complications 
we further assume that $M$ be orientable. 

%In order to keep notational consistency with the main reference \cite{Asselle:2020b}, we hereafter 
%identify tangent and cotangent bundles of $M$ by means of the musical isomorphism 
%$$\flat: TM \to T^*M, \quad X\mapsto \flat (X) := g_{\pi(X)}(X,\cdot)$$
%induced by a fixed metric $g$ on $M$. We recall that the (Levi-Civita connection of the) metric $g$ induces a splitting of $TTM$ into vertical and horizontal sub-bundles 
%\begin{equation}
%TTM \cong HTM \oplus VTM\cong TM\oplus TM,
%\label{splitting}
%\end{equation}
%as well as a metric $g_{TM}$ on $TM$, usually referred to as the \textit{Sasaki metric}, which is uniquely determined by setting $g_{TM}$ equal to $g$ on $HTM$ and $VTM$ 
%and by requiring that the splitting \eqref{splitting} be orthogonal. Also, we have a canonical almost complex structure $J$ on $TM$ determined by the identity 
%$$\omega_\std (\cdot, J \cdot) = g_{TM}(\cdot,\cdot),$$
%where $\omega_\std$ is the pull-back to $TM$ of the standard symplectic form on $T^*M$ via the isomorphism $\flat$. We readily see that, decomposing a tangent vector $\zeta = (\zeta^\hor,\zeta^\ver)\in T_XTM$, $X\in TM$, into horizontal and vertical part $\zeta^\hor,\zeta^\ver \in T_{\pi(X)}M$, we have 
%$$J_X (\zeta^\hor,\zeta^\ver) = (-\zeta^\ver,\zeta^\hor).$$ 

Referring to \cite{Asselle:2020b} for the details, we see that, for $s>\frac 12$, the fractional Sobolev space $H^s(\T,M)$, $\T:=\R/\Z$, of $H^s$-loops in 
$M$ has a natural structure of Hilbert manifold, and for any $r\in [-s,s]$ (see also Lemma~\ref{prop:multiplicationSobolev}) there exists a vector bundle 
$$\pi_r:\mathcal M^{r} \to H^s(\T,M)$$
over $H^s(\T,M)$, whose typical fiber is given by ``$H^r$-sections'' of the pull-back bundle $\ccc^*(T^*M)$, where $\ccc:\T\to M$ is a smooth loop. Local parametrizations for $M^r$ can be given as follows:
for any open subset $U\subset \R^n$ and any smooth diffeomorphism 
$$\T\times U \to \T\times M, \qquad (t,q)\mapsto (t,\varphi(t,q))$$
we have an associated bijection 
$$\varphi_* : H^s(\T,U)\times H^r (\T, (\R^n)^*) \to \M^r, \qquad (\q,\p)\mapsto \big (\varphi(\cdot,\q(\cdot)), \diff \varphi(\cdot, \q(\cdot))^{-*} \p \big ).$$
The collection of the maps $\varphi_*$, for $\varphi$ varying in the above class, form a collection of trivializing maps inducing the structure of a smooth Hilbert vector bundle on $\pi_r:\M^r\to H^s(\T,M)$. 

%choose $\epsilon>0$ smaller than the injectivity radius of $(M,g)$ and let $\ccc\in C^\infty(\T,M)$ be a fixed smooth loop. %Define 
%$$\varphi_\ccc : H^s(\ccc^*\mathbb O_\epsilon) \to H^s(\T,M), \quad \xi \mapsto \varphi_\ccc(\xi),$$
%where $\O_\epsilon:= \{X \in TM \ | \ |X|_{\pi(X)}<\epsilon\}$ is the $\epsilon$-disc bundle over $M$ and $\varphi_\ccc(\xi)(t) := \exp_{\ccc(t)}(\xi(t))$. Then, the corresponding local parametrization 
%of $\M^r$ is given by 
%$$\widetilde \Phi^r_\ccc : H^s(\ccc^*\mathbb O_\epsilon)  \times H^r(\ccc^*TM) \to \U_\ccc\subset \M^r, \quad (\xi,\eta) \mapsto (\varphi_\ccc (\xi),T_\xi \varphi_\ccc [\eta]),$$
%where $\eta$ is ``pointwise'' interpreted as a vertical vector in $T_{\xi(t)} TM$. 
%After choosing a trivialization of $\ccc^*TM$, we set %also view the local parametrization as a map 
%$$\Phi^r_\ccc : H^s (\T,B_\epsilon(0)) \times H^r(\T,\R^n)\to \U_\ccc \subset \mathcal M^r, \quad (\xi, \eta) \mapsto (f_t(\xi(t)), \diff f_t(\xi(t))[\eta(t)]),$$
%where $f_t(\cdot):= \exp_{\ccc(t)}(\cdot)$. %for a suitable time-depending local parametrization $f_t:B_\epsilon(0)\to M$. 
%
%Similarly, we can define a vector bundle 
%$$\pi_r^*:\M^{r,*}\to H^s(\T,M)$$
%whose typical fibre is given by the space of $H^r$-sections of the pull-back bundle $\ccc^*T^*M$, and whose local parametrizations are given by 
%$$\Phi^{r,*}_\ccc:H^s (\T,B_\epsilon(0)) \times H^r(\T,(\R^n)^*)\to \U_\ccc^* \subset \mathcal M^{r,*}, \quad (\xi,\nu) \mapsto (f_t(\xi(t)), \diff f_t (\xi(t))^{-*}[\nu(t)]).$$

The choice of a Riemannian metric $g$ on $M$ induces in a natural way 
a Riemannian metric $\langle \cdot,\cdot \rangle_r$ for the bundle $\M^r$, for every $r$.  %which along a smooth loop $\ccc\in C^\infty(\T,M)$ is given by
%$$\langle \xi,\zeta\rangle_r = \int_0^1 g_\ccc \big ( (\text{id}+\nabla_{\dot \ccc}^*\nabla_{\dot \ccc})^r \xi, \zeta \big )\, \diff t = \langle (\text{id} + \nabla_{\dot \ccc}^* \nabla_{\dot \ccc})^r\xi,\zeta\rangle,$$
%where $\nabla_{\dot \ccc}$ denotes the Levi-Civita covariant derivative along $\ccc$, $\nabla_{\dot \ccc}^* = - \nabla_{\dot \ccc}$ is the $L^2$-adjoint operator, and 
%$$\langle \cdot ,\cdot \rangle := \int_0^1 g_\ccc (\cdot,\cdot)\, \diff t$$ 
%is the $L^2$-metric. 
The norm induced by the $L^2$-metric will be denoted by $\|\cdot\|$. Similarly, we denote by 
$\|\cdot \|_\infty$ the induced $L^\infty$-norm. 
%The corresponding metrics on $\M^{r,*}$ are defined in the usual way and denoted in the same way as for $\M^r$ for notational convenience. 

For $s\in (\frac 12,1)$ we consider the Hilbert-bundle $\pi_{1-s}:\mathcal M^{1-s}\to H^s(\T,M)$. 
Given a smooth time-depending Hamiltonian function $H:\T\times T^*M\to \R$ such that 
\begin{equation}
H(t,q,p) = \frac 12 |p-\theta_q|_q^2 + U(t,q), \quad \forall t\in \T,
\label{eq:growthcondition}
\end{equation}
outside a compact set $K\subset T^*M$, where $\theta\in \Omega^1(M)$ is a (possibly time-depending) one-form on $M$ (a so called \textit{magnetic potential}) and $U:\T\times M\to \R$ is a smooth time-depending potential, we can define the Hamiltonian action functional by
\begin{align}
\A_H : \mathcal M^{1-s} \to \R,\quad \A_H(\q,\p) &= \int_0^1 (\q,\p)^*\lambda - \mathcal H(\q,\p), \label{AH}
								\end{align}
where $\lambda$ denotes the Liouville one-form on $T^*M$ and
$$\mathcal H:\M^{1-s}\to \R,\quad \mathcal H(\q,\p):= \int_0^1 H(t,\q(t),\p(t))\, \diff t.$$

Throughout this paper we will assume without further mentioning it that all critical points of $\A_H$ are \textit{hyperbolic}, which in this case is equivalent to the fact 
that the second differential of $\A_H$ at each critical point seen as a symmetric bounded 
bilinear form is non-degenerate. %As it is well-known, such a condition can be achieved by a generic choice of the potential $U$ (see e.g. \cite{Abbondandolo:2006jf}).

In the theorem below we summarize the properties of the functional $\A_H$, referring to \cite[Section 2]{Asselle:2020b} for the proof. We shall notice that 
the Palais-Smale condition is proved in \cite{Asselle:2020b} only for Hamiltonians which are kinetic (up to a constant) outside a compact set, 
that is only for $\theta\equiv 0$ and $U\equiv \text{c}$, however the proof extends verbatim to the case in which $\theta$ is an arbitrary one-form on $M$ and $U$ is a time-depending potential. 

We recall that we can define a natural Riemannian metric on $\mathcal M^{1-s}$ by 
\begin{equation}
\langle \cdot,\cdot\rangle_{\M^{1-s}}:= \langle \cdot^\hor ,\cdot ^\hor\rangle_s + \langle \cdot^\ver,\cdot^\ver\rangle_{1-s},
\label{metricms}
\end{equation}
where the splitting $T \M^{1-s}  \cong H \M^{1-s}\oplus V \M^{1-s}$ into horizontal and vertical subbundle is induced by the $L^2$-connection (roughly speaking, the Levi-Civita connection of $g$ applied pointwise).

A sequence $(\q_n,\p_n)\subset \mathcal M^{1-s}$ is called a \textit{Palais-Smale sequence} for $\A_H$ if $\A_H(\q_n,\p_n)\to a$ for some $a\in \R$ and $\|\diff \A_H(\q_n,\p_n)\|\to 0$. Here, with slight abuse of notation, we denote with $\|\cdot\|$ the dual norm on $T^*_{(\q_n,\p_n)}\mathcal M^{1-s}$ induced by the Riemannian metric 
$\langle \cdot,\cdot\rangle_{\mathcal M^{1-s}}$ in \eqref{metricms}. 

\begin{thm}
\label{lem:compactness}
For every $s\in (\frac 12, 1)$, the following statements hold:
\begin{enumerate}
\item $\A_H$ is well-defined over $\mathcal M^{1-s}$ and at least of class $C^{1,1}$. More precisely, there exists $k=k(s)\in \N$ such that 
$\A_H:\M^{1-s}\to \R$ is of class $C^k$, with $k(s)\to +\infty$ 
as $s\downarrow \frac 12$.
\item The operator $\diff \mathcal H$ is compact. 
\item  Critical points of $\A_H$ correspond to one-periodic solutions of Hamilton's Equation.
\item $\A_H$ satisfies the Palais-Smale condition.
\item The action of critical points of $\A_H$ is uniformly bounded from below. 
\end{enumerate}
\end{thm}
\begin{proof}
Statements (1)-(4) are proved in \cite[Section 2]{Asselle:2020b}. To prove (5) we observe that by assumption 
$$\delta :\T \times T^*M \to \R, \quad \delta(t,q,p) := H(t,q,p) - \frac 12 |p-\theta_q|_q^2 - U(t,q),$$
is a smooth compactly supported function and set 
$$c:= \max \Big \{ \|\delta\|_\infty, \|\partial_p \delta\|_\infty, \|U\|_\infty,\|\theta\|_\infty\Big \}.$$
We also notice that, setting $\langle \dot \q,\p\rangle := \int_0^1 (\q,\p)^*\lambda$, we have 
$$\A_H(\q,\p) = \langle \dot \q,\p\rangle - \frac 12 \|\p-\theta_\q\|^2 - \int_0^1 U(t,\q(t))\, \diff t - \int_0^1 \delta(t,\q(t),\p(t))\, \diff t$$ 
and hence 
$$\diff \A_H (\q,\p) [0,\p] = \langle \dot \q,\p\rangle - \langle \p-\theta_\q,\p\rangle -\int_0^1 \partial_p \delta (t,\q(t),\p(t))\cdot \p (t)\, \diff t.$$
Finally, we compute for $(\q,\p)\in$ crit$(\A_H)$
\begin{align*}
\A_H(\q,\p) &= \A_H (\q,\p) - \diff \A_H(\q,\p) [0,\p] \\
		&=  \frac 12 \|\p\|^2 - \frac 12 \|\theta_\q\|^2 - \int_0^1 U(t,\q(t))\, \diff t - \int_0^1 \delta (t,\q(t),\p(t))\, \diff t + \int_0^1 \partial_p \delta (t,\q(t),\p(t)) \cdot \p(t)\, \diff t\\
		&\geq \frac 12 \|\p\|^2 - c (\|\p\| +3)\\
		&\geq - \frac12 c^2 - 3c. \qedhere
\end{align*}
\end{proof}

%%%%%%%%%%

\subsection{Stable and unstable manifolds}
In this subsection we recall the definition of stable resp. unstable manifold, 
referring to \cite{Shub:1987} or \cite[Appendix C]{AM:05} for further details. 
Let $\M$ be Hilbert manifold modeled on the infinite dimensional real separable Hilbert space $\HH$, and let $F$ be a $C^1$ vector field on $\M$. We denote by $\Phi^F:\Omega(F)\subset \R\times \M\to \M$ 
the local flow of $F$. If $x\in \M$ is an hyperbolic rest point of $F$, meaning that the spectrum of the Jacobian of $F$ at $x$ is disjoint from the imaginary axis, then we can define the 
 \textit{unstable} and \textit{stable manifolds} of $x$ as
\begin{align*}
W^u(x;F) &:= \Big \{p\in \M \ \Big |\ (-\infty,0]\times \{p\} \subset \Omega(F), \ \ \lim_{t\to -\infty} \phi^F(t,p) = x \Big \},\\
W^s(x;F) &:= \Big \{p\in \M \ \Big |\ [0,+\infty)\times \{p\} \subset \Omega(F), \ \ \lim_{t\to +\infty} \phi^F(t,p) = x \Big \}.
\end{align*}

The stable manifold theorem implies that
such sets are actually $C^1$-submanifolds of $\M$ with dimension given by the Morse index resp. co-index of $x$ whenever 
$F$ admits a global $C^1$ Lyapunov function which is twice differentiable and non-degenerate at $x$. %Indeed, they are images of $C^1$-embeddings 
%$$\HH^u\to \M, \quad \HH^s\to \M,$$
%where the splitting $\HH=\HH^u\oplus \HH^s$ is obtained by identifying a neighborhood of $x$ in $\M$ with a neighborhood of $0$ in $\HH$ and then considering the partition of the spectrum of $\diff^2 f(x)$ into positive resp. negative part.

%%%%%%%%%%

\subsection{Essential subbundles} 
In this subsection we will recall the definition and some general properties of essential subbundles that will be useful later on, referring to \cite{AM:05} for further details. For our purposes, we will 
only need to consider essential subbundles of the tangent bundle of a Hilbert manifold $\M$, but everything can be extended verbatim to general Hilbert bundles over Banach manifolds. 

Thus, let $\M$ be a Hilbert manifold modeled on the infinite dimensional separable real Hilbert space $\HH$. The 
\textit{Grassmannian} of $\HH$ is the space 
$$\text{Gr}\, (\HH) := \big \{ V\subset \HH \ \big |\ V \text{ closed linear subspace}\big \}$$
of all closed linear subspaces of $\HH$, endowed with the operator norm topology. Since we will be only interested in subspaces with infinite dimension and codimension, all 
subspaces of $\HH$ appearing hereafter are supposed to be infinite dimensional and codimensional without further specifying it. Also, we will
denote the connected component $\text{Gr}_{\infty,\infty}(\HH)$ of all infinite dimensional and codimensional subspaces again with $\text{Gr}\, (\HH)$. 
 
Given $V,W\in \text{Gr}\, (\HH)$, we will say that $V$ is a \textit{compact perturbation} of $W$ it $P_V-P_W$ is a compact operator, where $P_V$, $P_W$ denotes the orthogonal projection onto $V$ resp. $W$. 
In this case, the \textit{relative dimension} of $V$ with respect to $W$ is the integer 
$$\dim (V,W) := \dim (V\cap W^\perp) - \dim (V^\perp \cap W).$$ 

For $m\in \N$, the \textit{(m)-essential Grassmannian} $\text{Gr}_{(m)}^* (\HH)$ is the quotient space of $\text{Gr}_{\infty,\infty} (\HH)$ by the equivalence relation 
$$V\sim_{(m)} W \quad \Leftrightarrow \quad V \ \text{is a compact perturbation of } W, \ \dim (V,W)\in m\Z.$$
The space $\text{Gr}_{(1)}(\HH)$ is called simply the \textit{essential Grassmannian} of $\HH$. 

The tangent bundle $T\M \to \M$ is a smooth Hilbert bundle with typical fibre $\HH$ and structure group $GL(\HH)$. Since $GL(\HH)$ is contractible (see \cite{Kuiper:1965}), the tangent bundle is always trivial. 
We define
\begin{align*}
\text{Gr}\, (T\M) &= \bigcup_{p\in \M} \text{Gr}\,(T_p\M)\to \M,\\ 
 \text{Gr}_{(m)} (T\M) & = \bigcup_{p\in \M} \text{Gr}_{(m)}(T_p\M)\to \M,\ \ m\in \N.
 \end{align*}
 It is not hard to see that the above bundles have smooth structures.
A smooth section of $\text{Gr}\,(T\M) \to \M$ is just a smooth subbundle of $T\M\to \M$. 
Similarly, a smooth section of $\text{Gr}_{(m)}(T\M)\to \M$ will be called an \textit{(m)-essential subbundle} of $T\M \to \M$, or simply an \textit{essential subbundle} if $m=1$. 

\begin{rmk}
\label{rmk:liftability}
It is natural to ask which conditions ensure that an $(m)$-essential subbundle, $m\in \N$, is liftable e.g. to a true subbundle or to a $(0)$-essential subbundle. Following \cite[Section 1.4]{AM:05} we see that 
every $(m)$-essential subbundle is liftable to a true subbundle when the homotopy groups of $\M$ vanish in all degrees congruent to $1,2,3,$ and $5$ mod $8$, and to a $(0)$-essential subbundle if $\M$ 
is simply connected. 

An equivalent definition of essential subbundles can be given as follows: Let $\{(\varphi_\alpha,\U_\alpha)\}_{\alpha\in \mathcal A}$ be an atlas of $\M$, and let $\mathcal E_\alpha\subset T\U_\alpha$ be a subbundle for 
every $\alpha\in \mathcal A$. The family $\mathcal E:=\{\mathcal E_\alpha\}$ is called an \textit{essential subbundle} of $T\M$, if 
$$\mathcal E_\alpha \Big |_{\U_\alpha\cap \U_\beta} \ \text{is a compact perturbation of } \mathcal E_{\beta} \Big |_{\U_\alpha\cap \U_\beta}, \quad \forall \alpha,\beta\in \mathcal A.$$
Such an essential subbundle $\mathcal E$ is an $(m)$-essential subbundle, $m\in \N$, if we additionally have 
\begin{equation*}
\hspace{5cm} \dim (\mathcal E_\alpha ,\mathcal E_\beta)=0 \quad \text{mod} \ m, \quad \forall \alpha,\beta\in \mathcal A. \hspace{5cm}\qed
\end{equation*}
\end{rmk}

An essential subbundle $\mathcal E$ of $T\M$ is called \textit{strongly integrable} if $\M$ admits an atlas $\{(\varphi_\alpha,\U_\alpha)\}$ such that each $\varphi_\alpha$ maps $\mathcal E$ into the essential subbundle of $T\HH$ 
represented by a constant closed linear subspace $V\subset \HH$
\begin{equation}
D\varphi_\alpha (p) \mathcal E(p) = [V], \quad \forall p \in \U_\alpha,\quad \forall \alpha,
\label{eq:integrable1}
\end{equation}
and for every $\alpha,\beta$ the transition map
$$\tau:= \varphi_\alpha\circ \varphi_\beta^{-1}:\text{dom} (\tau)= \varphi_\beta(\U_\alpha\cap \U_\beta)\subset \HH \to \HH$$
satisfies the following property: for every bounded $A\subset \text{dom}(\tau)$, 
\begin{equation}
QA \text{ is pre-compact if and only if } Q\tau(A) \ \text{is pre-compact},
\label{eq:integrable2}
\end{equation}
where $Q$ denotes a projector with kernel $V$. We shall observe that the definition above does not depend on the choice of the projector $Q$, but only on the subspace $V$. Indeed, the set $QA$ is pre-compact if and 
only if the projection of $A$ into the quotient space $\HH/V$ is pre-compact. In this case, the atlas $\{(\varphi_\alpha,\U_\alpha)\}$ is called a \textit{strong integrable structure} modeled on $(\HH,V)$ for the essential
subbundle $\mathcal E$. 

In what follows we will be interested in the inverse construction, since it will allow us to build essential subbundles starting from ``natural'' local data: 
If $\{(\varphi_\alpha,\U_\alpha)\}$ is an atlas of $\M$ such that \eqref{eq:integrable2} is satisfied with respect to some $V\in \text{Gr}\, (\HH)$, 
then defining $\mathcal E$ by \eqref{eq:integrable1} yields a (by construction strongly integrable) essential subbundle of $T\M$. Such an $\mathcal E$ is liftable to an $(m)$-essential subbundle, $m\in\N$, if and only if
$$\dim (D\tau(\xi) V,V)=0 \quad \text{mod } m.$$

We finish this subsection recalling that there is also a notion of \textit{integrable subbundle}, which is obtained by requiring that for all transition maps $\tau$
\begin{equation}
D\tau(\xi) V \ \text{is a compact perturbation of } V, \quad \forall \xi \in \text{dom}(\tau),
\label{eq:integrable3}
\end{equation} 
hold instead of \eqref{eq:integrable2}. As one readily sees, strong integrability is strictly more restrictive than integrability, because \eqref{eq:integrable2} implies \eqref{eq:integrable3} by differentiation, but on the other hand a non-linear map whose differential at every point is compact need not be compact. 

The reason why strong integrability is better suited for our purposes than integrability is the following: in case of a functional $f$ whose critical points have infinite Morse index and co-index, the integrability of an 
essential subbundle is not enough to conclude pre-compactness of the intersection $W^u(x)\cap 
W^s(y)$ between the unstable and stable manifolds of two critical points $x$ and $y$, even if the integrable essential subbundle is ``nicely'' related with 
the negative gradient flow of $f$ and the Palais-Smale condition holds, see \cite[Section 3]{AM:05}. 
In fact, the intersection $W^u(x)\cap W^s(y)$, though finite dimensional, might consist of infinitely many curves with no cluster points besides 
$x$ and $y$. As explained in \cite[Section 0.3]{AM:05}, this follows from the fact that the conditions for an integrable essential subbundle to be ``nicely'' 
related with the negative gradient flow of $f$ are of local nature, whereas compactness
involves a global condition. On the other hand, if the essential subbundle is additionally strongly integrable, then we can define in a natural way a global condition which implies pre-compactness of the intersections; see \cite[Section 6]{AM:05} or Section \ref{s:precompactness} for further details. Such a condition is formulated in terms of invariance of an essentially vertical family, whose definition we now recall, under 
a suitable negative pseudo gradient vector field for $f$. 

\begin{dfn}
\label{dfn:evf}
Let $\M$ be a Hilbert manifold modeled on the infinite dimensional separable real Hilbert space $\HH$ endowed with a complete Riemannian metric, and let 
$\mathcal E\subset T\M$ be a strongly integrable essential subbundle modeled on $(\HH,V)$ with strong integrable structure $\mathcal A = \{(\varphi_\alpha,\mathcal U_\alpha)\}$. Denote further with $Q:\HH\to \HH$ 
a projector with kernel $V$. A family $\mathcal F$ of subsets of $\M$ is called an \textit{essentially vertical family} for the strong integrable structure $\mathcal A$ of $\mathcal E$ if it satisfies:
\begin{enumerate}
\item[(i)] $\mathcal F$ is an ideal of $\mathcal P(\mathcal M)$, that is, it is closed under finite unions and if $A\in \mathcal F$ then any subset of $A$ is also contained in $\mathcal F$,
\item[(ii)] for every point $p\in \M$ there is a local chart $(\varphi, \mathcal U)\in \mathcal A$ such that every set $A\subset \mathcal U$ with $\varphi(A)$ bounded belongs to $\mathcal F$ if and only if 
$Q\varphi (A)$ is pre-compact, and 
\item[(iii)] if $B\subset \mathcal M$ is contained in an $\epsilon$-neighborhood of some $A_\epsilon\in \mathcal F$ for every $\epsilon >0$, then $B\in \mathcal F$. 
\end{enumerate}
Any set $A\in \mathcal F$ is called \textit{essentially vertical set}.
\end{dfn}

Condition (iii) is in other words saying that $\mathcal F$ is closed with respect to the Hausdorff distance. We shall observe that such a condition was not required in \cite{AM:05}, the reason being 
that it is not needed in the proof of the pre-compactness result for the intersection between stable and unstable manifolds of critical points, see \cite[Theorem 6.5]{AM:05}.
Here we instead require it since it is needed 
in the proof of the abstract transversality theorem~\ref{thm:morsesmalegeneral}, as well as in the proof of the fact that the space of vector fields preserving $\mathcal F$ is a module over the space of Lipschitz functions (see Proposition~\ref{prop:module}).
As we shall see, such a condition is trivially satisfied by the family $\mathcal F$ 
which naturally arises in the construction of the Morse complex for the Hamiltonian action. 

%%%%%%%%%%%%%%%%

\subsection{Some preliminary computations}

In this section we collect some preliminary estimates and computations about multiplication and commutator operators on Sobolev spaces that will be useful later on. In what follows we denote with $\{e_k:=e^{2\pi k Jt}\}$ the 
standard Hilbert basis of $L^2(\T,\R^n)$. For 
$$u= \sum_{k\in \Z} u_k e_k$$ 
and $s \in \R$ we set 
$$\|u\|_s^2 := \sum_{k\in \Z} (1+|k|)^{2s} |u_k|^2.$$
For $s\geq 0$ we define the Sobolev space
$$H^s(\T,\R^n) := \Big \{u\in L^2(\T,\R^n) \ \Big |\ \|u\|_s<+\infty\},$$
and set $H^{-s}(\T,\R^n)$ to be its dual space. In a similar way we define the Sobolev spaces 
$$H^s(\T,\text{Hom}(\R^n,\R^n)), \qquad s\in \R,$$
where in this case the $H^s$-norm is defined using any equivalent matrix-norm for the Fourier coefficients of a loop $A:\T\to \text{Hom}(\R^n,\R^n)$ with respect to the basis $\{e_k\}$.
For any $\sigma\in \R$ we finally define the $\sigma$-power of the Laplacian $\Delta$ by 
$$\Delta^\sigma u := (2\pi)^{2\sigma} \sum_{k\in \Z^*} |k|^{2\sigma} u_k e_k.$$

We start with the following elementary 

\begin{lem}
For $\alpha,\beta,\gamma\geq 0$ such that $\alpha+\beta>1$, $0\leq \gamma \leq \beta$ and $\gamma < \alpha+\beta -1$ we have 
$$\sum_{h=1}^{+\infty} \frac{1}{h^\alpha(k+h)^\beta} = O( \frac1{k^\gamma}) \qquad \text{for}\ k\to +\infty,$$
\label{lem:abbo1}
\end{lem}
\begin{proof}
Since $\gamma\leq \beta$ we have 
\begin{align*}
\sum_{h=1}^{+\infty}  \frac{1}{h^\alpha(k+h)^\beta} & \leq \sum_{h=1}^{+\infty}  \frac{1}{h^{\alpha+\beta-\gamma}(k+h)^\gamma} \\
										&= \frac{1}{k^\gamma} \sum_{h=1}^{+\infty}  \frac{1}{h^{\alpha+\beta-\gamma}(1+h/k)^\gamma}\\
										&\leq  \frac{1}{k^\gamma} \sum_{h=1}^{+\infty}  \frac{1}{h^{\alpha+\beta-\gamma}}
										\end{align*}
										and the claim follows  as the latter series converges by the assumption $\alpha+\beta-\gamma>1$.
\end{proof}

The following result about multiplication operators on Sobolev spaces can be found in \cite[Lemma 28]{Holst:2009} (see also \cite{Zolesio:1977}) for the case of real- or complex-valued functions.
As we are here dealing with vector-valued functions, we include an elementary proof for the sake of completeness. 

\begin{prop}\label{prop:multiplicationSobolev} 
Fix $s>1/2$. Then, for every $r\in [-s,s]$, the pointwise matrix-vector multiplication of a matrix-valued function with a vector-valued function extends uniquely to a continuous bilinear map 
$$H^s(\T,\mathrm{Hom}(\R^n,\R^n))\times H^r(\T,\R^n)\to H^r(\T,\R^n)$$
resp. 
$$H^r(\T,\mathrm{Hom}(\R^n,\R^n))\times H^s(\T,\R^n)\to H^r(\T,\R^n).$$
In other words, for every $r\in [-s,s]$ there exists $c_r>0$ such that 
\begin{align*}
\|A\cdot  u\|_r &\leq c_r \|A\|_s \| u\|_r, \qquad \forall A\in H^s(\T,\mathrm{Hom}(\R^n,\R^n)), \ \forall u \in H^r(\T,\R^n),\vspace{5cm}  \\
\|A\cdot  u\|_r &\leq c_r \|A\|_r \| u\|_s, \qquad \forall A\in H^r(\T,\mathrm{Hom}(\R^n,\R^n)), \ \forall u \in H^s(\T,\R^n).\vspace{5cm} 
\end{align*} 
\end{prop}

\begin{proof}
We prove only the first estimate being the proof of the second one identical. We first consider the case $r\in [0,s]$. Clearly, we can further assume that $r\in (0, 1/2]$, as the claim is obvious for $r=0$ and it follows 
from the standard estimates for the product of two Sobolev functions with supercritical Sobolev exponents for $r\in (1/2,s]$. 
By a standard density argument we can assume that $A$ and $u$ are smooth. We denote with $A_k$ the Fourier coefficients of $A$ and  with $u_k$ the Fourier coefficients of $\dot u$. Clearly, we have
$$A\cdot u = \sum_{k\in\Z} \Big (\sum_{h\in \Z} A_{k-h}u_h\Big ) e_k,$$
and thus using the Cauchy-Schwarz inequality
\begin{align*}
\|A\cdot u \|_r^2 & = \sum_{k\in \Z} (1+|k|)^{2r} \Big | \sum_{h\in\Z} A_{k-h}u_h\Big |^2 \\\
			&\leq \sum_{k\in\Z} (1+|k|)^{2r} \Big (\sum_{h\in\Z} \frac{1}{(1+|k-h|)^{2s}(1+|h|)^{2r}}\Big ) \Big ( \sum_{h\in \Z} \|A_{k-h}\|^2 (1+|k-h|)^{2s}|u_h|^2(1+|h|)^{2r}\Big ).
\end{align*}
We preliminary notice that, for fixed $k\in \Z$, the sum 
$$G_r(k):=\sum_{h\in\Z} \frac{1}{(1+|k-h|)^{2s}(1+|h|)^{2r}}$$
is finite, since by assumption $2s+2r>1$.
All we have to show is that 
$$G_r(k) = O ( \frac{1}{|k|^{2r}}) \qquad \text{for}\ |k|\to +\infty.$$
Indeed, if this is the case then the function $k\mapsto (1+|k|)^{2r}G_r(k)$ is bounded on $\Z$ and hence we obtain for some constant $c>0$
$$\|A\cdot u \|_r^2 \leq c \sum_{k\in \Z}\sum_{h\in \Z} \|A_{k-h}\|^2 (1+|k-h|)^{2s}|u_h|^2(1+|h|)^{2r} = c \|A\|_s^2 \|u \|_r^2.$$
Since $G_r(-k)=G_r(k)$, we can assume without loss of generality $k\in\N$. We split the sum as follows
\begin{align*}
G_r(k) & = G_1(k) + G_2(k)+G_3(k)+G_4(k)\\	
	&= \sum_{h\leq 0} \frac{1}{(1+|k-h|)^{2s}(1+|h|)^{2r}} + \sum_{h=1}^{\lfloor k/2\rfloor} \frac{1}{(1+|k-h|)^{2s}(1+|h|)^{2r}}\\
	& + \sum_{h=\lfloor k/2\rfloor +1}^{k-1} \frac{1}{(1+|k-h|)^{2s}(1+|h|)^{2r}}+ \sum_{h\geq k} \frac{1}{(1+|k-h|)^{2s}(1+|h|)^{2r}}
	\end{align*}
and estimate each of these terms separately. For $G_1(k)$ we use Lemma~\ref{lem:abbo1} with $\alpha=\gamma=2r$ and $\beta=2s$ to infer that for $k\to +\infty$
$$G_1(k) = \sum_{h=0}^{+\infty} \frac{1}{(1+k+h)^{2s}(1+h)^{2r}}= \sum_{h=1}^{+\infty} \frac{1}{(k+h)^{2s}h^{2r}} = O(\frac{1}{k^{2r}}).$$
Similarly, for $G_4(k)$ we have for $k\to +\infty$ 
$$G_4(k) = \sum_{h\geq k} \frac{1}{(1+ h-k)^{2s}(1+h)^{2r}}\leq \frac{1}{k^{2r}} \sum_{h\geq k} \frac{1}{(1+ h-k)^{2s}}= O(\frac{1}{k^{2r}})$$
as the latter series converges by the assumption $s>1/2$. For $G_3(k)$ we estimate similarly
\begin{align*}
G_3(k) & \lesssim \frac{1}{k^{2r}} \sum_{h=\lfloor k/2\rfloor +1}^{k-1} \frac{1}{(1+k-h)^{2s}}\leq \frac{1}{k^{2r}} \sum_{h=1}^{+\infty} \frac{1}{h^{2s}} = O(\frac{1}{k^{2r}}) \qquad \text{for} \ k \to +\infty.
\end{align*}
In order to bound $G_2(k)$ we preliminary compute, for $h=1,...,\lfloor k/2\rfloor$, using the fact that the function $x\mapsto x^{s-r}$ is monotonically increasing and sub-additive
$$(1+k-h)^{s-r} \geq (1+k)^{s-r} - h^{s-r}\geq (2^{s-r}-1)h^{s-r}.$$
This yields 
\begin{align*}
G_2(k) &= \sum_{h=1}^{\lfloor k/2\rfloor} \frac{1}{(1+k-h)^{2s}(1+h)^{2r}}\lesssim \frac{1}{k^{2r}} \sum_{h=1}^{\lfloor k/2\rfloor} \frac{1}{(1+k-h)^{2(s-r)}(1+h)^{2r}}\lesssim  \frac{1}{k^{2r}} \sum_{h=1}^{\lfloor k/2\rfloor} \frac{1}{h^{2s}}
\end{align*}
thus completing the proof in the case $r\in [0,s]$. 

The case of negative Sobolev exponents follows now by a simple duality argument. Indeed, for $r\in [0,s]$ and for every smooth $v$ with $\|v\|_r=1$ we have 
\begin{align*}
|\langle A\cdot u,v\rangle | &= |\langle u, A^T\cdot v\rangle | \leq \|u\|_{-r} \|A^T\cdot v\|_r \leq c \|u\|_{-r} \|A^T\|_s \|v\|_r = c \|u\|_{-r} \|A\|_s,
\end{align*}
where we have used the estimate proved above for the case of non-negative Sobolev exponents. Here, $\langle \cdot,\cdot \rangle$ denotes the duality pairing. 
Taking the supremum over $\|v\|_r=1$ yields 
$$ \|A\cdot u \|_{-r} = \sup_{\|v\|_r=1} |\langle A\cdot u,v\rangle | \leq c \|u\|_{-r} \|A\|_s,$$
thus finishing the proof. 
\end{proof}

\begin{cor}
For every $s>1/2$ the multiplication operators
\begin{align*}
& H^s(\T,\mathrm{Hom}(\R^n,\R^n)) \times H^{s-1}(\T,\R^n)\to H^{s-1}(\T,\R^n), \qquad (A,u) \mapsto A \cdot u\\
& H^{s-1}(\T,\mathrm{Hom}(\R^n,\R^n)) \times H^{s}(\T,\R^n)\to H^{s-1}(\T,\R^n), \qquad (A,u) \mapsto A \cdot u,
\end{align*}
are well-defined and bounded. Moreover, for fixed $A\in H^{s-1}(\T, \mathrm{Hom}(\R^n,\R^n))$, the operator 
\begin{align*}
H^s(\T,\R^n)\to H^{s-1}(\T,\R^n), & \qquad u \mapsto  A\cdot u,
\end{align*}
is compact.
\label{cor:multiplications-1}
\end{cor}

\begin{proof}
The first claim follows directly from Proposition~\ref{prop:multiplicationSobolev} noticing that  $s-1\in [-s,s]$ for $s>1/2$. The second claim follows also easily from Proposition~\ref{prop:multiplicationSobolev}. Indeed, since $1-s<1/2<s$, 
the considered operator can be written as the composition of the compact embedding $H^s\hookrightarrow H^{r}$, for any $1/2<r<s$, with the bounded (since $s-1 \in [-r,r]$) operator
\begin{align*}
H^{r}(\T,\R^n)\to H^{s-1}(\T,\R^n), & \qquad u \mapsto A\cdot u. 	\qedhere
\end{align*}
\end{proof}

The next technical lemma will be needed to prove the boundedness of the commutator operator 
$$[A,\Delta^{s-1}\circ \frac{\diff}{\diff t}] : H^r(\T,\R^n)\to H^{1-s}(\T,\R^n), \qquad r>\frac 12,$$
for fixed $A\in H^s(\T,\text{Hom}(\R^n,\R^n))$, $s\in (1/2,1)$, see Proposition~\ref{prop:abbo}.

\begin{lem}
\label{lem:abbo2}
Fix $s\in (1/2,1)$ and $r>1/2$. For $k\in \Z$ set 
$$F(k) := \sum_{h\in \Z} \frac{\big | |k|^{2(s-1)}k - |h|^{2(s-1)}h \big |^2}{(1+|k-h|)^{2s} (1+|h|)^{2r}} .$$
Then 
$$F(k)= O (\frac 1{|k|^{2(1-s)}}) \qquad \text{for}\ |k|\to +\infty.$$
\end{lem}
\begin{proof}
Notice that $F(k)$ is finite for every $k\in\Z$ since by the assumptions on $s$ and $r$
$$2s+2r-2(2s-1) = 2 (r-s+1)>1.$$
Also, since $F(-k)=F(k)$ it is enough to consider the case $k\to +\infty$. Given $k>0$, we split 
$$F(k) = F_1(k)+F_2(k)+F_3(k),$$
where 
\begin{align*}
F_1(k) &=  \sum_{h=0}^{k-1} \frac{\big | |k|^{2(s-1)}k - |h|^{2(s-1)}h \big |^2}{(1+|k-h|)^{2s} (1+|h|)^{2r}},\\
F_2(k) &= \sum_{h\geq k} \frac{\big | |k|^{2(s-1)}k - |h|^{2(s-1)}h \big |^2}{(1+|k-h|)^{2s} (1+|h|)^{2r}},\\
F_3(k) &= \sum_{h<0} \frac{\big | |k|^{2(s-1)}k - |h|^{2(s-1)}h \big |^2}{(1+|k-h|)^{2s} (1+|h|)^{2r}},
\end{align*}
and bound each of these terms separately. By the assumption on $s$ the functions $x\mapsto x^s$ and $x\mapsto x^{1-s}$ are monotonically increasing 
and subadditive, and we deduce 
\begin{align*}
k^{2(1-s)} F_1(k) &= k^{2(1-s)}  \sum_{h=0}^{k-1} \frac{\big ( k^{2s-1} - h^{2s-1} \big )^2}{(1+k-h)^{2s} (1+h)^{2r}}\\
			&= \sum_{h=0}^{k-1} \frac{\big ( k^{s} - k^{1-s}h^{2s-1} \big )^2}{(1+k-h)^{2s} (1+h)^{2r}}\\
			&\leq  \sum_{h=0}^{k-1} \frac{\big ( k^{s} - h^{s} \big )^2}{(1+k-h)^{2s} (1+h)^{2r}}\\
			&\leq \sum_{h=0}^{k-1} \frac{\big ( k - h \big )^{2s}}{(1+k-h)^{2s} (1+h)^{2r}}\\
			&\leq \sum_{h=0}^{k-1} \frac{1}{(1+h)^{2r}}\\
			&\leq \sum_{h=1}^{+\infty} \frac 1{h^{2r}}.
\end{align*}
The claim 
$$F_1(k) = O(\frac{1}{k^{2(s-1)}}) \qquad \text{for} \ k\to +\infty$$
follows as the latter series in the above chain of inequalities converges by the assumption $r>1/2$. 

By the subadditivity of the function $x\mapsto x^{2s-1}$ we find the following upper bound for $F_2(k)$: 
\begin{align*}
F_2(k) &=  \sum_{h= k}^{+\infty} \frac{\big ( h^{2s-1} - k^{2s-1}\big )^2}{(1+h-k)^{2s} (1+h)^{2r}}\\
      &\leq \sum_{h= k}^{+\infty} \frac{\big ( h - k\big )^{2(2s-1)}}{(1+h-k)^{2s} (1+h)^{2r}}\\
      &\leq \sum_{h= k}^{+\infty} \frac{\big ( 1+ h - k\big )^{2(2s-1)}}{(1+h-k)^{2s} (1+h)^{2r}}\\
      &= \sum_{h= k}^{+\infty} \frac{1}{(1+h-k)^{2(1-s)} (1+h)^{2r}}\\
      &= \sum_{\ell=1}^{+\infty} \frac{1}{\ell^{2(1-s)}(k+\ell)^{2r}}.
\end{align*}
Since $s\in (1/2,1)$ and $r>1/2$, Lemma~\ref{lem:abbo1} with $\alpha=\gamma=2(1-s)$ and $\beta=2r$ yields 
$$F_2(k) = O(\frac{1}{k^{2(1-s)}}) \qquad \text{for}\ k\to +\infty.$$
We now bound $F_3(k)$: 
\begin{align*}
F_3(k) &= \sum_{h=1}^{+\infty} \frac{\big ( k^{2s-1} + h^{2s-1}\big )^2}{(1+k+h)^{2s}(1+h)^{2r}}\\
	&\leq 2 \sum_{h=1}^{+\infty} \frac{ k^{2(2s-1)} + h^{2(2s-1)}}{(k+h)^{2s}h^{2r}}\\
	&= 2 k^{2(2s-1)}\sum_{h=1}^{+\infty} \frac{1}{(k+h)^{2s}h^{2r}} + 2 \sum_{h=1}^{+\infty} \frac{1}{(k+h)^{2s}h^{2(r+1-2s)}}.
\end{align*}
The first series in the last expression can be estimated using Lemma~\ref{lem:abbo1} with $\alpha=2r$, $\beta=\gamma=2s$, thus obtaining 
$$k^{2(2s-1)}\sum_{h=1}^{+\infty} \frac{1}{(k+h)^{2s}h^{2r}} = k^{2(2s-1)} O(\frac{1}{k^{2s}}) = O(\frac{1}{k^{2(1-s)}}) \qquad \text{for}\ k\to +\infty,$$
whereas applying Lemma~\ref{lem:abbo1} to the second series with $\alpha=2(r+1-2s)$, $\beta=2s$, and $\gamma=2(1-s)$ yields
$$\sum_{h=1}^{+\infty} \frac{1}{(k+h)^{2s}h^{2(r+1-2s)}}= O ( \frac{1}{k^{2(1-s)}}) \qquad \text{for}\ k\to +\infty.$$
This completes the proof of the lemma.
\end{proof}

\begin{prop}
\label{prop:abbo}
Fix $s\in (1/2,1)$. Then, for every $r>1/2$ there exists $c_r>0$ such that 
$$\Big \|[A,\Delta^{s-1}\circ \frac{\diff}{\diff t}] u \Big \|_{1-s}\leq c_r \|A\|_s \|u\|_r$$
for every $A\in H^s(\T, \mathrm{Hom}(\R^n,\R^n))$, and $u\in H^r(\T,\R^n)$. 
\end{prop}
\begin{proof}
By a standard density argument we may assume that both $u$ and $A$ are smooth. We write 
$$A= \sum_{k\in \Z} A_k e_k, \qquad u = \sum_{h\in\Z} u_he_h,$$
for the Fourier series expansion of $A$ and $u$ respectively, where $\{e_h\}$ denotes the standard Hilbert basis of $L^2$. A straightforward computation shows that 
$$[A,\Delta^{s-1}\circ \frac{\diff}{\diff t}] u = (2\pi)^{2s-1} i \sum_{k\in\Z} \Big ( \sum_{h\in\Z} \big ( |h|^{2(s-1)}h - |k|^{2(s-1)}k \big ) A_{k-h}u_h \Big) e_k.$$
Therefore, using the Cauchy-Schwarz inequality and Lemma~\ref{lem:abbo2} we obtain for some positive number $c$:
\begin{align*}
\Big \|[A,\Delta^{s-1}\circ \frac{\diff}{\diff t}] u \Big \|_{1-s}^2 & = (2\pi)^{2(2s-1)} \sum_{k\in\Z} (1+|k|)^{2(1-s)} \Big |  \sum_{h\in\Z} \big ( |h|^{2(s-1)}h - |k|^{2(s-1)}k \big ) A_{k-h}u_h\Big |^2\\
				&\leq (2\pi)^{2(2s-1)} \sum_{k\in\Z} (1+|k|)^{2(1-s)} F(k) \sum_{h\in\Z} (1+|k-h|)^{2s} |A_{k-h}|^2 (1+|h|)^{2r} |u_h|^2 \\
				&\leq c  \sum_{k\in\Z} \sum_{h\in\Z} (1+|k-h|)^{2s} |A_{k-h}|^2 (1+|h|)^{2r} |u_h|^2 \\
				&= c \sum_{h\in\Z} (1+|h|)^{2r} |u_h|^2 \sum_{k\in\Z} (1+|k-h|)^{2s} |A_{k-h}|^2\\
				&= c\|A\|_s^2 \|u\|_r^2
				\end{align*}
				as claimed. 
\end{proof}

\begin{rmk}
The proposition above implies that, for any fixed $A\in H^s(\T,\text{Hom}(\R^n,\R^n))$, the commutator 
$$ [A,\Delta^{s-1}\circ \frac{\diff}{\diff t}] :H^s(\T,\R^n)\to H^{1-s}(\T,\R^n)$$
is a compact operator, for it can be written as the composition of a bounded operator (the commutator itself defined on $H^{r},$ $r\in (1/2,s)$) with the 
compact embedding $H^s\hookrightarrow H^{r}$. Notice that the single operators 
$$A \Delta^{s-1} \frac{\diff}{\diff t}: H^s_0(\T,\R^n)\to H^{1-s}(\T,\R^n), \quad \Delta^{s-1} \frac{\diff}{\diff t} (A\cdot):H^s_0(\T,\R^n)\to H^{1-s}(\T,\R^n),$$
are only bounded. Also, it is not clear to us whether or not the commutator operator gain derivatives (unless $A$ is assumed to be sufficiently regular). In fact, one would expect from the theory of pseudo-differential operators that $[A,\Delta^{s-1}\frac{\diff}{\diff t}]$ be a pseudo-differential operator 
of order $2s-2$, and thus define a bounded operator from $H^s$ into $H^{2-s}$. This however will not be relevant for our purposes.

We shall finally observe the following consequence Proposition~\ref{prop:abbo} which will be useful later on: if $\{A_m\}\subset H^s(\T,\text{Hom}(\R^n,\R^n))$ is a bounded sequence 
and $\{u_m\}\subset H^s(\T,\R^n)$ weakly-converges in $H^s$ to $u\in H^s(\T,\R^n)$, then $\{[A_m,\Delta^{s-1} \frac{\diff}{\diff t}]u_m\}$ converges strongly in $H^{1-s}$ 
to $[A,\Delta^{s-1} \frac{\diff}{\diff t}]u$. \qed
\end{rmk}

\begin{rmk}
\label{rmk:easiercommutator}
Using the chain rule we see that
$$[A,\Delta^{s-1} \frac{\diff}{\diff t}] u = [A,\Delta^{s-1}] \dot u - \Delta^{s-1} \dot A \cdot u.$$
Therefore, combining Corollary~\ref{cor:multiplications-1} with Proposition~\ref{prop:abbo} we obtain that, for every $s\in (1/2,1)$ and every $A\in H^s(\T,\text{Hom}(\R^n,\R^n))$ fixed, also the commutator operator 
\begin{equation}
[A,\Delta^{s-1}] \circ \frac{\diff}{\diff t}:H^s(\T,\R^n)\to H^{1-s}(\T,\R^n)
\label{eq:easiercommutator}
\end{equation}
is well-defined and compact. Moreover, for any bounded sequence $\{A_m\}\subset H^s(\T,\text{Hom}(\R^n,\R^n))$ and any weakly-converging sequence $\{u_m\}\subset H^s(\T,\R^n)$, the sequence 
$\{[A,\Delta^{s-1}]\dot u_m\}$ is strongly converging in $H^{1-s}$. 
\end{rmk}
%%%%%%%%%%%%%%%%
%%%%%%%%%%%%%%%%
%%%%%%%%%%%%%%%%

\section{The (0)-essential subbundle $\mathcal E$}
\label{s:essentialsubbundle}

%%%%%%%%%%%%%%%%%%

In this section we show that there is a natural way to define a $(0)$-essential strongly integrable essential subbundle $\mathcal E^s\subset T\M^{1-s}$, 
for which the standard atlas of $\M^{1-s}$ (c.f. Subsection 2.1) is a strong integrable structure such that the negative eigenspace of the Hessian of $\A_H$ at each critical point $(\q,\p)$ of $\A_H$ 
belongs to the essential class $\mathcal E^s(\q,\p)$. In particular, we have a well-defined (integer valued) 
notion of relative Morse index for critical points of $\A_H$. In Section~\ref{s:precompactness} we will then show that $\mathcal E^s$ behaves ``nicely'' under a suitable pseudo-gradient flow of $\A_H$, thus allowing us to prove that the intersection between stable and unstable manifolds of any two critical points of $\A_H$ is finite dimensional and pre-compact. %As we shall see later, transversality issues are here much easier to overcome than in Floer homology, since the space of admissible deformations 
%of the negative gradient vector field is much larger. 

%We will perform the construction of such an essential subbundle in $T\M^{1-s}$. The corresponding $\mathcal E\subset T\M^{1-s,*}$ will be then obtained by identifying $H^{1-s}(\T,\R^n)$ with $H^{1-s}(\T,(\R^n)^*)$ in the usual way. 

For $s\in (1/2,1)$ fixed we set $\HH$ to be the Hilbert space 
$$\HH := H^s(\T,\R^n)\times H^{1-s}(\T,(\R^n)^*)$$
endowed with the standard Hilbert product
$$\langle \cdot, \cdot \rangle_\HH := \langle \text{pr}_1(\cdot),\text{pr}_1(\cdot)\rangle_s + \langle \text{pr}_2 (\cdot),\text{pr}_2(\cdot)\rangle_{1-s},$$
where $\text{pr}_1:\HH\to H^s(\T,\R^n)$ and $\text{pr}_2:\HH\to H^{1-s}(\T,(\R^n)^*)$ denote the projections onto the first and second factor respectively, 
and define the operator
\begin{equation}
L:\HH\to \HH, \quad L(\q,\p) := (T^*\Delta^{-s} \dot \p, -T\Delta^{s-1}\dot \q).
\label{Lflat}
\end{equation}
Here $\Delta$ denotes the Laplace-operator, $\Delta^\alpha$ is for every $\alpha\in \R$ defined by 
$$\Delta^\alpha u := \sum_{k\in \Z^*} (2\pi)^{2\alpha} |k|^{2\alpha} u_k e_k,$$
where $\{e_k\}$ is the standard Hilbert basis of $L^2(\T,\R^n)$ and $u=\sum_{k\in \Z} u_k e_k$ is the Fourier expansion of $u$, and
$T:\R^n\to (\R^n)^*$ denotes the canonical identification (notice that $T$ and $T^*$ commute with both $\Delta$ and $\diff /\diff t$). 
It is straightforward to check that $L$
is a bounded self-adjoint operator representing the quadratic form $\langle \p ,\dot \q\rangle$ with respect to the Hilbert product $\langle \cdot,\cdot\rangle_{\HH}$, i.e. 
$$\frac 12 \langle L(\q,\p),(\q,\p)\rangle_{\HH} = \langle \p ,\dot \q\rangle =\int_0^1 \p(t) [\dot \q(t)]\, \diff t, \qquad \forall (\q,\p)\in \HH.$$
Moreover, $L$ has $2n$-dimensional kernel $\HH^0$ given by constant loops and its eigenvalues are given by $\pm 1$ with corresponding eigenspaces 
\begin{align}
\HH^\pm &= \Big \{ (\q,\mp T \Delta^{s-1} \dot \q) \ \Big |\ \q \in H^s (\T,\R^n), \ \int_0^1 \q(t)\, \diff t =0 \Big \}.\label{eq:H-}
\end{align} 
In what follows, we denote with $H^s_0(\T,\R^n)\subset H^s(\T,\R^n)$ the subspace of loops with zero mean, and with $Q:\mathbb H\to \mathbb H$ the orthogonal projector onto $\HH^+$ (since $\HH^0$ is finite dimensional, there is no need to worry about it). 
For every open subset $U\subset \R^n$ and every diffeomorphism 
$$\T\times U \to \T \times M,\qquad (t,q)\mapsto (t,\varphi(t,q)),$$
we denote  the induced local parametrization by
$$\varphi_* : H^s(\T,U)\times H^{1-s}(\T, (\R^n)^*)\to \M^{1-s},$$
and denote by $\mathcal E_\varphi^s$ the local subbundle obtained by pushing forward  the constant subbundle of 
$$T \big (H^s(\T,U)\times H^{1-s}(\T,(\R^n)^*)\big ) \cong H^s(\T,U)\times H^{1-s}(\T,(\R^n)^*) \times \HH$$
given by taking in each tangent space the negative eigenspace $\HH^-:=\HH^-(L)$ of the operator $L$ under $\varphi_*$.
In this way, we obtain a family of local subbundles
\begin{equation}
\mathcal E^s := \{\mathcal E_\varphi^s\}.
\label{mathcalE}
\end{equation}

\begin{prop}
For every $s\in (1/2,1)$ fixed, $\mathcal E^s$ defines an integrable (0)-essential subbundle of $T\M^{1-s}$. Moreover, the negative eigenspace $\HH^u(\q,\p)$ of the Hessian 
$\diff^2 \A_H(\q,\p)$ of $\A_H:\M^{1-s}\to \R$ at a critical point $(\q,\p)$ belongs 
to the essential class $\mathcal E^s(\q,\p)$, meaning that $\HH^u(\q,\p)$ is a compact perturbation of $\mathcal E^s_\varphi$ for any local parametrization $\varphi_*$ such that $(\q,\p)\in \mathrm{Im}\, \varphi_*$. 
\label{prop:integrable}
\end{prop}

\begin{dfn}
\label{def:morseindex}
Let $(\q,\p)$ be a critical point of the Hamiltonian action functional $\A_H:\M^{1-s}\to \R$. Then, the \textit{relative Morse index} $\mathrm{m}(\q,\p)$ is defined as 
$$\mathrm{m}(\q,\p):= \mathrm{m}\big ((\q,\p),\mathcal E^s):= \dim (\HH^u(\q,\p), \mathcal E_\varphi^s),$$ 
that is, as the relative dimension between the negative eigenspace $\HH^u(\q,\p)$ of the Hessian of $\A_H$ at $(\q,\p)$ and $\mathcal E_\varphi^s$, where  $\varphi_*$ is any local parametrization 
of $\M^{1-s}$ such that $(\q,\p)\in \text{Im}\, \varphi_*$. By Proposition~\ref{prop:integrable}, $\mathrm{m}(\q,\p)$ is a well-defined integer, i.e. independent of the choice of $\varphi_*$. 
\end{dfn}

\begin{proof} 
We need to show that for every transition map $\tau_*$ we have:
\begin{itemize}
\item $\diff \tau_* (\HH^-)$ is compact perturbation of $\HH^-$.
\item $\dim (\HH^-, \diff \tau_* (\HH^-))=0$. 
\end{itemize}
A straightforward computation shows that 
$$\diff \tau_* (\q,\p)[\hh,\kk] = ( A^{-1}(t) \hh, B(t) \hh + A(t)^* \kk),$$
where 
\begin{align*}
A^{-1}(t) := A^{-1}(t ,\q(t)) &:= \diff \tau(t ,\q(t)) \in H^s (\T, GL(\R^n)),\\
B(t) := B(t,\q,\p) &:= (\diff (\diff \tau^{-*}(t,\q(t)) \cdot )\p(t) \in H^{1-s}(\T, \text{Hom}(\R^n,(\R^n)^*)).
\end{align*}
Therefore
\begin{align*}
\diff \tau_*(\q,\p)[\HH^-] &= \Big \{ \big (A^{-1} \hh, B\hh + A^{*} T \Delta^{s-1}\dot \hh \big ) \Big |\ \hh \in H^s_0(\T,\R^n)\Big \}\\
				&=\Big \{ \big ( \hh, BA\hh + TA^{-1} \Delta^{s-1}\dot{(A\hh)} \big ) \Big |\ \hh \in H^s_0(\T,\R^n)\Big \}
				\end{align*}
where we used the fact that $A^{-1}(t)\in GL(\R^n)$ for every $t\in \T$ and the identity $TA^{-1}=A^*T$. Also, for notational convenience, we dropped the dependence on $t$ everywhere. 

It follows that both $\HH^-$ and $\diff \tau_*(\q,\p)[\HH^-]$ are graph of a function $H^s_0(\T,\R^n)\to H^{1-s}(\T,(\R^n)^*)$, namely 
$$\hh \mapsto T\Delta^{s-1}\dot \hh \qquad \text{and}\qquad \hh\mapsto BA \hh + TA^{-1}\Delta^{s-1}\dot{(A\hh)}$$
respectively. As such, they are compact perturbation of each other if and only if the corresponding maps differ by a compact operator (add reference). The part $\hh \mapsto BA \hh$ is compact by Proposition~\ref{prop:multiplicationSobolev}, as for all $r\in (1/2,s]$ it factorizes as 
$$H^s_0 \hookrightarrow H^{r} \stackrel{BA}{\longrightarrow} H^{1-s}.$$ 
As the operator $T$ does not play any role as far as compactness is concerned, we are thus left with the operator 
$$\Delta^{s-1} - A^{-1}\Delta^{s-1}\dot{(A\hh)}$$
which after multiplying on both sides with $A$ yields the operator
$$[A,\Delta^{s-1} \circ \frac{\diff}{\diff t}] : H^s_0(\T,\R^n) \mapsto H^{s-1}(\T,\R^n).$$
The claim follows now from Proposition~\ref{prop:abbo}. As far as $(0)$-essentiality is concerned, we observe that both $\HH^-$ and $\diff \tau_*(\q,\p)[\HH^-]$ are in Fredholm pair with the vertical subspace 
$$\{0\}\times H^{1-s}(\T,\R^n)\subset 
H^s_0(\T,\R^n)\times H^{1-s}(\T,(\R^n)$$ and the Fredholm index is in both cases zero (this follows from the general fact that the graph of a function is always in Fredholm pair with the vertical subspace with Fredholm index zero). Therefore, applying Proposition~5.1 in \cite{AM:03} we obtain
$$\text{ind}\, (\HH^-, \{0\}\times H^{1-s}(\T,(\R^n)^*)) = \text{ind}\, (\diff \tau_*(\q,\p)[\HH^-], \{0\}\times H^{1-s}(\T,(\R^n)^*)) + \dim (\HH^-, \diff \tau_*(\q,\p)[\HH^-])$$
which implies that $\dim (\HH^-, \diff \tau_*(\q,\p)[\HH^-])=0$ as claimed. 

In order to prove the last assertion, we take any local parametrization $\varphi_*$ of $\M^{1-s}$ such that $\text{Im}\, \varphi_*$ contains the critical point $(\q,\p)$ of $\A_H$ and compute
\begin{align*}
\A_H \circ \varphi_* (\xi,\nu) &= \int_0^1 \big (\diff \varphi_t(\xi(t)\big )^{-*}[\nu(t)] \big (\partial_t \varphi_t(\xi(t)) + \diff \varphi_t(\xi(t))[\dot \xi (t)]\big )\, \diff t + \mathcal H(\varphi_\cdot(\xi),\diff \varphi_\cdot(\xi)^{-*}[\nu])\\
				&= \underbrace{\int_0^1 \nu(t) [\dot \xi(t)]\, \diff t}_{=: (1)} +\underbrace{\int_0^1 \nu(t) \big [\diff \varphi_t(\xi(t))^{-1}[\partial_t \varphi_t(\xi(t))]\big ]\, \diff t +\mathcal H(\varphi_\cdot(\xi),\diff \varphi_\cdot(\xi)^{-*}[\nu])}_{=:(2)}.
\end{align*}
Since the functional in (2) has compact differential (indeed, no derivatives of $\xi$ or $\nu$ appear), and since the negative eigenspace of the Hessian of the functional in (1) is precisely $\HH^-$, the claim follows. 
\end{proof}

As already observed, for our purposes we need to know more about $\mathcal E^s$ than integrability. What we need is indeed that $\mathcal E^s$ is strongly integrable. 
Following \cite[Page 335]{AM:05}, in order to show this, we
need to establish the following property: for any transition map $\tau_*=\psi_*^{-1} \circ \varphi_*$ and every $V\subset \varphi_*^{-1}(\text{Im}\, \varphi_* \cap \text{Im}\, \psi_*)$ bounded, $QV$ is pre-compact if and 
only if $Q \tau_*(V)$ is pre-compact. Clearly, if suffices to show the implication $QV$ pre-compact $\Rightarrow \ Q\tau_*(V)$ pre-compact, since the other implication follows replacing $\tau_*$ with $\tau_*^{-1}$. 
Notice that transition maps are of the form 
$$\tau_*(\q,\p) = \big ( \tau(\cdot, \q(\cdot)), \diff \tau (\cdot,\q(\cdot))^{-*} \p\big),\qquad \tau=\psi^{-1}\circ \varphi.$$
Thus, fix $\tau_*$ and consider a bounded set $V\subset \varphi_*^{-1}(\text{Im}\, \varphi_* \cap \text{Im}\, \psi_*)$ such that $QV$ is pre-compact. %This means that for any sequence $\{u_m\}\subset V$ we have that $\{u_m^0\}$ and $\{u_m^+\}$ are strongly converging to some $u_\infty^0\in \HH^0$ and $u^+_{\infty}\in \HH^+$ respectively, where 
%$$u_m = u_m^- + u_m^0 + u_m^+$$
%is the usual orthogonal decomposition of $u_m$ with respect to the orthogonal splitting $\HH = \HH^- \oplus \HH^0\oplus \HH^+$. 

\begin{lem}
\label{lem:strong1}
The following statements are equivalent:
    \begin{enumerate}[i)]
\item $Q\tau_*(V)$ is pre-compact.
\item For every sequences $\{u_m\}\subset V$ and $\{v_m\}\subset \HH^+$ such that $v_m\rightharpoonup 0$ we have $\langle \tau_*(u_m),v_m \rangle_\HH \to 0.$
\end{enumerate}
\end{lem}

\begin{proof}
Assume that $Q\tau_*(V)$ is pre-compact, and let $(u_m)\subset V$ and $(v_m)\subset \HH^+$ be sequences such that $v_m\rightharpoonup 0$. Then, up to taking 
a subsequence we have that $Q\tau_*(u_m)\to u$ with $u\in \HH^+$, and hence
\begin{align*}
\langle \tau_*(u_m),v_m\rangle_\HH &= \langle Q\tau_* (u_m),v_m\rangle_\HH \\
 						&= \langle Q\tau_* (u_m)-u,v_m\rangle_\HH+ \langle u,v_m\rangle_\HH\\
						&\leq \|Q \tau_*(u_m)-u\|_\HH \cdot \|v_m\|_\HH + \langle u,v_m\rangle_\HH\\
						&\to 0.\end{align*}
Conversely, assume that ii) holds. In order to show that $Q\tau_* (V)$ is pre-compact, we need to show that for every $(u_m)\subset V$ the sequence $Q\tau_*(u_m)$ contains a strongly converging 
subsequence. Thus, let $(u_m)\subset V$  be any sequence. Since $Q\tau_*(u_m)$ is bounded, we have that up to a subsequence $Q\tau_* (u_m) \rightharpoonup u$ for some $u\in \HH^+$. Therefore, 
choosing $v_m:= Q\tau_* (u_m)-u\rightharpoonup 0$ we obtain 
$$o(1) = \langle \tau_*(u_m), Q\tau_*(u_m) - u \rangle_\HH = \|Q\tau_*(u_m)\|^2_\HH - \underbrace{\langle \tau_*(u_m),u\rangle_\HH}_{\to \|u\|^2_\HH},$$
from which we deduce that $\|Q\tau_*(u_m)\|_\HH\to \|u\|_\HH.$ This readily shows that $Q\tau_*(u_m)\to u$.
\end{proof}

Thus, let $\{u_m\}\subset V$ any sequence. Up to passing to a subsequence we have 
$$u_m = (\q_m,\p_m)= (\q_m^-,T\Delta^{s-1}\dot \q_m^-) + (\q_m^+,-T\Delta^{s-1}\dot \q_m^+) + (\q_m^0,\p_m^0)\in \HH^-\oplus \HH^+\oplus \HH^0$$
with $\{\q_m^-\}$ weakly-converging in $H^s$, $\{\q_m^+\}$ strongly converging in $H^s$, and $\{(\q_m^0,\p_m^0)\}$ strongly converging in $\HH^0$. 
For any fixed sequence $\{v_m\}\subset \HH^+$ with $v_m\rightharpoonup 0$ we want to show that 
$$\langle \tau_*(u_m),v_m\rangle_{\HH} \to 0.$$

We notice that $\{\Delta^{s-1}\dot \q_m^+\}$ and $\{\diff \tau(\q_m)\}$ are strongly converging in $H^{1-s}$ by assumption, and that $\{[\diff \tau(\q_m),\Delta^{s-1}]\dot \q_m\}$ is strongly converging in $H^{1-s}$ by Remark~\ref{rmk:easiercommutator}. Further, we write 
$$\tau(\q_m) = \tilde \tau(\q_m) + \tau(\q_m)^0,$$
where $\tau(\q_m)^0:= \int_0^1 \tau(\q_m)\, \diff t$, and observe that $\tau(\q_m)^0$ converges in $\R^n$. Therefore, 
denoting with ``$\cong$'' any equality up to quantities which are strongly converging in $\HH$ (recall that this means strong convergence in $H^s$ resp. $H^{1-s}$ when the corresponding quantity appears 
in the first resp. second factor), we compute
\begin{align*}
\tau_*(u_m) &= ( \tau(\q_m), T \diff \tau(\q_m)[\Delta^{s-1}\dot \q_m^- - \Delta^{s-1}\dot \q_m^+] + \diff \tau(\q_m)^{-*}\p_m^0)\\
		&= ( \tau(\q_m), T \diff \tau(\q_m)[\Delta^{s-1}\dot \q_m] ) - 2\cdot  (0,T\diff \tau(\q_m)[\Delta^{s-1}\dot \q_m^+]) + (0,\diff \tau(\q_m)^{-*}\p_m^0)\\
		&\cong ( \tau(\q_m), T \diff \tau(\q_m)[\Delta^{s-1}\dot \q_m] )\\
		&= ( \tilde \tau(\q_m), T \diff \tau(\q_m)[\Delta^{s-1}\dot \q_m] ) + (\tau(\q_m)^0,0)\\
		&\cong ( \tilde \tau(\q_m), T \diff \tau(\q_m)[\Delta^{s-1}\dot \q_m] )\\
		&=( \tilde \tau(\q_m), T \Delta^{s-1} \diff \tau(\q_m)[\dot \q_m]) + (0, T [\diff \tau(\q_m),\Delta^{s-1}]\dot \q_m]) \\
		&\cong ( \tilde \tau(\q_m), T \Delta^{s-1} \diff \tau(\q_m)[\dot \q_m])\\
		&= ( \tilde \tau(\q_m), T \Delta^{s-1} \diff/\diff t [\tilde \tau(\q_m)]) \in \HH^-.
\end{align*}
In other words, $\tau_*(u_m)$ differs from an element in $\HH^-$ only up to a quantity which strongly converges in $\HH$, and this yields 
$$\langle \tau_* (u_m),v_m\rangle_\HH = o(1) +\underbrace{\langle ( \tilde \tau(\q_m), T \Delta^{s-1} \diff/\diff t [\tilde \tau(\q_m)]), v_m\rangle_\HH}_{=0} = o(1)$$
as claimed. Summarizing, we have proved the following
\begin{thm}[Strong integrability]
For every $s\in (1/2,1)$, $\mathcal E^s$ defines a strongly integrable (0)-essential subbundle of $T\M^{1-s}$ with the property that the negative eigenspace $\HH^u(\q,\p)$ of the Hessian 
of $\A_H:\M^{1-s}\to \R$ at every critical point $(\q,\p)$ belongs to the essential class $\mathcal E^s(\q,\p)$. \qed
\label{thm:strong2}
\end{thm}

\section{Pre-compactness of intersection between stable and unstable manifold}  
\label{s:precompactness}

%In the previous section we constructed a strongly integrable (0)-essential subbundle $\mathcal E\subset T\M^{1-s}$ 
%with the property that the negative eigenspace of the Hessian of $\A_H$ at every critical point belongs to the essential class determined by $\mathcal E$. 

Goal of this section will be to construct a negative pseudo-gradient $F$ for $\A_H$ such that the pair $(\A_H,F)$ satisfies the Palais-Smale condition and 
which preserves essentially vertical sets for the (0)-essential subbundle $\mathcal E^s$ constructed in Section~\ref{s:essentialsubbundle}. In virtue of \cite[Theorem 6.5]{AM:05}, this implies that the intersection between the stable and unstable manifold 
of any two critical points of $\A_H$ is a pre-compact finite dimensional manifold of dimension equal the difference of the relative Morse indices. 
To achieve this, we will first prove in Section~\ref{s:abstractvertical} an abstract result stating that the space of vector fields that preserve an essential vertical family 
 is a module over the space of bounded Lipschitz functions. In Section~\ref{s:pseudogradient} we will then employ this abstract result to construct $F$
first in a fixed local chart and then glue all local definitions by means of cut-off functions. This construction will be carried out in such a way that the Palais-Smale condition for 
the pair $(\A_H,F)$ follows from the ``metric'' Palais-Smale condition, see Proposition~\ref{lem:compactness}.
The vector field $F$ will be then perturbed in Section~\ref{s:transversality} within the class of vector fields preserving essentially vertical sets in order to achieve transversality.

\subsection{The space of vector fields that preserve an essentially vertical family}
\label{s:abstractvertical}

In this section we will prove an abstract result about the space of vector fields whose flow preserves a given essentially vertical family. Thus, let $\M$ be a Hilbert manifold
admitting a close embedding into a Hilbert space. We endow $\M$ with the induced distance. We further assume that 
$\mathcal F$ be an essentially vertical family for the strongly integrable essential subbundle $\mathcal E\subset T\M$. 
For a bounded Lipschitz vector field $F$ on $\M$, the essentially vertical family $\mathcal F$ is said to be $F$-\textit{positively 
invariant} if for every $A\in \mathcal F$ and every $T\geq 0$ the set $\phi^F([-T,T]\times A)$ is in $\mathcal F$. 

\begin{prop}
The space of bounded Lipschitz vector fields $F$ on $\M$ of class $C^k$ for which $\mathcal F$ is $F$-positively invariant is a module over the ring of bounded Lipschitz functions of class $C^k$ on $\M$. 
\label{prop:module}
\end{prop}

For the proof we will need the following homogenization result, which is certainly known to the experts, but which will be included here for completeness. 

\begin{thm}
Let $E$ be a real Banach space and let $X,Y$ be globally Lipschitz vector fields on $E$. Set, for $n\in \Z_+$, 
$$J_n := \bigcup_{k\in \Z} \Big [ \frac{2k}{n}, \frac{2k+1}{n}\Big ), \qquad Z_n(t,x) := \mathbb I_{J_n} (t) X(x) + \mathbb I_{J_n^c}(t) Y(x).$$
Then the flow $\phi_n:\R\times E\to E$ of $Z_n$ converges to the flow $\phi:\R\times E \to E$ of $((X+Y)/2$ uniformly on $[-T,T]\times E$, for every $T>0$.
\label{thm:homogenization}
\end{thm}

To prove Theorem~\ref{thm:homogenization} we will use the following

\begin{lem}
Let $\mathfrak U\subset C^0([a,b],E)$ be an equi-bounded and equi-continuous family of $E$-valued curves. Let $(\mu_n)$ be a sequence of real valued Radon measures on $[a,b]$ 
which converges to zero in the weak$\text{-}^*$topology of $C^0([a,b],E)^*$. Then the convergence 
$$\lim_{n\to +\infty} \int_{[a,b]} u \, \diff \mu_n = 0 \quad \text{in}\ E$$
is uniform in $u\in \mathfrak U$.  
\label{lem:homogenization}
\end{lem}

\begin{proof}
The statement would be trivially true if $E$ were finite dimensional, because in this case the family $\mathfrak U$ would be pre-compact in $C^0([a,b],E)$. The general case can be 
reduced to the scalar case by choosing $\varphi_{n,u}$ in $E^*$ such that $\|\varphi_{n,u}\|\leq 1$ and 
$$\langle \varphi_{n,u}, \int_{[a,b]} u\, \diff \mu_n \rangle = \left \| \int_{[a,b]} u\, \diff \mu_n \right \|.$$
Indeed, with such a choice the family 
$$\left \{ \langle \varphi_{n,u},u\rangle \right \}_{n\in \N, \ u\in\mathfrak U} \subset C^0([a,b],\R)$$
is equi-bounded and equi-continuous, thus pre-compact by the Arzela-Ascoli theorem. Hence, 
$$\lim_{n\to +\infty} \int_{[a,b]} \langle \varphi_{n,u},u\rangle \, \diff \mu_n =0,$$
uniformly in $u\in\mathfrak U$. The conclusion follows from the identity 
\begin{equation*}
 \left \| \int_{[a,b]} u\, \diff \mu_n \right \| = \langle \varphi_{n,u}, \int_{[a,b]} u\, \diff \mu_n \rangle = \int_{[a,b]}  \langle \varphi_{n,u},u\rangle \, \diff \mu_n. \qedhere
 \end{equation*}
\end{proof}

\begin{proof}[Proof of Theorem~\ref{thm:homogenization}]
Observe that, if $X,Y$ are $\kappa$-Lipschitz, then $Z_n(t,\cdot)$ is $\kappa$-Lipschitz with respect to the second variable for every $t\in \R$. 

\vspace{2mm}

\textbf{Claim 1.} $\phi_n(t,\cdot)$ is equi-Lipschitz with respect to the second variable, uniformly in $t\in [-T,T]$. 

\vspace{2mm}

\noindent In fact, since $Z_n$ is $\kappa$-Lipschitz with respect to the second variable,
\begin{align*}
\|\phi_n(t,x)-\phi_n(t,y)\| &\leq \Big \| x-y +\int_0^t \big (Z_n(s,\phi(s,x)- Z_n(s,\phi(s,y))\big )\, \diff s \Big \|\\
				&\leq \|x-y\| + \kappa \, \Big |\int_0^t \|\phi(s,x)-\phi(s,y)\| \,\diff s \Big |,
\end{align*}
and by the Gronwall's lemma we find 
$$\|\phi_n(t,x)-\phi_n(t,y)\|\leq e^{\kappa |t|} \|x-y\|,$$
for every $t\in \R$ and every $x,y\in E$. 

\vspace{2mm}

\textbf{Claim 2.} If the family of curves $\mathfrak U\subset C^0([a,b],E)$ is equi-bounded and equi-continuous, then 
$$\lim_{n\to +\infty} \int_a^b \Big (Z_n(s,u(s))- \frac{X+Y}{2}(u(s))\Big ) \, \diff s =0\quad \text{in}\ E,$$

uniformly in $u\in\mathfrak U$. 

\vspace{2mm}

Indeed, we have the identity 
$$ \int_a^b \Big (Z_n(s,u(s))- \frac{X+Y}{2}(u(s))\Big ) \, \diff s = \int_a^b \Big (\mathbb I_{J_n}(s) - \frac 12 \Big ) (X-Y) (u(s)) \, \diff s.$$
By Claim 1, the family of curves 
$$\{(X-Y)\circ u \}_{u\in\mathfrak U}\subset C^0([a,b],E)$$
is equi-bounded and equi-continuous. Moreover, $\mathbb I_{J_n}$ converges to $\frac 12$ in the weak$\text{-}^*$topology of $C^0([a,b],E)^*$. Then the conclusion follows from Lemma~\ref{lem:homogenization}.

\vspace{2mm}

\textbf{Claim 3.} For every $(t,x)\in \R\times E$, the sequence $(\phi_n(t,x))$ is a Cauchy sequence. 

\vspace{2mm}

\noindent By Claim 2 we have 
\begin{align*}
\|\phi_n(t,x)-\phi_m(t,x)\| & =\left \| \int_0^t \big (Z_n(s,\phi(s,x)- Z_n(s,\phi(s,y))\big )\, \diff s \right \|\\
				&= \left \| \int_0^t \Big ( \frac{X+Y}{2} (\phi_n(s,x)) - \frac{X+Y}{2}(\phi_m(s,x))\Big ) \, \diff s \right \| + \epsilon_{n,m}(t),
\end{align*}
where $\epsilon_{n,m}(t)$ is infinitesimal for $m,n\to +\infty$, uniformly in $t\in [-T,T]$. Since $(X+Y)/2$ is $\kappa$-Lipschitz, we deduce that 
$$\|\phi_n(t,x)-\phi_m(t,x)\| \leq \kappa \, \int_0^t \|\phi_n(s,x)-\phi_m(s,x)\| \,\diff s + \epsilon_{n,m}(t).$$
By Gronwall's lemma, $\|\phi_n(t,x)-\phi_m(t,x)\|\to 0$ for $n,m\to +\infty.$

We are now ready to complete the proof. By Claims 3 and 1, the sequence $\phi_n$ converges to some $\phi\in C^0(\R\times E,E)$ uniformly on $[-T,T]\times E$, for every $T>0$. Thus, it only remains to show that $\phi$ is the 
flow of $(X+Y)/2$. To this purpose fix $(t,x)\in E$. In the identity
$$\phi_n (t,x) = x + \int_0^t \Big (Z_n(s,\phi_n(s,x))- Z_n(s,\phi(s,x))\Big )\, \diff s + \int_0^t Z_n(s,\phi(s,x))\, \diff s,$$
the first integral converges to zero since $Z_n$ is $\kappa$-Lipschitz in the second variable and $\phi_n(\cdot,x)\to \phi(\cdot,x)$ uniformly on bounded intervals, whereas the second integral converges to 
$$\int_0^t \frac{X+Y}{2} (\phi(s,x))\, \diff s$$
by Claim 2. Therefore, taking the limit for $n\to +\infty$ yields 
$$\phi(t,x) = x+ \int_0^t \frac{X+Y}{2}(\phi(s,x))\, \diff s,$$
which says that $\phi$ is the flow of $(X+Y)/2$. 
\end{proof}

\begin{proof}[Proof of Proposition~\ref{prop:module}]
Let $X$ and $Y$ be bounded Lipschitz vector fields of class $C^k$ for which $\mathcal F$ is positively invariant. If $f\in C^k(\M)$ is Lipschitz and bounded, then $fX$ is Lipschitz and 
$$\phi^{fX} ([-T,T]\times A )\subset \phi^X ( [-\|f\|_\infty T, \|f\|_\infty T]\times A).$$
It follows that $\mathcal F$ is $fX$-positively invariant. By Theorem~\ref{thm:homogenization}, the flow $\phi^{X+Y}$ is the limit of the flow of the non-autonomous flow $\phi^{2Z_n}$, where 
$$Z_n(t,x) := \mathbb I_{J_n}(t) X(x) + \mathbb I_{J_n^c} Y(x), \quad J_n =\bigcup_{k\in\Z} \Big [\frac{2k}{n},\frac{2k+1}{n}\Big ),$$
uniformly on $[-T,T]\times \M$. By the stability property of $\mathcal F$ (see Definition 2.3), it is enough to check that if $A\in\mathcal F$ then also the set 
$$\phi^{Z_n}([-T,T]\times A)$$
belongs to $\mathcal F$, or, equivalently ($\mathcal F$ is an ideal), that the sets 
$$\phi^{Z_n}([0,T]\times A), \quad \text{and}\quad \phi^{Z_n}([-T,0]\times A)$$
belong to $\mathcal F$. We consider the first of these two sets, the argument for the second one being analogous. Since 
$$\phi^{Z_n} \big ([0,\frac 1n]\times A\big ) = \phi^X \big ([0,\frac 1n]\times A\big )$$
belongs to $\mathcal F$, and 
\begin{align*}
\phi^{Z_n} \big ([0,\frac{2k}n]\times A\big ) &\subset \phi^{Y} \Big ([0,\frac 1n]\times \phi^{Z_n} \big ([0,\frac{2k-1}{n}]\times A\big ) \Big ),\\
\phi^{Z_n} \big ([0,\frac{2k+1}n]\times A\big ) &\subset \phi^{X} \Big ([0,\frac 1n]\times \phi^{Z_n} \big ([0,\frac{2k}{n}]\times A\big ) \Big ),
\end{align*}
for every $k\in \Z_+$, an induction argument shows that 
$$\phi^{Z_n}\big ([0,\frac{k}{n}]\times A \big)$$
belongs to $\mathcal F$ for every $k\in\N$. The thesis follows taking $k\geq nT$. 
\end{proof}

%%%%%%%%%%%%%%%%%%%

\subsection{Proof of precompactness}
\label{s:pseudogradient}

In this subsection we use Proposition~\ref{prop:module} to construct a pseudo-gradient vector field $F$ for $\A_H$ which will enable us to prove pre-compactness of intersection between stable and unstable 
manifolds of critical points of $\A_H$ by employing the general compactness result \cite[Theorem 6.5]{AM:05}. To this purpose, the first step is to define what we mean by essentially vertical sets in $\M^{1-s}$.  

%we need to find a suitable family $\mathcal F$  of \textit{essentially vertical sets} which is positively invariant under the negative 
%gradient flow, meaning that for every $A\in \mathcal F$ and every $t\geq 0$ the set $\phi ([0,t]\times A)$ belongs to $\mathcal F$. 

\begin{dfn}
A bounded set $A\subset \mathcal M^{1-s}$ is called \textit{essentially vertical} if, for every local chart $(\varphi_*^{-1},\text{Im}\, \varphi_*)$,
$$Q \varphi_*^{-1} (A\cap \text{Im}\, \varphi_*)$$
is pre-compact in $\HH$, where $Q:\HH\to \HH$ is any projector with kernel the negative eigenspace $\HH^-$ of the operator $L$ defined in \eqref{Lflat}. 
We define
\begin{equation}
\mathcal F := \{A\subset \M^{1-s} \ |\ A \ \text{essentially vertical}\}
\label{def:essentiallyvertical}
\end{equation}
to be the family of all essentially vertical sets. 
\end{dfn}

Theorem \ref{thm:strong2} implies that the definition of essentially vertical sets is well-posed. Also, since every bounded set in $\M^{1-s}$ can be covered by 
finitely many domains of local charts, and in any local chart the condition for a subset to be essentially vertical is given by the pre-compactness of the projection of such a set onto a suitable subspace, 
we readily see that $\mathcal F$ is an essentially vertical family in the stronger sense of Definition~\ref{dfn:evf}, namely that $\mathcal F$ is closed with respect to the Hausdorff distance.

We now proceed to define the pseudo-gradient vector field $F$. We start observing that the metric on $\M^{1-s}$ defined in~\eqref{metricms} is, on every bounded set of $\M^{1-s}$, equivalent to 
the metric induced by embedding $M$ isometrically in $\R^{N}$, see Lemma 2.5 in \cite{Asselle:2020b}. Even though the two metrics might fail to be globally equivalent (see Appendix B of \cite{Asselle:2020b}),
the local equivalence is enough to apply Proposition~\ref{prop:module} on every bounded set of $\M^{1-s}$, and actually on every set of the form $\pi_{1-s}^{-1}(A)$, $A\subset H^s(\T,M)$ bounded. Thus, we consider 
a sequence $\{\varphi_{*,\ell}\}_{\ell\in \N}$ of local parametrizations such that 
$$\M^{1-s,*} = \bigcup_{\ell \in \N} \text{Im}\, \varphi_{*,\ell}.$$
Denoting with $\|\cdot\|$ and $\|\cdot\|_{\mathrm{flat}}$ the norm on $H^s(\T,U)\times H^{1-s}(\T,(\R^n)^*)$ induced by pull-back of the metric on $\M^{1-s}$ and the norm induced by the flat metric on $\R^n$ respectively, and with $\nabla$ and $\nabla^{\mathrm{flat}}$ the respective gradients, we see that the equivalence of the two metrics
implies that for every $\ell \in \N$ we can find a constant $\alpha(\ell)>0$ such that 
$$\alpha(\ell) \|\nabla ( \A_H\circ \varphi_{*,\ell}\|^2 \leq \|\nabla^{\mathrm{flat}} (\A_H\circ \varphi_{*,\ell})\|^2_{\mathrm{flat}} \qquad \text{on}\ \ H^s(\T,U)\times H^{1-s}(\T,(\R^n)^*).$$
This suggests to define, for every $\ell \in \N$, the vector field $F_\ell$ on $\text{Im}\, \varphi_{*,\ell}$ by pushing forward the vector field 
\begin{equation}
- \frac{1}{\alpha(\ell)} \nabla^{\mathrm{flat}} (\A_H\circ \varphi_{*,\ell})
\label{eq:localvectorfield}
\end{equation}
and then a vector field $F$ on $\M^{1-s}$ by gluing together the family of vector fields $\{F_\ell\}$ using a partition of unity $\{\chi_\ell\}$ subordinated to the open covering $\{\text{Im}\, \varphi_{*,\ell}\}$. 

A straightforward computation shows that the vector field in~\eqref{eq:localvectorfield} is a compact perturbation of $\alpha(\ell)^{-1} L(\cdot)$, where $L$ is the operator defined in \eqref{Lflat}. Hence, applying
Proposition A.18  in Appendix A of \cite{Rabinowitz:1986}, we see that $\chi_\ell \cdot F_\ell$ preserves the essentially vertical family $\mathcal F$ for every $\ell\in \N$. From Proposition~\ref{prop:module}
we deduce that 
$$\sum_{\ell=1}^L \chi_\ell \cdot F_\ell$$
also preserves the essentially vertical family $\mathcal F$ for every $L\in\N$, thus also $F$ as one sees by passing to the limit $L\to +\infty$. As by construction $\A_H$ is clearly a Lyapounov function 
for $F$, the last piece of information we need in order to be able to apply \cite[Theorem 6.5]{AM:05} is that the pair $(\A_H,F)$ satisfy the Palais-Smale condition. This however readily follows by construction of $F$: Indeed, if $(\q_m,\p_m)$ is a Palais-Smale sequence 
for the pair $(\A_H,F)$, then noticing that in the series below all but finitely many terms vanish we compute
\begin{align*}
\diff \A_H (\q_m,\p_m) [F(\q_m,\p_m)]&= \diff \A_H (\q_m,\p_m) \Big [ \sum_{\ell=1}^{+\infty} \chi_\ell\cdot F_\ell (\q_m,\p_m)\Big ]\\
						&= - \sum_{\ell=1}^{+\infty}\frac{\chi_\ell(\q_m,\p_m)}{\alpha(\ell)}\cdot \diff (\A_H\circ \varphi_{*,\ell}) (\varphi_{*,\ell}^{-1}(\q_m,\p_m)) \big [\nabla^{\mathrm{flat}}(\A_H\circ \varphi_{*,\ell})(\varphi_{*,\ell}^{-1}(\q_m,\p_m))\big ]\\
						&=-\sum_{\ell=1}^{+\infty}\frac{\chi_\ell(\q_m,\p_m)}{\alpha(\ell)}\cdot \|\nabla^{\mathrm{flat}} (\A_H\circ \varphi_{*,\ell})(\varphi_{*,\ell}^{-1}(\q_m,\p_m))\|^2_{\mathrm{flat}}\\
						&\leq -\sum_{\ell=1}^{+\infty} \chi_\ell(\q_m,\p_m) \cdot \|\nabla ( \A_H\circ \varphi_{*,\ell})(\varphi_{*,\ell}^{-1}(\q_m,\p_m))\|^2\\
						&= -\sum_{\ell=1}^{+\infty} \chi_\ell(\q_n,\p_n) \cdot \|\nabla \A_H (\q_m,\p_m)\|^2\\
						&= - \|\nabla \A_H (\q_m,\p_m)\|^2,
\end{align*}
which implies that $(\q_m,\p_m)$ is also a Palais-Smale sequence for the pair $(\A_H,-\nabla \A_H)$, thus pre-compact by Proposition~\ref{lem:compactness}. Conversely, any Palais-Smale sequence for 
$(\A_H,-\nabla \A_H)$ is also a Palais-Smale sequence for $(\A_H,F)$ as it is bounded and hence contained in the union of finitely many $\text{Im}\, \varphi_{*,\ell}$'s.

We are now in position to apply Theorem 6.5 in \cite{AM:05} to deduce the following

\begin{thm}
The essentially vertical family $\mathcal F$ is $F$-positively invariant, where $F$ is the negative pseudo gradient vector field for $\A_H:\M^{1-s,*}\to \R$ defined above. As a corollary, the intersection 
$$W^u((\q,\p);F) \cap W^s((\tilde \q,\tilde \p);F)$$
between the unstable and stable manifold of any pair of critical points $(\q,\p),(\tilde \q,\tilde \p)$ of $\A_H$ is pre-compact. \qed
\label{thm:positiveinvariance}
\end{thm}

%%%%%%%%%%%%%%%%
%%%%%%%%%%%%%%%%
%%%%%%%%%%%%%%%%

\section{Transversality}
\label{s:transversality}

In this section we prove an abstract transversality result for vector fields on a Hilbert manifold $\M$, and then apply it to show that, up to a small generic perturbation, we can assume that the stable and unstable 
manifolds of critical points of $\A_H$ intersect transversally. The proof of the abstract result follows the line of the one given in \cite{Abbondandolo:2006lk} in the case of finite Morse indices, however the fact that we are here dealing with infinite dimensional stable and unstable manifolds forces us to impose further conditions on the perturbation in order not to destroy all good properties of the given vector field.

%This will be achieved in such a way that all good properties discussed in Sections \ref{s:essentialsubbundle} and \ref{s:precompactness} will still be true for the perturbed (gradient-like) vector field. 

We start by recalling some definitions and basic facts about transversality in a Banach setting, referring to \cite{AR:67} for the details. We also recall that a subspace $Y$ of a topological space $X$ is said \textit{generic}
if it contains a countable intersection of open and dense subspaces of $X$. Baire theorem guarantees that a generic subspace of a complete metric space is dense. 

%%%%%%%%%%%%%%%%

\subsection{General facts about transversality} Let $\varphi:\BB\to \BBB$ be a map of class $C^k$ between $C^k$-Banach manifolds, $k\geq 1$. A point $v\in \BBB$ is called a \textit{regular value} for $\varphi$ 
if for every $p\in \varphi^{-1}(v)$ the differential $D\varphi(p):T_p\BB\to T_{v}\BBB$ is a left inverse, that is, if it is onto and its kernel is complemented. In this case, $\varphi^{-1}(v)$ is a submanifold of class $k$ of $\BB$.
The map $\varphi$ is said a \textit{Fredholm map} if its differential at every point is a Fredholm operator, i.e. a bounded operator with finite dimensional kernel and co-kernel. When the index of the differential is constant 
(for instance, when $\BB$ is connected), this integer is called the \textit{Fredholm index} of $\varphi$. 

The basic tool in order to deal with genericity questions is the Sard-Smale theorem \cite{Smale:1965}. In its original formulation, it applies to Fredholm maps whose domain is a Banach manifold which is a Lindel\"of space. 
 Here, we need the following generalization which is due to F. Quinn and A. Sard \cite[Theorem 1]{QS:72}, where the Lindel\"of property is replaced by the $\sigma$-properness of the map. We recall that a continuous map 
 $\varphi:\BB\to \BBB$ is \textit{proper} if the inverse image of every compact set is compact, and $\sigma$-\textit{proper} if $\BB$ is the countable union of open sets, on the closure of each of which $f$ is proper.
 
\begin{thm}
\label{thm:sard}
Let $\varphi:\BB\to \BBB$ be a $\sigma$-proper Fredholm map of class $C^k$ between $C^k$-Banach manifolds with Fredholm index $m$. If $k\geq \max \{1,m\}$, then the set of regular values of $\varphi$ is 
generic in $\BBB$. 
\end{thm}

In order to apply the Sard-Smale theorem, the following proposition will be useful (see also \cite{Abbondandolo:2006lk}):

\begin{prop}
\label{prop:generic1}
Let $\BB,\BBB,\BBBB$ be Banach manifolds, and let $\varphi:\BB \to \BBB$ and $\psi:\BB\to \BBBB$ be maps of class $C^1$ with regular values $p\in \BBB$ and $q\in \BBBB$. Then:
\begin{enumerate}
\item $p$ is a regular value for $\varphi |_{\psi^{-1}(q)}$ if and only $q$ is a regular value for $\psi |_{\varphi^{-1}(p)}$, and
\item $\varphi |_{\psi^{-1}(q)}$ is a Fredholm map if and only if $\psi |_{\varphi^{-1}(p)}$ is a Fredholm map. In this case, the Fredholm indices coincide.
\end{enumerate}
\end{prop}

The proposition is an immediate consequence of the following linear statements.

\begin{prop}
\label{prop:generic2}
Let $E,F,G$ be Banach spaces, and let $A\in \mathcal L(E,F)$, $B\in \mathcal L(E,G)$ be left inverses. Then: 
\begin{enumerate}
\item $A |_{\ker B}$ is a left inverse if and only if $B|_{\ker A}$ is a left inverse, and
\item $A|_{\ker B}$ is Fredholm if and only if $B|_{\ker A}$ is Fredholm. In this case, the Fredholm indices coincide. 
\end{enumerate}
\end{prop}

\begin{proof}
Let $R\in \mathcal L(F,E)$ and $S\in \mathcal L(G,E)$ be right inverses of $A,B$ respectively. 
\begin{enumerate}
\item Suppose that $R_0\in \mathcal L(F,\ker B)$ is a right inverse of $A|_{\ker B}$, that is, a right inverse of $A$ with range in $\ker B$. Then, the map $S_0:= (\text{Id}_E - R_0 A)S$ is a right inverse of $B$, 
being a perturbation of $S$ by an operator with range in $\ker B$, and it takes values in $\ker A$, for 
$$AS_0 = AS - AR_0AS = AS - \text{Id}_F AS =0.$$
Therefore, $S_0$ is a right inverse of $B|_{\ker A}$.
\item The kernels of $A|_{\ker B}$ and $B|_{\ker A}$ coincide and are equal to $\ker A \cap \ker B$. Moreover, since $R:F\to RF$ is an isomorphism and since $\text{Id}_E-RA$ is a projector onto $\ker A$, 
$$\text{coker} A|_{\ker B} = \frac{F}{A\ker B} \stackrel{R}{\cong} \frac{RF}{RA\ker B}\cong \frac{\ker A + RF}{\ker A + RA \ker B} = \frac{E}{\ker A + \ker B}.$$
From this we easily conclude that $A|_{\ker B}$ is Fredholm if and only if $B|_{\ker A}$ is Fredholm, since both assertions are equivalent to the fact that the pair $(\ker A,\ker B)$ is Fredholm, i.e. 
$\ker A\cap \ker B$ is finite dimensional and $\ker A + \ker B$ is finite codimensional. Moreover, the index of $A|_{\ker B}$ and $B|_{\ker A}$ equals the index of the Fredholm pair $(\ker A,\ker B)$. \qedhere
\end{enumerate}
\end{proof}

In the following lemma we single out a useful family of linear mappings whose kernel is complemented.

\begin{lem}
\label{lem:generic3}
Let $E,F,G$ be Banach spaces, and assume that $A\in \mathcal L(E,G)$ has complemented kernel and finite codimensional range. Then, for every $B\in \mathcal L(F,G)$ the kernel of the operator 
$$C\in \mathcal L(E\times F,G), \quad C(e,f) := Ae - Bf,$$
is complemented in $E\times F$. 
\end{lem}
\begin{proof}
Let $E_0:= \ker A$, $E_1\subset E$ be a closed complement of $E_0$, and $P_0,P_1$ be associated projectors. Let $G_1:= \text{ran} A$, $G_0\subset G$ be a (finite dimensional) complement of $G_1$,
and $Q_0,Q_1$ be associated projectors. Then, $A$ induces an isomorphism from $E_1$ to $G_1$, whose inverse will be denoted by $T\in \mathcal L(G_1,E_1)$. The equation $C(e,f)=0$ is equivalent to 
$AP_1 e =Bf$, which is equivalent to the system 
$$\left \{\begin{array}{l} AP_1 e = Q_1 B f, \\ Q_0 Bf = 0,\end{array}\right .$$
again equivalent to 
\begin{equation}
\left \{\begin{array}{l} P_1 e = TQ_1Bf, \\ Q_0Bf =0.\end{array}\right .
\label{eq:complemented}
\end{equation}
Since $Q_0 B$ has finite rank, its kernel - say $F_0$ - has a (finite dimensional) complement $F_1$. By \eqref{eq:complemented}, the kernel of $C$ is 
$$\ker C = \{ (e_0+TQ_1 Bf_0,f_0 ) \in E\times F \ |\ e_0\in E_0, \ f_0\in F_0\},$$ 
and the closed linear subspace $E_1\times F_1$ is a complement of $\ker C$. 
\end{proof}

%%%%%%%%%%%%%%%%

\subsection{An abstract transversality theorem} Let $\M$ be a Hilbert manifold modeled on the infinite dimensional separable real Hilbert space $\HH$, and let 
$F$ be a complete gradient-like \textit{Morse vector field} of class $C^k$ on $\M$ (i.e. all of its rest points are hyperbolic). Recall that a rest point $x$ is said \textit{hyperbolic} if the spectrum of the Jacobian of $F$ at $x$ is disjoint from the imaginary axis. In this case, the \textit{linear 
unstable space} $\HH^u_x$ is the subspace of $\HH$ corresponding to the part of the spectrum with positive real part.
We further assume that $F$ admits a Lyapounov function $f$ such that the pair $(F,f)$ satisfies the Palais-Smale condition 
(this implies, in particular, that the set of rest points of $F$ is at most countable), and that there exists a strongly integrable $(0)$-essential sub-bundle $\mathcal E\subset T\M$ 
and an essentially vertical family $\mathcal F$ (in the sense of Definition~\ref{dfn:evf}) such that\footnote{Here we use the same notation as in \cite{AM:05}. We shall notice that Condition (C3) in \cite{AM:05} is formulated in terms of admissible presentations and Hausdorff measure of non-compactness. However, for our applications it is more convenient to express it in terms of invariance of an essentially vertical family.}:
\begin{enumerate}
\item[(C1)] for every rest point of $F$, the linear unstable space $\HH^u_x$ belongs to the essential class $\mathcal E(x)$, 
\item[(C3)] $\mathcal F$ is positively $F$-invariant, that is, if $A\in\mathcal F$ then $\phi\big ([0,T]\times A)\in \mathcal F$ for every $T\geq 0$  
\end{enumerate}
Here, $\phi$ denotes the flow of $F$. In particular, the intersection between stable and unstable manifolds of any two rest points of $F$ is pre-compact by \cite[Theorem 6.5]{AM:05}. 
In what follows, we denote by 
$$\mathrm{m}(x) := \dim (\HH^u_x,\mathcal E(x))$$
the relative Morse index of the rest point $x\in \text{sing}\, (F)$. Our goal is to show that $F$ can be perturbed away from the set of its rest points in order to obtain a vector field which has the \textit{Morse-Smale 
property up to order k}, that is, for which the stable and unstable manifolds of critical points whose relative Morse indices differ at most by $k$ intersect transversally. Rather than considering a particular space of perturbations, we list the properties that such a space should have. 
Thus, consider two neighborhoods $\U\subset \V\subset \M$ of sing$\, (F)$ such that each rest point of $F$ belong to a different connected component of $\V$ and
let $\mathfrak X$ be a Banach space consisting of vector fields of class $C^k$ in $\M$ such that:
\begin{enumerate}
\item[(B1)] every $X\in \mathfrak X$ vanishes on $\U$, 
\item[(B2)] for every $X\in \mathfrak X$ with $\|X\|_{\mathfrak X} <1$, $\text{sing}\, (F+X) = \text{sing}\, (F)$, $f$ is a Lyapounov function for $F+X$, and $(f,F+X)$ satisfies the Palais-Smale condition, 
\item[(B3)] the convergence in $\mathfrak X$ implies the $C^k_{\text{loc}}$ convergence, 
\item[(B4)] $\mathfrak X$ is closed under multiplication by a vector space of real functions on $M$ which contains bump functions,
\item[(B5)] $ \{X(p)\, |\, X\in \mathfrak X\}$ is dense in $T_p\M$ for every $p\in \M\setminus \V$, and 
\item[(B6)] every vector field $F+X, X\in \mathfrak X$, satisfies Condition (C3) with respect to $\mathcal E$.
\end{enumerate}

We denote by $\mathfrak X_1$ the unit ball in $\mathfrak X$. Notice that, for every $X\in \mathfrak X_1$, $F+X$ trivially satisfies (C1) with respect to $\mathcal E$ by (B1) and (B2), and 
the intersection 
$$W^u(x;F+X)\cap W^s(y;F+X)$$
of stable and unstable manifolds any two rest points $x,y\in \text{sing}\, (F)$ is pre-compact by (B1), (B2), and (B6).

\begin{thm}
The subset of $\mathfrak X_1$ of vector fields $X$ for which $F+X$ is Morse-Smale up to order $k$ is generic.
\label{thm:morsesmalegeneral}
\end{thm}

Here, order $k$ of the Morse-Smale property is precisely the order of differentiability of the 
vector field $F$ and of its perturbations. This fact is determined by the regularity versus Fredholm index assumption required by the Sard-Smale theorem \ref{thm:sard}. In a finite-dimensional setting the problem does not occur because $C^k$ functions can always be $C^k$-approximated by smooth ones, while such an approximation may not be possible on an infinite-dimensional Hilbert space (see for instance \cite{NS:73,LL:86}). 

Theorem \ref{thm:morsesmalegeneral}, whose proof will take up the rest of this subsection, generalizes Theorem 2.20 in \cite{Abbondandolo:2006lk}, where the case of finite Morse indices is treated. As already pointed out, in this setting more assumptions are needed in order to ensure that all good properties of $(f,F)$ are not destroyed after perturbation.

Thus, let $x,y\in \text{sing}\, (F)$ with
$$\mathrm{m}(x)-\mathrm{m}(y)\leq k$$
be two rest points of $F$ such that $W^u(x;F)$ and $W^s(y;F)$ have a non-empty intersection. 
Denote by $\mathcal C$ the space of $C^1$-curves $u:\R\to \M$ such that 
$$\lim_{t\to - \infty} u(t) = x, \quad \lim_{t\to+\infty} u(t) = y. \quad \text{and}\ \lim_{t\to \pm \infty} u'(t) =0.$$
The space $\mathcal C$ is a smooth Banach manifold and its tangent space at $u$ is 
$$T_u\mathcal C = \Big \{v \ C^1\text{-section of} \ u^*T\M \ \Big |\ \lim_{t\to \pm \infty} v(t) = \lim_{t\to \pm \infty} \nabla_t v(t) =0\Big \},$$
where $\nabla_t$ denotes the covariant derivative associated to some connection on $u^*T\M$. The Banach manifold $\mathcal C$ is the base space of a smooth Banach bundle $\BB$ whose fiber at $u\in \mathcal C$ is the space 
$$\BB_u :=\Big \{ v \ C^0\text{-section of } u^*T\M \ \Big |\ \lim_{t\to \pm \infty} v(t) =0 \Big \},$$
endowed with the $C^0$-norm. The zero section of $\BB$ will be denoted by $\mathbb O_\BB$.

The fact that $F$ is of class $C^k$ and Property (B3) on the space $\mathfrak X$ imply that the map 
$$\Phi:\mathfrak X_1 \times \mathcal C \to \BB, \quad (X,u) \mapsto u' - (F+X)\circ u,$$
is of class $C^k$. Notice that, for a fixed $X\in \mathfrak X_1$, the map $\Phi_X:= \Phi(X,\cdot)$ is a section of the Banach bundle $\BB$. Moreover, the set  
$$\mathcal Z:= \Phi^{-1}(\mathbb O_\BB)$$ 
is precisely the union over all $X\in \mathfrak X_1$ of the set $W^u(x;F+X) \cap W^s(y;F+X)$. 

For $(X,u) \in \mathcal Z$, we denote by $D^f \Phi(X,u)$ the fibrewise differential 
$$D^f \Phi(X,u):\mathfrak X\times T_u\mathcal C \to \BB_u, \quad D^f \Phi (X,u) [(Y,v)]=D^f_1 \Phi(X,u)[Y] + D^f_2\Phi(X,u)[v],$$
where  
$$D^f_1 \Phi (X,u)[Y] = -Y\circ u, \quad  D^f_2\Phi(X,u)[v] = \nabla_t v - \nabla (F+X) (u) [v].$$

In the next lemma we discuss the properties of the partial fibrewise differential $D^f_2(X,u)$. 

\begin{lem}
\label{lem:morsesmalegeneral1}
For $(X,u)\in \mathcal Z$ the following statements hold:
\begin{enumerate}
\item $D^f_2 \Phi(X,u)$ is a Fredholm operator of index $\mathrm{m}(x)-\mathrm{m}(y)$.
\item $D^f_2 \Phi(X,u)$ is onto if and only if $W^u(x;F+X)$ and $W^s(y;F+X)$ meet transversally along $u$, and
\item For every $a<b$ and every $w\in \BB_u$ there exists $v\in T_u\mathcal C$ such that 
$$D^f_2 \Phi(X,u)[v](t) = w(t) , \quad \forall t\in (-\infty,a] \cap [b, +\infty).$$
\end{enumerate}
\end{lem}

\begin{proof}
We choose a trivialization $U:\overline \R\times \HH \to u^*T\M$ and identify $\HH$ with $T_{u(0)} \M$ in such a way that $U(0)=\text{Id}_\HH$. Then, the (linearized)
flow defined by $D^f_2(X,u)$ can be read on $\HH$ as 
$$Y(t):= U(t)^{-1} \diff \Phi^{F+X} (t,u(0)),$$ 
where $\Phi^{F+X}$ is the flow of $F+X$ and $\diff$ denotes differentiation with respect to the variable on $\M$. The function $Y:\R\to \mathcal L(\HH)$ is the 
solution of an asymptotically hyperbolic linear system on $\R$ in $\HH$ 
\begin{equation}
\left \{\begin{array}{l} Y'(t) = A(t) Y(t), \\ Y(0)= \text{Id}, \end{array}\right . 
\label{asymptoticallyhyperbolicsystem}
\end{equation}
where $A(-\infty), A(+\infty)$ are conjugated to a local representation of $\nabla F(x),\nabla F(y)$ respectively (by construction $X$ vanishes in a 
neighborhood of $x$ and $y$). The operator $D^f_2\Phi (X,u)$ can thus be read on $\R\times \HH$ as 
\begin{equation}
D_A:= \frac{\diff}{\diff t} - A(t):C^1_0(\R,\HH)\to C^0_0(\R,\HH),
\label{localrep}
\end{equation}
where 
\begin{align*}
C^0_0 (\R,\HH) &:=\Big \{v:\R\to \HH \ \Big |\ \lim_{t\to \pm \infty} v(t) =0 \Big\},\\
C^1_0 (\R,\HH) &:= \Big \{v:\R\to \HH \ \Big |\ \lim_{t\to \pm \infty} v(t) = \lim_{t\to \pm\infty} \dot v(t) =0\Big \}.
\end{align*}
By the identification $\HH=T_{u(0)}\M$ we further have 
\begin{equation}
W_A^u = T_{u(0)} W^u(x), \quad W^s_A = T_{u(0)} W^s(y),
\label{localrepstable}
\end{equation}
where
\begin{align*}
W_A^u := \Big \{v \in \HH \ \Big |\ \lim_{t\to -\infty} Y(t) v =0 \Big \},\\ 
W_A^s := \Big \{v\in \HH \ \Big |\ \lim_{t\to +\infty} Y(t) v=0\Big \},
\end{align*}
are the \textit{unstable} resp. \textit{stable space} of the system $Y'=AY$. We shall observe, before going further, that the properties 
of $D^f_2\Phi (X,u)$ we are interested in are independent of the choice of the trivialization. Indeed, choosing a different trivialization changes $Y$ to $GY$, for some 
$G\in C^1(\overline \R; GL(\HH))$, and $A, D_A$ respectively to 
$$-G'G^{-1} + GAG^{-1}, \quad G\circ D_A \circ G^{-1}.$$
\begin{enumerate}
\item Theorem 5.1 in \cite{AM:03} together with \eqref{localrepstable} implies that $D_A$ (hence, also $D^f_2\Phi (X,u)$) is Fredholm if and only if $W^u(x)$ and $W^s(y)$ 
have Fredholm intersection at $u(t)$ for some $t\in \R$ (hence for every $t\in \R$); see also \cite[Corollary 8.2, (iv)]{AM:03}. In this case, the Fredholm index of $D^f_2\Phi (X,u)$ is given by 
$$\text{ind}\,  (T_{u(t)} W^u (x) , T_{u(t)} W^s(y)).$$

Using Properties (B1), (B2), (B6) and Proposition
3.1 in \cite{AM:05} we obtain that, for every $t\in \R$, the tangent space $T_{u(t)}W^u(x)$ belongs to the essential class $\mathcal E(u(t))$, and since $\mathcal E$ is (0)-essential,
we also have the identity between integers 
$$\dim (T_{u(t)}W^u(x), \mathcal E(u(t)) ) = \mathrm{m}(x), \quad \forall t\in \R.$$
On the other hand, using again Properties (B1), (B2), (B6), from Corollary 3.2 in \cite{AM:05} we deduce that, for any $t\in \R$, $(T_{u(t)}W^s(y),\mathcal E(u(t)))$ is a Fredholm pair 
with Fredholm index 
$$\text{ind}\, (T_{u(t)}W^s(y),\mathcal E(u(t))) = - \mathrm{m}(y).$$
Therefore, $(T_{u(t)} W^u (x) , T_{u(t)} W^s(y))$ is a Fredholm pair with index 
\begin{align*}
\text{ind} \, (T_{u(t)} W^u (x)& , T_{u(t)} W^s(y)) \\
 &= \text{ind}\, (T_{u(t)}W^s(y),\mathcal E(u(t))) + \dim (T_{u(t)}W^u(x), \mathcal E(u(t)) ) \\
 &= \mathrm{m}(x)- \mathrm{m}(y).
 \end{align*}
\item This is precisely \cite[Corollary 8.2, (iii)]{AM:03}.
\item After trivialization the claim can be reformulated as follows: For every $a<b$ and every $v\in C^0_0(\R,\HH)$ we can find $w\in C^1_0(\R,\HH)$ such that 
\begin{equation}
D_A w(t) = v(t)
\label{partialhyperbolicity}
\end{equation}
holds on $(-\infty,a]$ as well as on $[b,+\infty)$. Adapting the proof of \cite[Proposition 4.2]{AM:03} to our setting yields that the operators $D_A |_{(-\infty,a]}$ 
and $D_A|_{[b,+\infty)}$ are onto, and hence we can find $w_a \in C^1_0((-\infty,a],\HH)$, $w_b \in C^1_0([b,+\infty), \HH)$ such that 
$$D_A |_{(-\infty,a]} w_a = v|_{(-\infty,a]}, \quad D_A |_{[b,+\infty)} w_b = v|_{[b,+\infty)}.$$ 
Now, any $w\in C^1_0(\R,\HH)$ such that $w|_{(-\infty,a]} =w_a$ and $w|_{[b,+\infty)} =w_b$ satisfies \eqref{partialhyperbolicity}. \qedhere
\end{enumerate}
\end{proof}

\begin{lem}
\label{lem:morsesmalegeneral2}
The map $\Phi$ is transverse to the zero-section $\mathbb O_\BB$. In particular, $\mathcal Z$ is a $C^k$-submanifold of $\mathfrak X_1\times \mathcal C$. 
\end{lem}
\begin{proof}
All we have to show is that, for every $(X,u)\in\mathcal Z$, the kernel of the fibrewise differential $D^f\Phi(X,u)$ is complemented in $\mathfrak X\times T_u\mathcal C$, and that
$$D^f\Phi(X,u):\mathfrak X\times T_u\mathcal C\to \BB_u$$
is onto, that is, that for every $v\in \BB_u$ we can find $(Y,w)\in \mathfrak X\times T_u\mathcal C$ such that 
\begin{equation}
-Y \circ u + \nabla_t w - \nabla (F+X) (u) [w] = v.
\label{transversality1}
\end{equation}
By Lemma \ref{lem:morsesmalegeneral1}, $D_2^f\Phi(X,u)$ is Fredholm, and hence in particular has finite codimensional range and complemented kernel. Thus, Lemma \ref{lem:generic3} shows that $D^f\Phi(X,u)$ has complemented kernel as well.

Using again a local trivialization of $u^*T\M$ we see that \eqref{transversality1} can be read as 
\begin{equation}
-Y(t) + D_A w(t) = v(t), \quad \forall t\in \R,
\label{transversality2}
\end{equation}
where $Y\in C^k(\R,\HH)$ is such that:
\begin{enumerate}
\item[i)] $Y(t)=0$ for every $t\in \R\setminus (a,b)$, for suitable $-\infty < a<b<+\infty$, by Property (B1), and
\item[ii)] $Y(t)\in V_t$, $V_t\subset \HH$ dense subspace for every $t\in (a,b)$, by Property (B5). 
\end{enumerate}

By Lemma \ref{lem:morsesmalegeneral1}, (1) we know that the operator $D_A$ is Fredholm with Fredholm index $m(x)-m(y)$. In particular, the range of the restriction of the operator
\begin{equation}
(Y,w) \mapsto - Y + D_A w 
\label{operatortransversality}
\end{equation}
to $\{0\}\times C^1_0(\R,\HH)$ is a closed finite codimensional subspace of $C^0_0(\R,\HH)$. Therefore,
the claim follows provided we can show that the range of the operator in \eqref{operatortransversality}
is dense in $C^0_0(\R,\HH)$.

Thus, let $v \in C^0_0(\R,\HH)$. By Lemma~\ref{lem:morsesmalegeneral1}, (3) we can find $w_-\in C^1_0((-\infty,0],\HH)$ resp. $w_+\in C^1_0([0,+\infty),\HH)$ which solves \eqref{transversality2} on $(-\infty,0]$ resp. $[0,+\infty)$ with $Y\equiv 0$. If $w_\pm$ match to form a solution of \eqref{transversality2} on $\R$, then we are done. Otherwise, we choose $\mu>0$ such that the restriction $u|_{[-\mu,\mu]}$ is entirely contained outside $\U$, the neighborhood of $\text{sing}\, (F)$ over which every $Y\in \mathfrak X$ vanishes, and consider the convex combination 
$$\tilde w(t) := \varphi(t) \cdot w_-(t) + (1-\varphi(t)) \cdot w_+(t),$$
where $\varphi:\R\to [0,1]$ is a smooth non-increasing function such that 
$$\varphi |_{(-\infty, -\mu ]} \equiv 1, \quad \varphi|_{[\mu,+\infty)} \equiv 0.$$
By construction, $\tilde w \in C^1_0(\R,\HH)$ satisfies \eqref{transversality2} on $(-\infty,-\mu] \cup [\mu,+\infty)$. Therefore, we can rewrite \eqref{transversality2} as 
\begin{equation}
\label{transversality3} 
Y(t) = D_A \tilde w (t) - v(t) =: \tilde v(t), \quad \forall t\in \R,
\end{equation}
where $\tilde v \in C^0_0(\R,\HH)$ has support in $[-\mu,\mu]$. Fix now $\epsilon>0$, 
and choose $\delta >0$ such that 
$$|t_0-t_1|< \delta \quad \Rightarrow \quad \|\tilde v(t_0)-\tilde v(t_1)\| < \frac{\epsilon}{4}.$$
Recall that this is possible since $\tilde v$ has compact support. Fix now 
$$-\mu := \tau_0 < \tau_1 < ... < \tau_{\ell-1}< \tau_\ell := \mu$$
such that $|\tau_{i-1}-\tau_i|<\delta$ for every $i=1,...,\ell$, and let 
$$\sigma_i := \frac{\tau_{i-1}+\tau_i}{2}, \quad i=1,...,\ell,$$
be the middle point of the interval $[\tau_{i-1},\tau_i]$. By ii), for every $i=1,...,\ell$ we can find 
$Y_i\in V_{\sigma_i}$ such that 
$$\|Y_i - \tilde v (\sigma_i)\| < \frac{\epsilon}{4}.$$
For $\rho >0$ small enough we now consider a partition of unity $(\varphi_i)_{i=1,...,\ell}$ of $(-\mu-\rho,\mu+\rho)$ subordinated to the open covering $(-\tau_{i-1}-\rho, \tau_i + \rho),  i =1,...,\ell,$ such that $\varphi_i |_{[\tau_{i-1}+\rho, \tau_i-\rho]}\equiv 1$ for every $i=1,...,\ell$, and define 
$$Y(t) := \sum_{i=1}^\ell \varphi_i(t) \cdot Y_i.$$
Property (B4) implies that $Y\in \mathfrak X$. Furthermore, by construction we have for $t\in [\tau_{i-1}+\rho,\tau_i-\rho]$
\begin{align*}
\|Y(t) - u(t)\| &= \| Y_i - \tilde v (t)\| \leq \|Y_i - \tilde v(\sigma_i)\| + \|\tilde v(\sigma_i)-\tilde v (t)\| < \frac{\epsilon}{2},
\end{align*}
whereas for $t \in (\tau_{i-1} - \rho, \tau_{i-1}+\rho)$, $i=2,...,\ell$,
\begin{align*}
\|Y(t) - \tilde v(t)\| &= \| \varphi_{i-1} (t) Y_{i-1} + \varphi_i (t) Y_i - \tilde v (t)\| \\
		&= \| \varphi_{i-1} (t) ( Y_{i-1} - \tilde v (t)) + \varphi_i (t) (Y_i - \tilde v(t))\| \\
		&\leq \|Y_{i-1}-\tilde v (t) \| + \|Y_i - \tilde v(t)\| \\
		&\leq \|Y_{i-1} - \tilde v (\sigma_{i-1})\| + \|\tilde v (\sigma_{i-1})-\tilde v (t)\| + \|Y_i - \tilde v(\sigma_i)\| + \|\tilde v (\sigma_i)-\tilde v(t)\| \\
		&< \epsilon,
\end{align*}
where we used the fact that $\varphi_{i-1}$ and $\varphi_i$ are the only two functions of the chosen partition of unity that do not vanish identically on $(\tau_{i-1} - \rho, \tau_{i-1}+\rho)$ and that 
$$|t-\sigma_{i-1}| ,\ |t-\sigma_i| < \delta, \quad \forall t \in (\tau_{i-1} - \rho, \tau_{i-1}+\rho).$$
Finally, one shows in an identical fashion that $\|Y(t)-\tilde v (t)\|<\epsilon$ also for $t\in (-\mu-\rho,-\mu+\rho)\cup (\mu-\rho,\mu+\rho)$, thus obtaining
$$\|Y-\tilde v\|_\infty <\epsilon,$$
which is to say that 
$$\|D^f\Phi(X,u) (Y,\tilde w) - v\|_\infty <\epsilon.$$
This completes the proof.
\end{proof}

Let $S\subset \U$ be a small smooth sphere center in $x$ and transversal to the flow of $F$, and hence also to the flow of $F+X$ for every $X\in \mathfrak X$ since $X$ vanishes on $\U$ by property (B1). We denote by $\mathcal Z_0\subset \mathcal Z$ the codimension-one $C^k$-submanifold given by pairs $(X,u)\in \mathcal Z$ such that $u(0)\in S$, and by 
$$\pi:\mathcal Z_0 \to \mathfrak X_1, \quad (X,u) \mapsto X,$$
the restriction to $\mathcal Z_0$ of the projection onto the first factor of $\mathfrak X_1\times \mathcal C$. 

\begin{lem}
\label{lem:morsesmalegeneral3}
The following statements hold:
\begin{enumerate}
\item The map $\pi$ is Fredholm of index $\mathrm{m}(x)-\mathrm{m}(y)-1$.
\item $X\in \mathfrak X_1$ is a regular value of $\pi$ if and only if $W^u(x;F+X)$ and $W^s(y;F+X)$ meet transversally.
\end{enumerate}
\end{lem}
\begin{proof}
$\ $

\begin{enumerate}
\item[(1)] By Lemma \ref{lem:morsesmalegeneral2}, $\mathbb O_\BB$ is a regular value of $\Phi$. Since any $X\in \mathfrak X_1$ is a regular value for the projection $\mathfrak X_1\times \mathcal C \to \mathfrak X_1$, Statement (1) follows immediately from Proposition \ref{prop:generic1}, (2) combined with Lemma \ref{lem:morsesmalegeneral1}, (1). Indeed, 
the projection $\mathcal Z\to \mathfrak X_1$ is Fredholm with index $\mathrm{m}(x)-\mathrm{m}(y)$, and restricting to $\mathcal Z_0$ just reduces the dimension of the kernel by one in each tangent space.
\item[(2)] Follows trivially from Lemma \ref{lem:morsesmalegeneral1}, (2) combined with Proposition \ref{prop:generic1}, (1). \qedhere
\end{enumerate}
\end{proof}

The last piece of information we need in order to apply the Sard-Smale theorem \ref{thm:sard} is the $\sigma$-properness of the map $\pi$; notice indeed that the classical Sard-Smale theorem is not applicable 
in this setting as the spaces under consideration are not Lindel\"of. 

\begin{prop}
\label{prop:sigmaproperness}
The map $\pi:\mathcal Z_0\to \mathfrak X_1$ is $\sigma$-proper. 
\end{prop}

In order to prove the proposition above, we need to show that we can write $\mathfrak X_1\times \mathcal C$ as countable union of open sets $(\U_n)$ in such a way that 
\begin{equation}
\pi^{-1}(\mathcal K)\, \cap\, \overline{\U_n} \quad \text{compact},\quad \forall \mathcal K \subset \mathfrak X_1 \ \text{compact}.
\label{eq:properness1}
\end{equation}
For a fixed $X\in \mathfrak X_1$, the pre-image $\pi^{-1}(\{X\})$ can be identified with
$$W^u(x;F+X) \cap W^s(y;F+X),$$ 
and as such is pre-compact by the Properties (B1), (B2), and (B6) of the space $\mathfrak X$. Moreover, the only source of non compactness is represented by those flow-lines which converge in the Hausdorff distance to a broken flow line from $x$ to $y$, see \cite[Proposition 8.2]{AM:05} or Proposition~\ref{prop:brokenflowlines}. In particular, the set 
$$\pi^{-1}(\{X\}) \cap \Big \{(X,u)\ \Big |\ u \in \mathcal C\ \text{is such that}\ \ \text{dist} \big (u(\cdot), \text{sing}\, (F)\setminus \{x,y\}\big )\geq \frac 1n \Big \}$$
is compact for every $n\in\N$. This suggest that \eqref{eq:properness1} should hold for 
\begin{equation}
\mathcal U_n :=\Big  \{(X,u) \in \mathfrak X_1\times  \mathcal C\ \Big |\  \text{dist} \big (u(\cdot), \text{sing}\, (F)\setminus \{x,y\}\big )> \frac 1n \Big \}, \quad \forall n \in\N.
\label{eq:un}
\end{equation}
As it turns out, \eqref{eq:properness1} is implied by the following generalization of \cite[Theorem 6.5]{AM:05} to sequences of vector fields in $\mathfrak X_1$ converging to some $X_\infty\in \mathfrak X_1$. 

\begin{thm}
Let $(X_m)\subset \mathfrak X_1$ be a sequence such that $X_m\to  X_\infty\in \mathfrak X_1$. Then, denoting by $\phi_m$ the flow on $\M$ defined by $F+X_m$ for every $m\in\N$, the following statement holds: 
Let $(p_m)\subset \M$, $(s_m)\subset (-\infty,0]$, and $(t_m)\subset [0,+\infty)$ be such that $\phi_m(s_m,p_m)\to x$ and $\phi_m(t_m,p_m)\to y$. Then, the sequence $(p_m)$ is compact. 

As a corollary, the union
$$\bigcup_{m\in \N} W^u(x;F+X_m)\cap W^s(y;F+X_m)$$ 
is pre-compact.
\label{thm:precompactnessforsequences}
\end{thm}

\begin{proof}[Proof of Proposition~\ref{prop:sigmaproperness}] Clearly, it suffices to show that, for every $n\in\N$ and every sequence $(X_m)\subset \mathfrak X_1$ converging to some $X_\infty \in \mathfrak X_1$, the set 
$$\Big (\bigcup_{m\in \N \cup \{\infty\}} W^u(x;F+X_m) \cap W^s(y;F+X_m) \Big )\  \cap \ \overline \U_n$$
is compact. By Theorem~\ref{thm:precompactnessforsequences}, we have that 
$$\overline{\bigcup_{m\in \N \cup \{\infty\}} W^u(x;F+X_m) \cap W^s(y;F+X_m)}$$
is compact. The proof of Proposition 8.2 in \cite{AM:05} now extends to this more general setting showing that the only 
source of non-compactness of 
$$\bigcup_{m\in \N} W^u(x;F+X_m)\cap W^s(y;F+X_m)$$
is given by those sequences of flow lines that converge in the Hausdorff distance to a broken flow line of $F+ X_m$ for some $m\in\N\cup\{\infty\}$,
thus completing the proof.
\end{proof}

\begin{proof}[Proof of Theorem~\ref{thm:precompactnessforsequences}]
The proof is analogous to the one of \cite[Theorem 6.5]{AM:05}. However, to be able to pass from a fixed vector field to a converging sequence of vector fields we have to employ the stronger definition of essentially vertical families given in Definition~\ref{dfn:evf}. For this reason, we only provide a sketch of the proof highlighting where modifications of the original argument are needed, and refer to \cite{AM:05} for the details. 

We start noticing that Lemma 6.7 in \cite{AM:05}, being a local statement, trivially generalizes to our more general setting. Indeed, by Property (B1), all vector fields $X_m$, $m\in \N\cup \{\infty\}$, vanish on a neighborhood of $\text{sing}\, (F)$. 

Without loss of generality, we can assume that $(p_m)$ is bounded away from $x$ and $y$, and, by a minimality argument, that there are no sequences $(s_m')\subset (-\infty,0]$ resp. $(t_m')\subset [0,+\infty)$ such that,
up to a subsequence, 
$$\phi_m(s_m',p_m)\to z_-,\qquad \text{resp.}\quad \phi_m(t_m',p_m)\to z_+,$$ 
for some rest point $z_-$ with $f(z_-)<f(x)$ resp. $z_+$ with $f(z_+)>f(y)$, where $f$ is a Lyapounov function for $F$, and by Property (B2), of $F+X_m$ for all $m\in \N\cup \{\infty\}$. Also, up to neglecting finitely many $m$'s, we can suppose that for all $m\in \N$ we have
$$\phi_m(s_m,p_m) \in U, \quad \phi_m(t_m,p_m)\in V,\quad p_m\not \in U\cup V,$$
where $U$ and $V$ and closed neighborhoods of $x$ and $y$ respectively where all $X_m$ vanish.  In virtue of Lemma 6.7 in \cite{AM:05}, up to shrinking $U$ further we can assume that for a suitable sequence $s''_m \in (s_m,0)$ there holds
$\{\phi_m(s_m'',p_m) \ |\ m\in \N\} \subset \partial U$ and
\begin{align}
\{ \phi_m(s_m'',p_m) \ |\ m\in \N\} & \in \mathcal F, \label{gru-1}\\
\limsup_{n\to +\infty} f(\phi_m(s_m'',p_m)) &< f(x).\label{gru0}
\end{align}
Similarly, up to shrinking $V$ further we can assume that for a suitable sequence $t_m''\in (0,t_m)$ there holds 
$\{\phi_m(t_m'',p_m) \ |\ m\in \N\} \subset \partial V$ and
\begin{align}
\{ \phi_m(t_m'',p_m) \ |\ m\in \N\} \cap A \ &\text{pre-compact},\ \forall A \in \mathcal F, \label{gru}\\
\liminf_{n\to +\infty} f(\phi_m(t_m'',p_m)) &> f(y). \label{gru1}
\end{align}
Since $X_m\to X_\infty$ and $(f,F+X_m)$ satisfies the Palais-Smale condition for every $m\in \N\cup\{\infty\}$, by \eqref{gru0}
and \eqref{gru1} we 
have that the sequence $(t_m''-s_m'')$ is bounded, say $t_m''-s_m''\leq T,\  \forall m\in \N$. We now claim that 
\begin{equation}
\bigcup_{m\in \N} \phi_m \big ([0,T]\times \phi_m(s_m'',p_m) \big ) \in \mathcal F.
\label{gru2}
\end{equation}
Notice that this completes the proof; indeed, 
$$\{\phi_m(t_n'',p_m)\ |\ m\in \N\} \subset \bigcup_{m\in \N} \phi_m \big ([0,T]\times \phi_m(s_n'',p_m) \big )$$
is essentially vertical as well, and hence in \eqref{gru} we can take $A= \{\phi_m(t_n'',p_m)\ |\ m\in \N\}$, thus 
obtaining that the sequence $(\phi_m(t_m'',p_m))$ is compact. By the boundedness of $t_m''$ and by 
the fact that the vector fields $F+X_m$ are all complete, we conclude that also the sequence $p_m$ is compact.

In order to prove \eqref{gru2}, we start noticing that for fixed $m\in \N\cup \{\infty\}$ the set 
$$\phi_m\big ([0,T]\times \phi_m(s_m'',p_m) \big )$$
is essential vertical, as $\mathcal F$ is positively $(F+X_m)$-invariant by Property (B6). Let now $\epsilon >0$ be fixed.
Since $X_m\to X_\infty$, we have that there exists $m_\epsilon \in \N$ such that 
$$\bigcup_{m\geq m_\epsilon} \phi_m\big ([0,T]\times \phi_m(s_m'',p_m) \big ) \subset B_\epsilon \big (\phi_\infty \big ([0,T]\times \{ \phi_m(s_m'',p_m)\ |\ m\in\N\} \big ) \big ),$$
where $B_\epsilon (A) := \{p \in \M\ |\ \text{dist}\, (p,A)<\epsilon\}$ denotes the open $\epsilon$-neighborhood of the set $A\subset \M$.
Notice that
$$\phi_\infty \big ([0,T]\times \{ \phi_m(s_m'',p_m)\ |\ m\in\N\} \big )$$
is essentially vertical by \eqref{gru-1} and by the fact that $\mathcal F$ is positively $(F+X_\infty)$-invariant.

 Hence 
\begin{align*}
\bigcup_{m\in \N} & \phi_m\big ([0,T]\times \phi_m(s_m'',p_m) \big )  \\
&= \Big (\bigcup_{m<m_\epsilon}  \phi_m\big ([0,T]\times \phi_m(s_m'',p_m) \big )\Big ) \cup \Big ( \bigcup_{m\geq m_\epsilon}  \phi_m\big ([0,T]\times \phi_m(s_m'',p_m) \big ) \Big )
\end{align*}
is contained in the open $\epsilon$-neighborhood of the essentially vertical set 
$$ A_\epsilon:= \Big (\bigcup_{m<m_\epsilon}  \phi_m\big ([0,T]\times \phi_m(s_m'',p_m) \big )\Big ) \cup\Big ( \phi_\infty \big ([0,T]\times \{ \phi_m(s_m'',p_m)\ |\ m\in\N\} \big )\Big ).$$
Since $\epsilon$ is arbitrary, \eqref{gru2} follows from the stability property of $\mathcal F$, see Property (iii) in Definition~\ref{dfn:evf}.
\end{proof}

\begin{proof}[Proof of Theorem \ref{thm:morsesmalegeneral}]
For any pair of distinct rest points $x,y$ of $F$ such that $\mathrm{m}(x)-\mathrm{m}(y)\leq k$, let $\mathfrak X_1(x,y)$ be the set of regular values of the map $\pi$. By Lemma~\ref{lem:morsesmalegeneral3}, Part (1), and Proposition~\ref{prop:sigmaproperness}, the $C^k$-map $\pi$ is 
Fredholm of index $\mathrm{m}(x)-\mathrm{m}(y)-1<k$ and $\sigma$-proper, so the Sard-Smale theorem \ref{thm:sard} implies that the set $\mathfrak X_1(x,y)$ of regular values for $\pi$ is generic in $\mathfrak X_1$. By Lemma \ref{lem:morsesmalegeneral3}, Part (2), for 
every $X\in \mathfrak X_1(x,y)$ the manifolds $W^u(x;F+X)$ and $W^s(y;F+X)$ meet transversally. 

Since the set sing$\, (F)$ is at most countable, the intersection 
\begin{equation}
\label{eq:generic}
\bigcap\ \Big \{ \mathfrak X_1(x,y) \ \Big |\ x\neq y \in \text{sing}\, (F), \ \mathrm{m}(x)-\mathrm{m}(y)\leq k\Big \}
\end{equation} 
is also a generic subset of $\mathfrak X_1$. By Property (B2), for any $X\in \mathfrak X_1$ the vector field $F+X$ has exactly the same rest points of $F$, so for every $X$ in the generic subset \eqref{eq:generic} the vector 
field $F+X$ satisfies the Morse-Smale property up to order $k$. 
\end{proof}

%%%%%%%%%%%%%%%%

\subsection{The Morse-Smale property for the Hamiltonian action} In this section we show how to apply the abstract transversality theorem~\ref{thm:morsesmalegeneral} to achieve - after a generic perturbation 
of the negative pseudo gradient vector field $F$ of $\A_H$ constructed in Subsection~\ref{s:pseudogradient} - transverse intersection of stable and unstable manifolds of critical points of the Hamiltonian action $\A_H$ whose relative Morse indices differ at most by two.
To do this we need to ensure that $F$ is at least of class $C^2$, that is, that $\A_H$ is at least of class $C^3$. Using the converse of Taylor's theorem (see e.g. \cite{AR:67}) and an argument 
analogous to the one in the proof of Theorem 4 in \cite{Hofer:1994bq}, we see that $\A_H$ is of class $C^3$ provided $H^{1-s}$ compactly embeds\footnote{Actually, compact embedding in $L^3$ should be enough, and this would yield the weaker condition $s\in (1/2,5/6)$.} in $L^4$. 
According to the Sobolev embedding theorem for fractional Sobolev spaces (see e.g. \cite[Theorem 6.5]{Dinezza:2012}), this is the case when 4 is strictly smaller than the critical exponent
$$p^*:= \frac{2}{1-2(1-s)},$$
which in turns implies that $s<3/4$. Thus, hereafter we will assume $s\in (1/2,3/4)$.

%The class of perturbations will be constructed in such a way that all properties of $-\nabla A_H$ with respect to the 
%(0)-essential sub-bundle $\mathcal E$ defined in Section 3 will remain true for the perturbed vector fields. 

\begin{thm}[Morse-Smale property up to order two for the Hamiltonian action]
\label{thm:transverseintersection}
Let $H:\T\times T^*M \to \R$ be a smooth time-depending Hamiltonian satisfying the growth condition~\eqref{eq:growthcondition} and such that the critical points of the corresponding Hamiltonian action functional $\A_H:\M^{1-s}\to \R$ are all hyperbolic. Denote by 
$F$ the negative pseudo-gradient vector field constructed in Subsection~\ref{s:pseudogradient}, for some $s\in (1/2,3/4)$. Then, there exists a generic set $\mathcal P$ of small 
perturbations $P$ of $F$ which satisfy the following properties:
\begin{enumerate}
\item $P$ is a complete Morse gradient-like vector field with Lyapounov function $\A_H$.
\item $P$ coincides with $F$ in a neighborhood of the critical point set of $\A_H$, and $\mathrm{sing}\, (P)=\mathrm{crit}\, (\A_H)$. 
In particular, the negative eigenspace of the Jacobian of $P$ at each rest point $(\q,\p)\in \M^{1-s}$ belongs to the essential class
$\mathcal E^s(\q,\p)$.  
\item The pair $(\A_H,P)$ satisfies the Palais-Smale condition: Every sequence $(\q_n,\p_n)\in \M^{1-s}$ such that $\A_H(\q_n,\p_n)\to a$, for some $a\in \R$, and $\diff \A_H(\q_n,\p_n)[P(\q_n,\p_n)] \to 0$ admits
a converging subsequence.
\item The essential vertical family $\mathcal F$ defined in \eqref{def:essentiallyvertical} is $P$-positively invariant, that is, for each $A\in \mathcal F$ the set $\Phi^P([-T,T]\times A)$ belongs to $\mathcal F$ for every $T\geq 0$, where $\Phi^P$ denotes the flow of $P$. In particular, the intersection between the stable and unstable manifolds 
$$W^u((\q_0,\p_0);P) \cap W^s((\q_1,\p_1);P)$$
of any two rest points $(\q_0,\p_0), (\q_1,\p_1)$ of $P$ (i.e. critical points of $\A_H$) is pre-compact.
\item The stable and unstable manifolds of any two rest points $(\q_0,\p_0), (\q_1,\p_1)$ of $P$ such that $m(\q_0,\p_0)-m(\q_1,\p_1)\leq 2$ intersect transversally. In particular, if non-empty,
$$W^u ((\q_0,\p_0);P) \cap W^s( (\q_1,\p_1);P)$$
is a submanifold of $\M^{1-s}$ of dimension $m(\q_0,\p_0)-m(\q_1,\p_1)\leq 2$.
\end{enumerate}
\end{thm}

\begin{proof}
We fix two open neighborhoods $\U\subset \V \subset \M^{1-s}$ of the set of critical points of $\A_H$ such that every critical point of $\A_H$ belongs to a different 
connected component of $\V$ and consider the space $\mathfrak X$ of all vector fields $X$ on $\M^{1-s}$ of class $C^2$ such that:
\begin{enumerate}
\item[i)] $X$ vanishes identically on $\U$,
\item[ii)] $X$ preserves the essential vertical family $\mathcal F$. 
\end{enumerate}
On $\mathfrak X$ we introduce a \textit{Whitney norm} $\|\cdot\|_{\mathfrak X}$ as follows: we pick a smooth function $\chi:[0,+\infty)\to \R$ such that 
$$0< \chi (\sigma ) < \frac 12 \inf_{\M^{1-s}_\sigma \setminus \U} - \diff \A_H [F], \quad \forall \sigma \geq 0,$$
where $\M^{1-s}_{\sigma}:= \{(\q,\p)\in \M^{1-s}\ |\ \|\p\|_{1-s}\leq \sigma \}$, and define for every $X\in \mathfrak X$
$$\|X\|_{\mathfrak X} := \|\chi^{-1} \cdot X\|_{C^2},$$
where with slight abuse of notation we set 
$$(\chi^{-1}\cdot X)(\q,\p) := \frac{1}{\chi (\|\p\|_{1-s})} \cdot X(\q,\p),\quad \forall (\q,\p)\in \M^{1-s}.$$
Clearly, the topology induced by the norm $\|\cdot\|_{\mathfrak X}$ coincides with the topology given by 
the $C^2_{\text{loc}}$-convergence. We shall notice that the infimum in the definition of $\chi$ is strictly positive for every fixed $\sigma\geq 0$ since the pair $(\A_H,F)$ satisfies the Palais-Smale condition
but might go to zero for $\sigma\to +\infty$. Properties (B1), (B3), and (B4) follow immediately from the definition of $\mathfrak X$ and Proposition~\ref{prop:module}. Moreover, Property (B5) holds even in 
a stronger form, as the set $\{X(p)\, |\, X\in \mathfrak X\}$ coincides with $T_p\M^{1-s}$ for every $p\in \M^{1-s}\setminus \V$. To see this, consider a local chart $(\varphi_*^{-1}, \text{Im},\ \varphi_*)$ of $\M^{1-s}$ 
such that $p\in \text{Im}\, \varphi_*$, and denote with slight abuse of notation the image of $p$ under $\varphi_*^{-1}$ again with $p$. Clearly, for every 
$v\in T_p(H^s(\T,\R^n)\times H^{1-s}(\T,(\R^n)^*))\cong H^s(\T,\R^n)\times H^{1-s}(\T,(\R^n)^*)$, the constant vector field $F_v\equiv v$
preserves essentially vertical sets (as its flow is just a translation). Therefore, by multiplying $F_v$ by a smooth cut-off function supported in a small ball around $p$ we see that there is a vector field $\tilde F_v$ on $\M^{1-s}$ which preserves essentially vertical sets and such that $F_v(p)=v$. 

We now denote by $\mathfrak X_1$ the open unit ball in $\mathfrak X$ with respect to the Whitney norm $\|\cdot\|_{\mathfrak X}$. By the very definition of $\|\cdot\|_{\mathfrak X}$ and the 
triangle inequality we see that, for every $X\in\mathfrak X_1$, 
$$\inf_{\M^{1-s}_\sigma \setminus \U} - \diff \A_H [F + X] \geq \frac 12 \inf_{\M^{1-s}_\sigma \setminus \U} - \diff \A_H[F] >0,\quad \forall \sigma\geq 0,$$
so that the singular set of $P:=F+ X$ coincides with the singular set of $F$ for every $X\in\mathfrak X_1$. Furthermore,
for every $(\q,\p)\in \M^{1-s}\setminus \U$ we have 
\begin{align*}
\diff \A_H ((\q,\p)) [P(\q,\p)] & \leq \frac 12  \diff \A_H(\q,\p)[F(\q,\p)].
					\end{align*}
This implies at once that $\A_H$ is a Lyapounov function for $P$ for every $X\in\mathfrak X_1$, and that every Palais-Smale sequence for the pair
$(\A_H,P)$ must eventually enter $\U$, where we have $P\equiv F$. Thus, the pair $(\A_H,P)$ satisfies the Palais-Smale condition. In other words, Property (B2) also holds for $\mathfrak X$. 

Finally, by construction the essentially vertical family $\mathcal F$ is positively $P$-invariant for every $X\in \mathfrak X$. All claims follow now with 
\begin{equation*}
\mathcal P := \Big \{P:= F + X \ \Big |\ X \in \mathfrak X_{1} \Big \}. \qedhere 
\label{perturbations}
\end{equation*}
% and hence 
%in particular that the intersection 
%$$W^u((\q_0,\p_0);P)\cap W^s ((\q_1,\p_1);P)$$ 
%is pre-compact for every $P\in \mathcal P$ and every pair of critical points of $\A_H$. This proves Claim (4). Finally, from Proposition \ref{prop:essentialinvariance} we deduce that Property (B6) 
%is satisfied as well for every $P\in \mathcal P$, and hence Claim (5) follows from the general transversality theorem, c.f. Theorem \ref{thm:morsesmalegeneral}.  
\end{proof}

%%%%%%%%%%%%%%%%
%%%%%%%%%%%%%%%%
%%%%%%%%%%%%%%%%

\section{The Morse complex}
\label{s:morsecomplex}

In this section, building on the results of Sections \ref{s:essentialsubbundle}, \ref{s:precompactness}, and \ref{s:transversality}, we construct the Morse complex of the triple $(\A_H,\M^{1-s},\mathcal E^s)$,
where $s\in (1/2,3/4)$ is arbitrary and $\mathcal E^s$ is the strongly integrable (0)-essential subbundle defined in Section \ref{s:essentialsubbundle}. To ease the notation, we hereafter write $\mathcal E$ instead of $\mathcal E^s$. 
The contents of this section are an easy adaptation of the general theory developed in \cite[Sections 8-11]{AM:05}; for this reason, we only provide a sketch of the construction referring to \cite{AM:05} for the proofs. 

Before proceeding further, we shall make some comments:
\begin{itemize}
\item The assumption on $s$ guarantees that the Hamiltonian action $\A_H$ is at least of class $C^3$. As we saw in Section \ref{s:transversality}, this allows us to use the Sard-Smale theorem 
to achieve the Morse-Smale property up to order 2, which is precisely what one needs to define the Morse complex.  
\item Changing the (0)-essential subbundle $\mathcal E$ by a compact perturbation changes the Morse-complex by a shift of the indices (when $\M^{1-s}$ is connected).
\item For a (0)-essential subbundle which does not come from a true subbundle, coherent orientations for the intersection between stable and unstable manifolds cannot be defined in general, and hence
one obtains just a Morse complex with $\Z_2$-coefficients. However, the (0)-essential subbundle $\mathcal E$ can be given an orientation (in a suitable sense) which allows to use $\Z$-coefficients 
instead of $\Z_2$-coefficients.  To keep the exposition as elementary as possible, we rather not pursue this direction in this paper and work with $\Z_2$ coefficients instead. 
\end{itemize}

%%%%%%%%%%%%%%%%

\subsection{Broken flow-lines}
Let $(\q,\p), (\tilde \q,\tilde\p)\in \M^{1-s}$ be two critical points of $\A_H$. As shown in \ref{thm:transverseintersection}, (4), for every perturbation $P\in\mathcal P$, where $\mathcal P$ is defined as in \eqref{perturbations}, the intersection 
$$W^u((\q,\p);P)\cap W^s ((\tilde \q,\tilde \p);P)$$ 
is pre-compact. Consider a sequence of flow lines from $(\q,\p)$ to $(\tilde \q,\tilde \p)$, and the sequence of their closure 
$$S_n := \overline{\Phi^P (\R\times \{(\q_n,\p_n)\})} \cup \{ (\q,\p),(\tilde \q,\tilde \p)\}, \quad (\q_n,\p_n) \in W^u((\q,\p);P)\cap W^s ((\tilde \q,\tilde \p);P).$$
By pre-compactness we can assume that $(\q_n,\p_n) \to (\q_\infty,\p_\infty)$, and the continuity of $\Phi^P$ would yield 
$$\Phi^P(\cdot, (\q_n,\p_n)) \to \Phi^P ( \cdot, (\q_\infty,\p_\infty))$$
uniformly on compact subsets of $\R$. However, it may happen that $(\q_\infty,\p_\infty)\not \in W^u((\q,\p);P)$ or $(\q_\infty,\p_\infty) \not \in W^s((\tilde \q,\tilde \p);P)$, 
so $\Phi^P(\cdot,(\q_\infty,\p_\infty))$ could be a flow line connecting two other rest points, and the convergence would not be uniform on $\R$. In this case, $S_n$ converges up to a subsequence
to a \textit{broken flow line} from $(\q,\p)$ to $(\tilde \q,\tilde \p)$ in the Hausdorff distance, which we recall is the union 
$$\mathbb S = \mathbb S_1 \cup ... \cup \mathbb S_k,$$ 
where each $\mathbb S_i$ is the closure of a flow line from $\xi_{i-1}$ to $\xi_i$, with 
$$(\q,\p) = \xi_0 \neq \xi_1 \neq ... \neq \xi_{k-1} \neq \xi_k = (\tilde \q,\tilde \p) \in \text{crit}\, (\A_H).$$
It is straightforward to check that a compact set $\mathbb S\subset \M^{1-s}$ is a broken from line from $(\q,\p)$ to $(\tilde \q,\tilde \p)$  if and only if the following three conditions are satisfied:
\begin{enumerate}
\item[i)] $(\q,\p), (\tilde \q,\tilde \p)\in \mathbb S$.
\item[ii)] $\mathbb S$ is $\Phi^P$-invariant.
\item[iii)] the intersection $\mathbb S \cap \A_H^{-1}(c)$
is a singleton if $c \in [\A_H(\q,\p),\A_H(\tilde \q,\tilde \p)]$ and is empty otherwise. 
\end{enumerate}

\begin{prop}
\label{prop:brokenflowlines}
Let $(\q_n,\p_n) \subset W^u((\q,\p);P)\cap W^s ((\tilde \q,\tilde \p);P)$, and set 
$$S_n:= \overline{\Phi^P (\R\times \{(\q_n,\p_n)\})} \cup \{ (\q,\p),(\tilde \q,\tilde \p)\}.$$
 Then, $(S_n)$ has 
a subsequence which converges to a broken flow line from $(\q,\p)$ to $(\tilde \q,\tilde \p)$ in the Hausdorff distance. \qed
\end{prop}

It is worth noticing that the proof of the proposition above, which is given in \cite[Section 8]{AM:05}, uses only the pre-compactness of the intersection 
$W^u((\q,\p);P)\cap W^s ((\tilde \q,\tilde \p);P)$, and hence in particular it is independent of the other properties of the triple $(\A_H,\M^{1-s},\mathcal E)$.

%%%%%%%%%%%%%%%%

\subsection{Intersections of dimension 1 and 2}
We fix $P\in \mathcal P$, where $\mathcal P$ is given by \eqref{perturbations}, and consider the quadruple $(\A_H,\M^{1-s},P,\mathcal E)$. The Morse-Smale condition 
up to order zero (c.f. Theorem \ref{thm:transverseintersection}, (5)) implies that for a broken flow line from $(\q,\p)$ to $(\tilde \q,\tilde \p)$ we have 
$$\text{m}(\q,\p) > \text{m}(\xi_1) > ... > \text{m}(\xi_{k-1})> \text{m}(\tilde \q,\tilde \p),$$ 
where m$(\cdot) =$m$(\cdot,\mathcal E)$ denotes the relative Morse index with respect to the (0)-essential sub-bundle $\mathcal E$. 

We assume henceforth that $(\q,\p)$ and $(\tilde \q,\tilde \p)$ are critical points of $\A_H$ with 
$$\text{m}(\q,\p) - \text{m}(\tilde \q,\tilde \p) \in \{1,2\}.$$
In this case, Theorem \ref{thm:transverseintersection} implies that 
$$W^u((\q,\p);P)\cap W^s ((\tilde \q,\tilde \p);P)$$
is a submanifold of $\M^{1-s}$ of dimension 1 resp. 2 with compact closure. 

In the former case, we readily see that the intersection consists of finitely many connected components. Indeed 
each connected component is a flow line from  $(\q,\p)$ to $(\tilde \q,\tilde \p)$, and the set $C$ of their closures is discrete in the Hausdorff distance. On the other hand, by the Morse-Smale property 
at order zero, these are the only 
broken flow lines from $(\q,\p)$ to $(\tilde \q,\tilde \p)$, and hence $C$ is compact by Proposition \ref{prop:brokenflowlines}. Notice that the restriction of the flow $\Phi^P$ to the closure 
of a component of the intersection is conjugated to the shift flow on $\overline \R$:
$$\R\times \overline \R \ni (t,u) \mapsto u+t \in \overline \R.$$

In the latter case, the quotient of a connected component $W$ of the intersection by the $\R$-action, $W/\R$, is either a circle or an open interval. In other words, $W$ is described by a 
one-parameter family of flow lines $u_\lambda$, where $\lambda$ ranges in $S^1$ or in $(0,1)$. In the first case, one easily verifies that $\overline W = W\cup \{(\q,\p),(\tilde \q,\tilde \p)\}$ is 
homeomorphic to a 2-sphere, and that $\Phi^P|_{\overline W}$ is conjugated to the exponential flow on the Riemann sphere. In the second case, we see using Proposition \ref{prop:brokenflowlines}
 that $\overline W\setminus W$ consists of broken flow lines which have precisely one intermediate rest point. Then, $\Phi^P|_{\overline W}$ is semi-conjugated to the product of two shift-flows on $\overline \R$.
 
 \begin{prop}[Semi-conjugacy]
 \label{prop:productshift}
 Let $(\q,\p)$ and $(\tilde \q,\tilde \p)$ be two critical points of $\A_H$ with 
$$\text{m}(\q,\p) - \text{m}(\tilde \q,\tilde \p) =2,$$
and let $W$ be a connected component of $W^u((\q,\p);P)\cap W^s ((\tilde \q,\tilde \p);P)$ such that $W/\R$ is an open interval. Then, there exists a 
continuous surjective map 
$$h:\overline \R\times \overline \R\to \overline W$$ 
such that the following hold:
\begin{enumerate}
\item $\Phi^P_t (h(u,v)) = h(u+t,v+t)$ for every $(u,v)\in \overline \R\times \overline \R$ and every $t\in \R$. 
\item $h(\R^2)=W$, and there exists critical points $\xi,\xi'$ of $\A_H$ with 
$$\text{m}(\xi)=\text{m}(\xi') = \text{m}(\q,\p)-1,$$ 
and connected components $W_1,W_2,W_1',W_2'$ of $W^u((\q,\p);P)\cap W^s(\xi;P)$, $W^u(\xi;P)\cap W^s ((\tilde \q,\tilde \p);P)$, and
$W^u((\q,\p);P)\cap W^s(\xi';P)$, $W^u(\xi';P)\cap W^s ((\tilde \q,\tilde \p);P)$ respectively, such that 
$$W_1\cup W_2 \neq W_1'\cup W_2',$$ 
and 
\begin{align*}
h (\R\times \{-\infty\}) &= W_1, \quad h(\{+\infty\} \times \R) = W_2, \\
h(\{-\infty\} \times \R) &= W_1', \quad h (\R\times \{+\infty\}) = W_2'.
\end{align*}
\item The restrictions of $h$ to $\R^2$, $\{\pm \infty\}\times \R$, and $\R\times \{\pm\infty\}$, are diffeomorphisms of class $C^1$. \qed
\end{enumerate}
 \end{prop} 

The proof of Proposition \ref{prop:productshift}, as well as of Proposition \ref{prop:productshift2} below, can be found in \cite[Section 11]{AM:05}. Here we just notice that, in Claim (2) it may 
happen that $\xi=\xi'$, and in this case even that $W_1=W_1'$ or $W_2=W_2'$,  but the last two identities cannot hold simultaneously. When $\xi\neq \xi'$, $h$ is injective, so it is a conjugacy.

\begin{prop}
\label{prop:productshift2}
Let $(\q,\p)$, $(\bar \q,\bar \p)$, and $(\tilde \q,\tilde \p)$ be critical points of $\A_H$ with 
$$\text{m}(\q,\p)=  \text{m}(\bar \q,\bar \p) +1 = \text{m}(\tilde \q,\tilde \p) +2,$$
and let $W_1,W_2$ be connected components of $W^u((\q,\p);P)\cap W^s ((\bar \q,\bar \p);P)$ and $W^u((\bar \q,\bar \p);P)\cap W^s ((\tilde \q,\tilde \p);P)$ respectively. Then, there exists 
a unique connected component $W$ of $W^u((\q,\p);P)\cap W^s ((\tilde \q,\tilde \p);P)$ such that $\overline{W_1\cup W_2}$ belongs to the closure of $W$ with respect to the Hausdorff distance. \qed
\end{prop}
%%%%%%%%%%%%%%%%

\subsection{The boundary homomorphism}
As in the previous subsections, we consider a tuple $(\A_H,\M^{1-s},P,\mathcal E)$ for some $P\in \mathcal P$. 
For every $k\in \Z$ we set $C_k(P)$ to be the $\Z_2$-vector space generated by the critical points of $\A_H$ with relative Morse index $k$. We show now how to define on $C_k(P)$ 
a boundary operator $\partial_k$. Notice preliminarly that, by Proposition \ref{lem:compactness}, the action of critical points of $\A_H$ is uniformly bounded from below, and 
hence in particular every interval of the form $(-\infty,a]$ contains only finitely many critical points of $\A_H$. 

For any pair of critical points $(\q,\p)$, $(\tilde \q,\tilde \p)$ of $\A_H$ such that 
$$\text{m}(\q,\p) =  \text{m}(\tilde \q,\tilde \p)+1= k$$
we define 
$$\sigma \big ((\q,\p), (\tilde \q,\tilde \p) \big ) := \# \Big \{ \text{c.c. of } W^u((\q,\p);P)\cap W^s ((\tilde \q,\tilde \p);P)\Big \} \ \ \text{modulo} \ 2$$
and
$$\partial_k (\q,\p):= \sum_{k-1} \ \sigma \big ((\q,\p), (\tilde \q,\tilde \p) \big )\cdot (\tilde \q,\tilde \p),$$
where with slight abuse of notation $\sum_{k-1}$ denotes the sum over all critical points of $\A_H$ with relative Morse index $k-1$. By the observation above, the sum is indeed a finite sum. 

\begin{thm}
For every $k\in \Z$ we have $\partial_{k-1}\circ \partial_k =0$. 
\label{thm:boundaryoperator}
\end{thm}
\begin{proof}
We compute using linearity 
\begin{align*}
\partial_{k-1}\circ \partial_k (\q,\p) &= \partial_{k-1} \left ( \sum_{k-1}  \sigma \big ((\q,\p), (\tilde \q,\tilde \p) \big )\cdot (\tilde \q,\tilde \p)\right )\\
						&= \sum_{k-1} \sigma \big ((\q,\p), (\tilde \q,\tilde \p) \big )\cdot \partial_{k-1} (\tilde \q,\tilde \p)\\
						&=  \sum_{k-1} \sigma \big ((\q,\p), (\tilde \q,\tilde \p) \big )\cdot \sum_{k-2} \sigma \big ((\tilde \q,\tilde \p), (\bar \q,\bar \p)\big )\cdot (\bar \q,\bar \p)\\
						&= \sum_{k-2}  \Big ( \sum_{k-1} \sigma \big ((\q,\p), (\tilde \q,\tilde \p) \big )\cdot  \sigma \big ((\tilde \q,\tilde \p), (\bar \q,\bar \p)\big ) \Big )\cdot (\bar \q,\bar \p)\\
						&=0,
\end{align*}
as Proposition \ref{prop:productshift} implies that broken flow lines between two critical points whose relative Morse indices differ by two always come in pairs. 
\end{proof}

We call the pair $(C_*(P),\partial_*)$ the \textit{Morse complex} of the Hamiltonian action $\A_H$ with respect to the (0)-essential sub-bundle $\mathcal E$, and denote by $HM_*(T^*M,\mathcal E)$
the induced homology, which we hereafter refer to as the \textit{Hamiltonian Morse homology} of the cotangent bundle $T^*M$. As the notation suggests, the
Hamiltonian Morse homology of the cotangent bundle is independent of the choices of the perturbation $P\in \mathcal P$ and of the Hamiltonian function, as soon as the growth condition~\eqref{eq:growthcondition} is satisfied. In fact,
$HM_*(T^*M,\mathcal E)$ is isomorphic to the singular homology of the free loop space of $M$. These facts will be proved in the next section.

%%%%%%%%%%%%%%%%
%%%%%%%%%%%%%%%%
%%%%%%%%%%%%%%%%

\section{Functoriality}

Goal of this section will be to prove the following 

\begin{thm}
Let $H_0,H_1:\T\times T^*M\to \R$, be two smooth Hamiltonians satisfying the growth condition~\eqref{eq:growthcondition} and such that every critical point of $\A_{H_0}$ and $\A_{H_1}$ is hyperbolic, and let $P_0,P_1$ be corresponding vector fields 
as in Theorem~\ref{thm:transverseintersection}. Then the Morse complexes $(C_*(P_0),\partial_*)$ and $(C_*(P_1),\partial_*)$ are isomorphic. In particular, the induced homology is independent 
both of the Hamiltonian and the vector field, and it can be therefore denoted by $HM_*(T^*M,\mathcal E)$.
Furthermore, $HM_*(T^*M,\mathcal E)$ is isomorphic to the singular homology of $\Lambda M$, the free loop space of $M$, thus also to the Floer homology of $T^*M$. 
\label{thm:functoriality}
\end{thm}

The theorem above is a consequence of the following more general functorial property of the Morse homology. 

\begin{thm}
The following statements hold:
\begin{enumerate}
\item Let $H$ be an Hamiltonian as in Theorem~\ref{thm:functoriality}, and let $P,\tilde P$ be two vector fields as in Theorem~\ref{thm:transverseintersection}. 
Then the Morse complexes $(C_*(P),\partial_*)$ and $(C_*(\tilde P),\partial_*)$ are isomorphic. In particular, the induced homology does not depend on the chosen vector field
and can be denoted by $HM_*(T^*M,\mathcal E, \A_{H})$. 
\item Suppose that $H_0$ and $H_1$ are Hamiltonians as in Theorem~\ref{thm:functoriality} such that $H_0\leq H_1$. Then, there is a sequence of homomorphisms of Abelian groups 
$$\phi_{H_0,H_1}: HM_k(T^*M,\mathcal E, \A_{H_0})\to HM_k (T^*M,\mathcal E, \A_{H_1}),\qquad \forall k\in \Z,$$
such that $\phi_{H_1,H_2}\circ \phi_{H_0,H_1}=\phi_{H_0,H_2}$ and $\phi_{H_0,H_0}=\mathrm{id}$ (actually, $\phi_{H_0,H_0+c}=\mathrm{id}$ for any $c\geq 0$).
\end{enumerate}
\label{thm:functoriality2}
\end{thm}

\begin{proof}[Proof of Theorem~\ref{thm:functoriality}]
By assumption we can find constants $c_-,c_+\in \R$ such that 
$$H_0+c_- \leq H_1\leq H_0+c_+.$$
By the functorial property of Morse homology we see that there are homomorphisms of Abelian groups $\phi_{H_0+c_-,H_1}$ and $\phi_{H_1,H_0+c_+}$ such that 
$$\phi_{H_1,H_0+c_+}\circ \phi_{H_0+c_-,H_1}  = \phi_{H_0+c_-,H_0+c_+} = \text{id}.$$
In particular, $\phi_{H_0+c_-,H_1}$ is injective and $\phi_{H_1,H_0+c_+}$ is surjective, which means that the Morse homology $HM_*(T^*M,\mathcal E, \A_{H_1},P_1)$ is at least as rich as 
the Morse homology $HM_*(T^*M,\mathcal E, \A_{H_0},P_0)$. Exchanging the role of $H_0$ and $H_1$ yields the claim. 

The isomorphism between $HM_*(T^*M,\mathcal E)$ and the singular homology of $\Lambda M$ is constructed exactly as in \cite{Abbondandolo:2006jf} for the case of Floer homology. Taking advantage 
of the freedom in the choice of the Hamiltonian function, we choose $H$ of the form\footnote{As already observed, for a generic choice of the potential $U$, all critical points of $\A_H$ are hyperbolic.} 
$$H(t,q,p) = \frac 12 |p|^2 + U(t,q).$$
Such an Hamiltonian is the Fenchel dual of the Tonelli Lagrangian 
$$L:\T\times TM\to \R, \quad L(t,q,p)=\frac12 |v|^2 - U(t,q).$$
The Legendre transform $(t,q,p)\mapsto (t,q,\partial H/\partial p(t,q,p))$ establishes a one-to-one correspondence between the set of one-periodic solutions to Hamilton's equations on $T^*M$, denoted $\mathcal P(H)$, and the set of one-periodic solutions  of the Euler-Lagrange equation on $TM$, denoted $\mathcal P(L)$, which in local coordinates read 
$$\frac{\diff}{\diff t} \frac{\partial L}{\partial v}(t,q(t),\dot q(t)) = \frac{\partial L}{\partial q}(t,q(t),\dot q(t)).$$
In the latter formulation, one-periodic solutions correspond to critical points of the Lagrangian action
$$\mathcal S: H^1(\T,M)\to \R,\qquad \mathcal S(q) := \int_0^1 L(t,q(t),\dot q(t))\, \diff t.$$
Moreover, the Morse index of every one-periodic solution $q\in \mathcal P(L)$ coincides with the relative Morse index of the corresponding solution $x\in \mathcal P(H)$.
As it turns out, one can apply infinite dimensional Morse theory (in fact, the Morse complex approach) as developed by Palais \cite{Palais:1963jp} to the functional $\mathcal S$, and the
resulting Morse homology is isomorphic to the singular homology of $H^1(\T,M)$ (and hence, to the singular homology of $\Lambda M$, as the latter two spaces are homotopy equivalent). The claim follows now 
showing that there is a chain complex isomorphism between $\{CM_*(\mathcal S),\partial_*\}$ and the Morse complex $\{C_*(\A_H),\partial_*\}$ of 
the quadruple $(\A_H,\M^{1-s},P,\mathcal E)$.
This can be done adapting the argument of  \cite{Abbondandolo:2006jf} to the setting of this paper.
\end{proof}

Thus, we are left to prove Theorem~\ref{thm:functoriality2}. Even if we are dealing here with strongly indefinite functionals, the argument is very similar to the case of finite Morse indices (see for instance \cite{Abbondandolo:2006lk} and references therein). %For this reason, we provide here only a sketch of the proof, explaining where modifications 
%of the original argument are needed to deal with the case of infinite Morse indices. 

Thus, let $H:\T\times T^*M\to \R$ be a smooth Hamiltonian as in Theorem~\ref{thm:functoriality} and let $P,\tilde P$ be two vector fields as in Theorem~\ref{thm:transverseintersection}. 
We introduce the smooth Morse function
$$\varphi:\R\to \R, \qquad \varphi(r) =2r^3-3r^2+1,$$
which has a non-degenerate local maximum at $r=0$ with $\varphi(0)=1$ and a non-degenerate local minimum at $r=1$ with $\varphi(1)=0$.  Moreover, $\varphi'(r)$ diverges for $r\to \pm \infty$. On $\R\times \M^{1-s}$ we consider the function 
$$f: \R\times \M^{1-s}\to \R,\qquad f(r,x) := \varphi(r) + \A_H(x)$$
and the \textit{cone vector field}
$$F(r,\cdot) := \chi(r) P(\cdot) + (1-\chi(r))\tilde P(\cdot) - \varphi'(r) \frac{\partial}{\partial r},$$
where $\chi:\R\to \R$ is the characteristic function of the half-line $(-\infty,\frac 12]$. It is straightforward to check that $f$ is a non-degenerate Lyapounov function for $F$.

\begin{rmk}
The cone vector field $F$ has discontinuities at the hypersurface $\{r=1/2\}$. However, these still allow to have a well-defined Morse complex. Indeed, both vector fields 
$$\mathfrak P (r,\cdot) := P(\cdot) - \varphi'(r) \frac{\partial}{\partial r} \qquad \text{and} \qquad \mathfrak P(r,\cdot):= \tilde P(\cdot) - \varphi'(r) \frac{\partial}{\partial r}$$
are transverse to the hypersurface $\{r=1/2\}$ and point in the same direction. For $z=(r,x)\in  (-\infty,1/2)\times \M^{1-s}$ we denote by $\tau(z)>0$ the hitting time of the $\mathfrak P$-flowline 
through $z$ with the hypersurface $\{z=1/2\}$ and readily see that 
$$\Phi_F^t (z) = \left \{\begin{array}{l} \Phi_{\mathfrak P}^t (z) \qquad \qquad \qquad \ \ \, t\leq \tau(z), \\ \Phi_{\tilde{\mathfrak P}}^{t-\tau(z)} \circ \phi_{\mathfrak P}^{\tau(z)}(z) \qquad t >\tau(z).\end{array}\right .$$
From this one can easily deduce that the stable and unstable manifolds of critical points of $f$ are indeed smooth manifolds. The details are left to the reader. 
\qed
\end{rmk}

Clearly, the critical points of $f$ (equivalently, the rest points of $F$) are the points $(0,x)$ with $x\in \text{crit}\,  \A_{H}$ and $(1,y)$ with $y\in \text{crit} \, \A_{H}$.
In particular, by Proposition~\ref{lem:compactness} this implies that the action of critical points of $f$ is uniformly bounded from below. 
Also, the pair $(f,F)$ satisfies the Palais-Smale condition; indeed, if $(r_n,x_n)\subset \R\times \M^{1-s}$ is a Palais-Smale sequence for the pair $(f,F)$ then, up to extracting a subsequence, 
we either have that $r_n\to 0$ or $r_n\to 1$. Let us assume that $r_n\to 0$ (the other case being completely analogous). Without loss of generality we can suppose that $|r_n|<\frac 12$ for all $n\in \N$, so that for all $n\in\N$ we have
$$f(r_n,x_n) = \A_{H}(x_n) + \varphi(r_n),\qquad F(r_n,x_n)= P(x_n)-  \varphi'(r_n)\frac{\partial}{\partial r}.$$
This shows that $(x_n)$ is a Palais-Smale sequence for the pair $(\A_{H},P)$, thus pre-compact as the pair $(\A_{H},P)$ satisfies the Palais-Smale condition. 

If $\mathcal E\subset T\M^{1-s}$ denotes the (0)-essential subbundle defined in Section~\ref{s:essentialsubbundle},
we readily see that $\{0\}\times \mathcal E\subset T (\R\times \M^{1-s})$ is a strongly integrable (0)-essential subbundle with the property that the Hessian of $f$ at any critical point is a 
compact perturbation of $\{0\}\times \mathcal E$, and that the relative Morse indices of critical points of $f$ with respect to the (0)-essential subbundle $\{0\}\times \mathcal E$ are 
\begin{align*}
\mathrm{m}((0,x);\{0\}\times \mathcal E) &= \mathrm{m}(x;\mathcal E) + 1, \qquad \forall x \in \text{crit}\, \A_{H},\\
\mathrm{m}((1,y);\{0\}\times \mathcal E) &= \mathrm{m}(y;\mathcal E), \qquad \forall y \in \text{crit}\, \A_{H},
\end{align*}
where $\mathrm{m}(\cdot,\mathcal E)$ is the relative Morse index defined in Definition~\ref{def:morseindex}.
Finally, if $\mathcal F$ denotes the essentially vertical family defined in \eqref{def:essentiallyvertical}, we have that 
$$\bigcup_{J\subset \R} J\times \mathcal F$$
is an essentially vertical family for $\{0\}\times \mathcal E$ which is $F$-invariant (recall indeed that $\mathcal F$ is both $P$- and $\tilde P$-invariant). Therefore, Theorem 6.5 in \cite{AM:05} implies that the intersection between stable and unstable manifolds of any two critical points of $f$ is pre-compact. 

Let us now have a closer look at the stable and unstable manifolds of critical points of $f$. Denoting by $\sigma:\R\to (0,1)$ the solution of the Cauchy problem 
$$\left \{\begin{array}{l} \sigma'(t) = -\varphi'(\sigma(t)),\\ \sigma(0)=\frac 12,\end{array}\right.$$
a direct inspection to the flow of $F$ shows that for every $x,y\in \text{crit}\, \A_{H}$ we have 
\begin{align*}
W^s((0,x);F) &= \{0\} \times W^s(x;P),\\
W^u((0,x);F) &= \bigcup_{-\infty<s<1} \{r\} \times \Phi^{\tilde P}_{t_+(r)} \big (W^u(x;P)\big ),\\
W^s((1,y);F) &= \bigcup_{0<s<+\infty} \{r\} \times \Phi^{P}_{t_-(r)} \big (W^s(y;\tilde P)\big ),\\
W^u((1,y);F) &= \{1\}\times W^u(y;\tilde P),
\end{align*} 
where the functions $t_+:(-\infty,1)\to \R^+$ and $t_-:(0,+\infty)\to \R^-$ are defined implicitly by 
$$\left \{ \begin{array}{r} t_+(r) =0 \qquad \text{for}\ r\leq 1/2,\\ \sigma(t_+(r))=r \ \ \ \text{for}\ r \geq 1/2, \end{array}\right. \qquad  \qquad \left \{ \begin{array}{r} t_-(r) =0 \qquad \text{for}\ r\geq 1/2,\\ \sigma(t_-(r))=r \ \ \ \text{for}\ r \leq 1/2. \end{array}\right. $$
In particular, for every $x,x',y,y'\in \text{crit}\, \A_{H}$ we have
\begin{align*}
W^u((0,x);F) \cap W^s((0,x');F) &= \{0\} \times \big (W^u(x;P)\cap W^s(x'; P)\big ),\\
W^u((1,y);F) \cap W^s((1,y');F) &= \{1\} \times \big (W^u(y;\tilde P)\cap W^s(y';\tilde P)\big ),\\
W^u((0,x);F) \cap W^s((1,y);F) &= \bigcup_{0<r\leq 1/2} \{r\}\times \big (W^u(x;P) \cap \Phi^{P}_{t_-(r)}(W^s(y;\tilde P))\big ) \\
& \cup \bigcup_{1/2\leq r < 1} \{r\}\times \big ( \Phi^{\tilde P}_{t_+(r)}(W^u(x;P))\cap W^s(y;\tilde P)\big ) \\
W^u((1,y);F) \cap W^s((0,x);F) &= \emptyset.
\end{align*}

From this, we readily see that all but possibly the third intersection are transverse by assumption. Notice also that in case $x=y$ we have 
$$W^u((0,x);F) \cap W^s((1,x);F) = (0,1)\times \{x\}.$$ 
This can be seen also as following: We interpret $\A_H$ as a function of both $r$ and $x$ and observe that
$$\diff \A_H(r,x) [F(r,x)] = \chi(r)\cdot \diff \A_H(r,x)[P(x)] + (1-\chi(r))\cdot \diff \A_H(r,x)[\tilde P(x)] <0$$
for all $(r,x)\in \R\times \M^{1-s}$ such that $x\not \in \text{crit}\, \A_H$. Therefore, $\A_H$ is strictly decreasing along all non-constant flow lines of $F$, besides those which are up to time shifts of the form
$$t\mapsto (\sigma(t),x),\qquad x\in \text{crit}\, \A_H.$$
In particular, up to time shifts there is exactly one flow line going from $(0,x)$ to $(1,x)$ for every $x\in \text{crit}\, \A_H$, and the intersection
$$W^u((0,x);F) \cap W^s((1,x);F) = (0,1)\times \{x\}$$ 
is transverse. This can be seen roughly speaking as follows: at the hypersurface $\{r=1/2\}$ we have
$$\{1/2\}\times \big (W^u(x;P)\cap W^s(x;\tilde P)\big )$$
and the claim follows since $T_xW^u(x;P)$ belongs to the negative cone to $f$ at $x$ whereas $T_xW^s(x;\tilde P)$ belongs to the positive cone to $f$ at $x$. The details are left to the reader. 

Repeating the argument of Section~\ref{s:morsecomplex} word by word, we deduce that the Morse complex with $\Z_2$-coefficients with respect to the (0)-essential subbundle $\{0\}\times \mathcal E$ for the cone vector field $F$ is well-defined provided the vector fields $P$ and $\tilde P$ satisfy the following

\vspace{2mm}

\textbf{Transversality for pairs condition:} For every $x, y\in \text{crit}\, \A_{H}$ with $\mathrm m(x;\mathcal E)-\mathrm{m}(y;\mathcal E)\leq 1$ the intersection 
$$W^u(x;P) \cap W^s(y;\tilde P)$$

 is transverse.

\vspace{2mm}

%{\color{red} (Is this because $W^u(x;P_0) \cap \Phi^{P_0}_{t_-(s)}(W^s(y;P_1))$ and  $\Phi^{P_1}_{t_+(s)}(W^u(x;P_0))\cap W^s(y;P_1)$ are transverse if and only if $W^u(x;P_0) \cap W^s(y;P_1)$ is transverse?)}

\noindent Indeed, by the computations above, we see that stable and unstable manifolds of critical points of $f$ intersect transversally if and only if the transversality for pairs condition holds. 
Such a condition can be achieved by a generic small perturbation of $\tilde P$, as one sees adapting the arguments in Section~\ref{s:transversality}. 
It is worth observing that according to  \cite[Section 8]{AM:01} it is possible to define a Morse complex even if the transversality for pairs condition is not satisfied. However, in this case the boundary operator 
is not uniquely determined. 

 Recalling that we are using $\Z_2$-coefficients, we see that the boundary operator $\partial^F_*$ takes the form 
$$\partial^F_{*+1}:C_{*+1}(F)\cong C_*(P)\oplus C_{*+1}(\tilde P)\to  C_{*-1}(P)\oplus C_*(\tilde P)\cong C_*(F),\qquad \partial^F_{*+1} = \left (\begin{matrix} \partial^{P}_* & 0 \\ \Psi_* & \partial^{\tilde P}_{*+1}\end{matrix}\right ),$$
for some chain homomorphism $\Psi_*:C_*(P)\to C_*(\tilde P)$ (this follows from the fact that $0 = \partial_{*}^F\circ \partial_{*+1}^F$).% readily implies that 
%$\Phi_{*-1}\circ \partial^{P}_* = \partial_*^{\tilde P}\circ \Phi_*$). 
%\Phi_*$ is a chain homomorphism follows e seen by computing
%\begin{align*}
%0 = \partial^F_{k}\circ \partial^F_{k+1} &= \left (\begin{matrix} \partial^{P}_{k-1} & 0 \\ \Phi_{k-1} & \partial^{\tilde P}_{k}\end{matrix}\right )\circ \left (\begin{matrix} \partial^{P}_k & 0 \\ \Phi_k & \partial^{\tilde P}_{k+1}\end{matrix}\right )\\
% &= \left (\begin{matrix} \partial^{P}_{k-1}\circ \partial^{P}_{k} & 0 \\ \Phi_{k-1}\circ \partial^{P}_k + \partial_k^{\tilde P}\circ \Phi_k & \partial^{\tilde P}_{k}\circ \partial^{\tilde P}_{k+1}\end{matrix}\right ).
% \end{align*}
 %Thus, $\Phi_*$ descends to the desired homomorphism
 %$$\Phi_{H_0,H_1}: HM_*(T^*M,\mathcal E, \A_{H_0},P_0) \to HM_*(T^*M,\mathcal E, \A_{H_1},P_1).$$
 
We now explicitly determine $\Psi_*$. Using the computations above about stable and unstable manifolds of critical points 
 of $f$ it is easy to check that the pair $(P,\tilde P)$ satisfies the following two additional properties: 
 
 \vspace{2mm}

\textbf{Compactness for pairs condition:} For every $x,y\in \text{crit}\, \A_{H}$ with $\mathrm m(x;\mathcal E)-\mathrm{m}(y;\mathcal E)\leq 1$
the sets 
$$W^u(x;P) \cap \Phi^{P}(\R^-\times W^s(y;\tilde P)),\qquad \Phi^{\tilde P}(\R^+\times W^u(x;P))\cap W^s(y;\tilde P)$$ 

are pre-compact. 

\vspace{2mm}

\textbf{Finiteness for pairs condition:} For every $x\in \text{crit}\, \A_{H}$ the set 
$$\big \{y\in \text{crit}\, \A_{H}\ \big |\ \mathrm{m}(y;\mathcal E) = \mathrm{m}(x;\mathcal E), \ \text{and} \ W^u(x;P)\cap W^s(y;\tilde P)\neq \emptyset\big \}$$

is finite. 

\vspace{2mm}

\noindent These facts, together with the transversality for pairs condition, allow us to construct $\Phi_*$ as follows: Define $\text{crit}_{\kappa} \A_H$ to be the set of critical points of $\A_H$ with relative Morse 
index $\kappa$. For any $x,y\in \text{crit}_\kappa \A_{H}$ consider $p\in W^u(x;P)\cap W^s(y;\tilde P)$. The tangent space of $W^u(x;P)$ at any $p$ is a compact 
perturbation of $\mathcal E(p)$ with 
$$\dim (T_pW^u(x;P),\mathcal E(p))= \mathrm{m}(x;\mathcal E) = \kappa,$$
while the pair $(T_pW^s(y;\tilde P),\mathcal E(p))$ is Fredholm with 
$$\text{ind}\, (T_pW^s(y;\tilde P),\mathcal E(p))= - \mathrm{m}(y;\mathcal E)=-\kappa.$$
It follows that $(T_pW^u(x;P), T_pW^s(y;\tilde P))$ is a Fredholm pair of index 0 (see e.g. Proposition A.2 in \cite{AM:05}). By the transversality and compactness for pairs conditions, 
the intersection $W^u(x;P)\cap W^s(y;\tilde P)$ is a compact discrete set, thus finite. We can therefore define for every $x\in \text{crit}_\kappa \A_H$ 
(observe that this is indeed a finite sum by the finiteness for pairs condition)
\begin{align*}
\Psi_\kappa x &:= \sum_{y \in \text{crit}_\kappa \A_{H}} \Big (\# W^u(x;P)\cap W^s(y;\tilde P) \ \text{modulo}\ 2\Big ) \cdot y\\
		&= \ \ \ x \ \ \ + \sum_{\scriptsize \begin{matrix} y \in \text{crit}_\kappa \A_{H},\\ \ \A_H(y)<\A_H(x)\end{matrix}} \Big (\# W^u(x;P)\cap W^s(y;\tilde P) \ \text{modulo}\ 2\Big ) \cdot y.
\end{align*}
It is straightforward to check that such a homomorphism coincides with the one appearing in the expression for $\partial_*^F$
(in other words, the Morse complex $(C_*(F),\partial^F_*)$ is the \textit{mapping cone} of the homomorphism $\Psi$; see e.g. \cite[II.4]{Maclane:1967}). 
Moreover, ordering the critical points of $\A_H$ by increasing value of the action, we see that $\Psi_\kappa$ is represented by an upper-triangular matrix with $1$ on the diagonal entries. 
Part (1) of Theorem~\ref{thm:functoriality2} now readily follows, as an homomorphism of this form must be an isomorphism. Indeed, if $x_1,x_2,...$ are the critical points of index $\kappa$ of $\A_H$ ordered by increasing 
value of $\A_H$, the inverse of $\Phi_\kappa$ is defined inductively by 
\begin{align*}
\Psi_\kappa^{-1} x_1& =x_1, \\
\Psi_\kappa^{-1}x_\ell &= x_\ell - \sum_{j=1}^{\ell -1} \Big (\# W^u(x_\ell;P)\cap W^s(x_j;\tilde P) \ \text{modulo}\ 2\Big ) \cdot \Phi_\kappa^{-1}x_j,\qquad \forall \ell\geq 2.
\end{align*}

The homomorphism $\Phi_{H_0,H_1}$ in Part (2) for Hamiltonians $H_0\leq H_1$ with corresponding negative pseudo gradient vector fields $P_0$ and $P_1$
is constructed using the same ideas used for Part (1): One defines the \textit{cone function}  
$f:\R\times \M^{1-s}\to \R$ of the Hamiltonian actions 
$\A_{H_0}\geq \A_{H_1}$ by
$$f(r,\cdot):= \chi(r) \A_{H_0}(\cdot) + (1-\chi(r)) \A_{H_1}(\cdot) + \varphi(r),$$
which is a non-degenerate Lyapounov function of the cone vector field 
$$F(r,\cdot) := \chi(r) P_0(\cdot) + (1-\chi(r) P_1(\cdot) - \varphi'(r) \frac{\partial}{\partial r}.$$
Repeating the argument above, we see that the Morse complex with $\Z_2$-coefficients of the cone vector field $F$ with respect to the (0)-essential subbundle $\{0\}\times \mathcal E$ is well-defined, 
and that the boundary operator $\partial^F_*$ is the cone of some chain homomorphism $\Psi_{H_0,H_1}$, which thus induces a homomorphism $\Phi_{H_0,H_1}$ in homology.  
The fact that $\Phi_{H_0,H_0+c}=\text{id}$ for all $c\geq 0$ follows by taking $P_0=P_1$. 
 
 Thus, we are left to show the transitivity property
 $$\Phi_{H_1,H_2}\circ \Phi_{H_0,H_1} = \Phi_{H_0,H_2}$$
 for any triplet of Hamiltonians $(H_0,H_1,H_2)$ such that $H_0\leq H_1\leq H_2$. This will be achieved by iterating the cone vector field construction. Thus, choose negative pseudo gradient vector fields 
 $P_i$ for $\A_{H_i}$, $i=0,1,2$, as in the statement of Theorem~\ref{thm:transverseintersection} and denote with 
 $F_{ij}$, $0\leq 1\leq j\leq 2$, the cone vector field of $P_i$ and $P_j$:
 $$F_{ij} (r,\cdot) = \chi(r) P_i(\cdot) + (1-\chi(r))P_j(\cdot) + \varphi(r).$$
 Up to a generic small perturbation of $P_1$ and $P_2$ we can suppose that the Morse complex $(C_*(F_{ij}), \partial_*^{F_{ij}})$ is well-defined. Equivalently, this means that each pair $(P_i,P_j)$ satisfies 
 the transversality, compactness, and finiteness for pairs conditions. All we have to show is that  (recall that we are using $\Z_2$-coefficients)
 $$\Psi_{H_1,H_2}\circ \Psi_{H_0,H_1}+ \Psi_{H_0,H_2} = \partial^{P_2}\circ \mathcal P + \mathcal P\circ \partial^{P_0}$$
 for some homomorphism $\mathcal P_*:C_*(P_0)\to C_{*+1}(P_2)$ which is usually called the \textit{prisma operator}.
 
 \begin{rmk}
For arbitrary vector fields $F_0,F_1,F_2$ such that each pair $(F_i,F_j)$ satisfies the transversality, compactness, and finiteness for pairs conditions, the chain map $\Psi_{F_0,F_2}:C_*(F_0)\to C_*(F_2)$ 
need not be chain homotopic to the composition $\Psi_{F_1,F_2}\circ \Psi_{F_0,F_1}$. For instance, take $F_0=F_2$ and $F_1$ to be a vector field with no rest points\footnote{Every infinite dimensional Hilbert manifold 
admits a vector field with no rest points since its tangent bundle is trivial.}. In this case we have $\Psi_{F_0,F_2}=\text{id}$, whereas $\Psi_{F_1,F_2}= \Psi_{F_0,F_1}=0$, so that $\Psi_{F_0,F_2}$ is not chain 
homotopic to the composition $\Psi_{F_1,F_2}\circ \Psi_{F_0,F_1}$ unless the Morse complex of $F_0=F_2$ is contractible. 
 \qed
\end{rmk}

On $\R^2\times \M^{1-s}$ consider the cone vector field of $F_{01}$ and $F_{02}$ 
$$F_{01,02}(r,l,\cdot) = \chi(r) F_{01}(l,\cdot) + (1-\chi(r))F_{02}(l,\cdot) - \varphi'(r)\frac{\partial}{\partial r}.$$
It is immediate to check that the cone function 
\begin{align*}
f_{01,02}(r,l,\cdot) 
			&= \chi(s) \A_{H_0}(\cdot) + \chi(r)(1-\chi(l))\A_{H_1} (\cdot)+(1-\chi(r))(1-\chi(l))\A_{H_2}(\cdot) +  \varphi(l) +  \varphi(r)
			\end{align*}
is a non-degenerate Lyapounov function for $F_{01,02}$ and that its critical points (equivalently, the
rest points of $F_{01,02}$) are of the form 
\begin{align*}
 (0,0,x),\ (1,0,x), \qquad & x \in \text{crit}\, \A_{H_0},\\
 (0,1,y), \qquad & y\in \text{crit}\, \A_{H_1},\\
 (1,1,z), \qquad & z\in \text{crit}\, \A_{H_2}.
\end{align*}
In particular, the action of critical points of $f_{01,02}$ is uniformly bounded from below. Also, as one sees adapting the argument used for Part (1), 
the pair $(f_{01,02},F_{01,02})$ satisfies the Palais-Smale condition, the Hessian of $f_{01,02}$ at every critical point is 
a compact perturbation of the strongly integrable (0)-essential subbundle $\{(0,0)\}\times \mathcal E\subset T(\R^2\times \M^{1-s})$ and the relative Morse indices read
\begin{align*}
\mathrm{m}((0,0,x);\{(0,0)\}\times \mathcal E) &= \mathrm{m}(x;\mathcal E) + 2, \qquad \forall x \in \text{crit}\, \A_{H_0},\\
\mathrm{m}((1,0,x);\{(0,0)\}\times \mathcal E) &= \mathrm{m}(x;\mathcal E)+1, \qquad \forall x \in \text{crit}\, \A_{H_0},\\
\mathrm{m}((0,1,y);\{(0,0)\}\times \mathcal E) &= \mathrm{m}(x;\mathcal E)+1, \qquad \forall y \in \text{crit}\, \A_{H_1},\\
\mathrm{m}((1,1,z);\{(0,0)\}\times \mathcal E) &= \mathrm{m}(z;\mathcal E), \qquad \forall z \in \text{crit}\, \A_{H_2}.
\end{align*}
Finally, 
$$\bigcup_{A\subset \R^2} A\times \mathcal F$$
is an essentially vertical family for $\{(0,0)\}\times \mathcal E$ which is $F_{01,02}$-invariant. This implies in virtue of Theorem 6.5 in \cite{AM:05} that the intersection between stable and unstable manifolds 
of any two critical point of $f_{01,02}$ is pre-compact. Transversality can also be achieved by a small perturbation, so that the Morse complex 
$(C_*(F_{01,02}),\partial_*^{F_{01,02}})$ with $\Z_2$-coefficients with respect to the (0)-essential subbundle $\{(0,0)\}\times \mathcal E$ is well-defined. 

Writing 
$$C_{*+1}(F_{01,02}) \cong C_{*-1}(P_0) \oplus C_{*}(P_0)\oplus C_{*}(P_1) \oplus C_{*+1}(P_2)$$
it is easy to check that $\partial_*^{F_{01,02}}$ takes the form 
$$\partial_{*+1}^{F_{01,02}} = \left (\begin{matrix} \partial_{*-1}^{P_0} & 0 & 0 & 0 \\ * & \partial_{*}^{P_0} & * & 0 \\ (\Psi_{H_0,H_1})_{*-1} & 0 & \partial_{*}^{P_1} & 0 \\ * & (\Psi_{H_0,H_2})_{*} & * & \partial_{*+1}^{P_2} \end{matrix}\right ).$$
 To determine the missing terms we observe that 
\begin{align*}
F_{01,02}(r,l,\cdot) &= \chi(l) P_0(\cdot) + \chi(r)(1-\chi(l))P_1(\cdot)+(1-\chi(r))(1-\chi(l))P_2(\cdot) -  \varphi'(l)\frac{\partial}{\partial l}-  \varphi'(r)\frac{\partial}{\partial r}\\
				&= \chi(l) \Big (P_0(\cdot) - \varphi'(r)\frac{\partial}{\partial r}\Big ) + (1-\chi(l)) \Big (\chi(r)P_1(\cdot)+(1-\chi(r))P_2(\cdot)- \varphi'(r)\frac{\partial}{\partial r}\Big ) -  \varphi'(l)\frac{\partial}{\partial l}\\
				&= \chi(l) F_{00}(r,\cdot) + (1-\chi(l))F_{12}(r,\cdot) -  \varphi'(l)\frac{\partial}{\partial l}\\
				&=F_{00,12}(l,r,\cdot)
\end{align*}
and similarly $f_{01,02}(r,l,\cdot)=f_{00,12}(l,r,\cdot)$. In particular, the boundary operator $\partial_*^{F_{01,02}}$ can be rewritten as 
$$\partial_{*+1}^{F_{01,02}} = \partial_{*+1}^{F_{00,12}} = \left (\begin{matrix} \partial_{*-1}^{P_0} & 0 & 0 & 0 \\ \text{id}_{C_{*-1}(P_0)} & \partial_{*}^{P_0} & 0 & 0 \\ * & * & \partial_{*}^{P_1} & 0 \\ * & * & (\Psi_{H_1,H_2})_{*} & \partial_{*+1}^{P_2} \end{matrix}\right ).$$
Comparing the two expressions for $\partial_*^{F_{01,02}}$ yields 
$$\partial_{*+1}^{F_{01,02}} = \left (\begin{matrix} \partial_{*-1}^{P_0} & 0 & 0 & 0 \\  \text{id}_{C_{*-1}(P_0)} & \partial_{*}^{P_0} & 0 & 0 \\ (\Psi_{H_0,H_1})_{*-1} & 0 & \partial_{*}^{P_1} & 0 \\ \mathcal P_{*-1} & (\Psi_{H_0,H_2})_* & (\Psi_{H_1,H_2})_* & \partial_{*+1}^{P_2} \end{matrix}\right )$$
for some homomorphism $\mathcal P_{*-1}:C_{*-1}(P_0)\to C_*(P_2)$. The identity $0= \partial_{*+1}^{F_{01,02}}\circ \partial_{*+2}^{F_{01,02}}$ implies now that 
$$(\Psi_{H_1,H_2})_*\circ (\Psi_{H_0,H_1})_*+ (\Psi_{H_0,H_2})_* = \partial^{P_2}_{*+1} \circ \mathcal P_* + \mathcal P_{*-1} \circ \partial^{P_0}_*$$
as claimed.

%%%%%%%%%%%%%%%%
%%%%%%%%%%%%%%%%
%%%%%%%%%%%%%%%%

\appendix

\section{Flows which preserve an essentially vertical family}
\label{s:essentialinvariance}

In this appendix we prove some extra properties about flows which preserve an essentially vertical family. Thus, let $\M$ be a Hilbert manifold modeled 
on the real separable Hilbert space $\HH$, and let $\mathcal E$ be a strongly integrable essential sub-bundle of $T\M$ with strong integrable structure given by 
an atlas $\mathcal A$ which carries an essentially vertical family $\mathcal F$. For a $C^1$-vector field $F$ on $M$ we denote by 
$$\phi^F : \Omega (F) \subset \R\times \M\to \M$$
its local flow. Our first claim is that, if $F$ is complete and $\mathcal F$ is $F$-positively invariant, meaning that for every $A\in \mathcal F$ the set 
$\phi^F([-T,T]\times A) \in \mathcal F$ for every $T\geq 0$, then $\mathcal E$ is \textit{invariant} with respect to $F$. This means that, denoting with $\mathcal E(p)$ 
a local representative of $\mathcal E$ at $p$ for every $p\in \M$, the following holds: 
\begin{equation}
D\phi_t^F (p) \mathcal E(p) \ \text{is a compact perturbation of } \mathcal E(\phi_t^F(p)), \quad \forall t\in \R, \ \forall p\in\M.
\label{eq:invariance1}
\end{equation}
We shall notice that this notion does not depend on the choice of the local representative at $p$. 
Indeed, by \cite[Page 341]{AM:05}, $\mathcal E$ is invariant with respect to $F$ at $p\in \M$ if and only if 
\begin{equation}
(L_F \mathcal P)(p) \mathcal P(p)\in \mathcal L(T_p\M) \ \text{is a compact endomorphism},
\label{eq:invariance2}
\end{equation}
where $L_F$ denotes the Lie-derivative along $F$ and $\mathcal P$ denotes a projector onto a local representative of $\mathcal E$ in a neighborhood $\U$ of $p$ (more precisely, $\mathcal P$ is a section of the 
Banach bundle of linear endomorphisms of $T\U$ such that for every $p\in \U$, $\mathcal P(p)$ is a projector onto $\mathcal E(p)$). It is now straightforward to check 
that \eqref{eq:invariance2} is independent of the choice of the projector as well as of the choice of the local representative of $\mathcal E$.

\begin{prop}
\label{prop:essentialinvariance}
Let $F$ be a complete vector field on $\M$ for which the essentially vertical family $\mathcal F$ is positively invariant. Then, the strongly integrable essential sub-bundle $\mathcal E$ 
is positively invariant under $F$. 
\end{prop}
\begin{proof}
Being Condition \eqref{eq:invariance2} local, we can work in a neighborhood $\mathcal U$ of a fixed point $p\in \M$. Using a chart $\varphi$ in $\mathcal A$ 
we can further identify such a neighborhood with an open subset $U$ of $\HH$ in such a way that $p$ correspond to $0$, and assume that $\mathcal E$ is represented by the
constant sub-bundle corresponding to a closed linear subspace $V\subset \HH$. Let $P:\HH\to \HH$ 
be a projector onto $V$. By the very definition of $\mathcal F$, we see that essentially vertical sets which are contained in $\frac 12 \mathcal U:= \varphi^{-1}(\frac 12 U)$ correspond to subsets $A\subset \frac 12 U$ such that $(I-P)A$ is pre-compact. Clearly, the set $V\cap \frac 12 U$ belongs to $\mathcal F$. Then, the fact that 
$\mathcal F$ is $F$-positively invariant implies that, for $T>0$ small enough (namely, such that $\Phi^F([-T,T] \times  \frac 12 U)\subset U$), the map
$$\frac 12 U \to \HH,\quad x\mapsto (I-P) \Phi^F_t (P x),$$
has pre-compact image for every $t\in [-T,T]$. Differentiating at $x=0$ yields that the linear operator 
$$(I-P)D\phi^F_t(0) P$$
is compact for every $t\in [-T,T]$. Therefore, the operator
$$(L_FP)(0) P = [DF(0),P]P =(I-P)DF(0)P = (I-P) \frac{\diff}{\diff t} \Big |_{t=0} D\phi^F_t(0) P = \frac{\diff}{\diff t}\Big |_{t=0} (I-P)D\phi^F_t(0)P$$
is compact.
\end{proof}

\bibliography{_biblio}
\bibliographystyle{plain}

\end{document}